\documentclass[reqno]{amsart}
\usepackage[utf8]{inputenc}
\usepackage{enumitem}
\usepackage{amsmath,amssymb,amsbsy,amsfonts,amsthm,latexsym,amsopn,amstext,mathtools,leftidx,
            amsxtra,euscript,amscd,verbatim,epsf,colordvi,graphics}
\usepackage{ytableau}
\usepackage{color}
\usepackage{pgf,tikz}
\usetikzlibrary{decorations.pathreplacing}
\usetikzlibrary{arrows}
\usepackage{mathrsfs}
\usepackage{graphicx}

\usepackage{hyperref}
\usepackage{cleveref}
\usepackage[margin=11pt,font=footnotesize,labelfont=bf]{caption}
\newtheoremstyle{tplain}{3pt}{3pt}{\rmfamily}{}{\bfseries}{.}{0.5em}{}
\theoremstyle{tplain}

\usepackage{pgf,tikz}
\usepackage{tcolorbox}
\usepackage[enableskew]{youngtab}
\definecolor{darkgreen}{cmyk}{1,0,1,0}
\newtheorem{thm}{Theorem}
\newtheorem{lem}{Lemma}
\newtheorem{ex}{Example}
\newtheorem{cor}{Corollary}
\newtheorem{prop}{Proposition}
\newtheorem{obs}{Remark}
\newtheorem{defi}{Definition}
\def \cal{\mathcal}
\def \rm {\mathrm}

\def\Z { \mathbb{Z}}
\def\N { \mathbb{N}}
\def\C { \mathbb{C}}
\def\Z { \mathbb{Z}}
\def\C { \mathbb{C}}
\def\R { \mathbb{R}}
\def\LRAII {\mathsf{LR}}
\def \r {\mathbf{r}}
\def \t {\mathbf{t}}
\def \redu {\mathrm{red}}
\def\c { \mathrm{c}}
\def \k {\mathfrak{k}}
\def\wt { \mathrm{wt}}

\def \SST {SST_{2n}}
\def \SpT {SpT_{2n}}
\def \pr {\mathbf{pr}}
\def \s {\mathbf{s}}
\def \i {\overline i}
\def \g {\mathfrak{g}}
\def\domres {\mathrm{domres}}
\def\Hive {\mathrm{Hive}}
\def \res {\mathrm{res}}
\def \v {\mathrm{v}}

\newcommand{\bro}{\color{brown}}

\newcommand{\ora}{\color{orange}}

\newcommand{\oo}{\color{blue}}
\newcommand{\blue}{\color{blue}}
\newcommand{\red}{\color{red}}
\newcommand{\green}{\color{green}}

\makeatletter
\newcommand*\bigcdot{\mathpalette\bigcdot@{.5}}
\newcommand*\bigcdot@[2]{\mathbin{\vcenter{\hbox{\scalebox{#2}{$\m@th#1\bullet$}}}}}
\makeatother

\newcommand*\circled[1]{\tikz[baseline=(char.base)]{
            \node[shape=circle,draw,inner sep=1pt] (char) {#1};}}

\newcommand{\YT}[3]{
\vcenter{\hbox{
\begin{tikzpicture}[x={(0in,-#1)},y={(#1,0in)}] 
\foreach \rowi [count=\i] in {#3} {
 \foreach \e [count=\j] in \rowi {
  \draw (\i,\j) rectangle +(-1,-1);
  \draw (\i-0.5,\j-0.5) node {$#2\e$};
 }
}
\end{tikzpicture}
}}
}

\title[ Quantum Littlewood-Richardson rule and $\widehat{\mathfrak{g}}$-dominant tableaux] {Characterization of $\k$-highest weight and $\widehat{\mathfrak{g}}$-dominant tableaux  \\via $1$-$0$-slack recording tableaux in the quantum Littlewood-Richardson rule\\ and  lattice points in  flagged hive polytopes}

\author{Olga Azenhas}
\address{ University of Coimbra, CMUC, Department of Mathematics, Portugal}
\email{oazenhas@mat.uc.pt}



\begin{document}

\begin{abstract}We have  previously, for a given positive integer $n$, explicitly characterized by certain linear inequalities the ${\mathfrak{k}}$-highest weight tableaux of  shape length $2n$ or $2n-1$   produced by $1$-$0$-slack recording tableaux in the quantum Littlewood-Richardson (LR) rule. We now extend  that characterization for lower shape lengths. Then using the composition of promotion operators  defining the Naito-Suzuki-Watanabe bijection between ${\mathfrak{k}}$-highest weight tableaux and $\widehat{\mathfrak{g}}$-dominant tableaux, we also explicitly characterize by certain linear inequalities the corresponding  $\widehat{\mathfrak{g}}$-dominant tableaux when the shape length is $2n$ or $2n-1$. In addition, when the given  $n$ is $3$ or $ 4$ the ${\mathfrak{k}}$-highest weight and the
$\widehat{\mathfrak{g}}$-dominant tableaux via $1$-$0$-slack recording tableaux are  characterized by linear inequalities  for lower shape lengths.
 Since recording tableaux in the quantum Littlewood-Richardson rule are in natural bijection with Littlewood-Richardson-Sundaram (LRS) tableaux, we relate our results on the inverse  quantum LR rule with other two bijections for the Naito-Sagaki conjecture and establish  bijections  between ${\mathfrak{k}}$-highest weight tableaux, $\widehat{\mathfrak{g}}$-dominant tableaux and the lattice points
in  a   (disjoint) union of  flagged hive polytopes.

\end{abstract}
\maketitle

\tableofcontents

\section{Introduction} The Naito–Sagaki conjecture \cite{naitosagaki} states that the branching rule for the restriction of finite-dimensional, irreducible polynomial representations of $GL_{2n}(\mathbb{C})$ to $Sp_{2n}(\mathbb{C})$  enumerates certain rational paths in the Littelmann's path model \cite{littelmann} satisfying certain dominance conditions. After the study of some cases \cite{naitosagaki, torres}, the conjecture was combinatorially solved by Schumann-Torres \cite{schumanntorres}.
The obtained branching rule establishes a bijection between $\widehat\g$-dominant semi-standard tableaux of shape $\lambda\in Par_{\le 2n}$, partitions with at most $2n$ parts, in the alphabet $\cal{A}_{2n}=\{1<\cdots<2n\}$  with $\widehat\g$-weight $\mu\subset \lambda$ a partition with at most $n$ parts, and the set $LRS_{2n}(\lambda, \mu)$  of
Littlewood–Richardson–Sundaram tableaux of skew shape $\lambda/\mu$. For $\lambda$ a partition  with at most $2n$ parts and $\mu\subset \lambda$ with at most $n$ parts,  those semi-standard dominant tableaux are denoted in \cite{schumanntorres} by $\domres(\lambda,\mu)$, and in \cite{nsw} by $\SST^{{\widehat{\mathfrak{g}}-dom}}(\lambda,\mu)$ and we inter-twin between the two notations.
Here $ {\mathfrak{g}}:={\mathfrak{gl}}_{2n}$ and $\widehat {\mathfrak{g}}=\widehat{\mathfrak{gl}}_{2n}$ is the fixed point  Lie subalgebra
induced  by the  folding of the Dynkin diagram of type $A_{2n-1}$, along the middle vertex, into the Dynkin diagram of type $C_n$, which is isomorphic to the symplectic Lie algebra $\mathfrak{sp}(2n,\C)$ (we refer to \cite{schumanntorres,nsw} for details). More precisely, given  partitions $\lambda$ with at most $2n$ parts and $\mu\subset \lambda$ with at most $n$ parts, Schumann-Torres \cite{schumanntorres}  build on the Sundaram branching model \cite{sundaram,sundaram90} and provide a bijection between the sets
\begin{align}\phi:\domres(\lambda,\mu)\,(\SST^{{\widehat{\mathfrak{g}}-dom}}(\lambda,\mu))\overset{\sim}\longrightarrow  LRS_{2n}(\lambda,\mu)=\bigsqcup_{\nu \mbox{ \scriptsize{even}}}LRS_{2n}(\lambda/\mu,\nu).\label{phi}
\end{align}

Recently Naito-Suzuki-Watanabe \cite{nsw} have given a new proof of the Naito-Sagaki conjecture using another branching model that comes from the quantum Littlewood-Richardson (LR) rule \cite{watanabe},
\begin{align} \label{lrsIIintr}
 \LRAII^{AII}: SST_{2n}(\lambda) \overset{\sim}\longrightarrow
\bigsqcup_{\begin{smallmatrix}\mu\in Par_{\le n}\\
\mu\subseteq\lambda\\
Q\in Rec_{2n}(\lambda/\mu)\end{smallmatrix}}SpT_{2n}(\mu) \times \{Q\},
\end{align}
where $SpT_{2n}(\mu)$ is the set of symplectic tableaux of shape $\mu$, and $LRS_{2n}(\lambda,\mu)\overset{\sim}{\underset{}\longrightarrow} Rec_{2n}(\lambda/\mu)$  are in natural bijection \cite{azreco,watanabe}.
In this framework,
a non-standard realization of the symplectic Lie algebra is  $\k$,  another  certain subalgebra of $\mathfrak{g}$, isomorphic to the symplectic Lie algebra $ \mathfrak{sp}(2n,C)$ which  arises in $\imath$quantum symmetric pairs of type $AII$ \cite{wang,watanabe,nsw}, \cite{csh24}. Based on that quantum branching model by Watanabe they provide two bijections \cite[Theorem 5.18.]{nsw}. One of them, that we shall use in our discussion, asserts, for all $\lambda\in Par_{\le 2n}$ there exists the bijection

\begin{align}\Phi:SST^{\widehat{\mathfrak{g}}-dom}_{2n}(\lambda)\overset{\sim}\longrightarrow SST^{\mathfrak{k}-hw}_{2n}(\lambda) 
\end{align}
 such that for each $T\in SST^{\widehat{\mathfrak{g}}-dom}_{2n}(\lambda)$,
$\wt_{\widehat{\mathfrak{g}}}(T)=\wt_{\widehat{\mathfrak{k}}}(\Phi(T)),
$ where $\Phi$ is a combinatorial bijection expressed as a composition of certain promotion operators \cite[Definition 7.19]{nsw}.
Here, for each $\lambda\in Par_{\le 2n}$,
\begin{align}SST^{\mathfrak{k}-hw}
_{2n}(\lambda):=\bigsqcup_{\begin{smallmatrix}\mu\in Par_{\le n}\\
\mu\subseteq\lambda\end{smallmatrix}} {\LRAII^{AII}}^{-1}(\{S^{H,\mu}\}\times Rec_{2n}(\lambda/\mu)),
\end{align}
is the set of all $\mathfrak{k}$-highest weight tableaux of shape $\lambda$ in $SST_{2n}$, where for each $\mu\subset \lambda$ with at most $n$ parts, $S^{H,\mu}$ is the symplectic tableau of shape and $\k$-weight $\mu$, that is, the symplectic $\k$-highest weight tableau of shape $\mu$ in $SpT_{2n}(\mu)$. In particular,
$$SST^{\mathfrak{k}-hw}
_{2n}(\lambda,\mu):={\LRAII^{AII}}^{-1}(\{S^{H,\mu}\}\times Rec_{2n}(\lambda/\mu)):=\{S\in    SST^{\mathfrak{k}-hw}_{2n}(\lambda)| \wt_{\mathfrak{k}}(S)=\mu\} \subseteq SST_{2n}(\lambda).$$

For $\lambda\in Par_{\le 2n}$, and $\mu\subset \lambda$ with $\mu\in Par_{\le n}$, we consider   for the Naito-Sagaki conjecture, the Schumann-Torres bijection $\phi$ (or $\phi^{-1}$) above \eqref{phi}  and the bijection
$$\daleth:=\Phi^{-1}\circ {\LRAII^{AII}}^{-1}\circ \pi^{-1}\circ \lozenge, \mbox{ such that }  \daleth LRS_{2n}(\lambda,\mu)=SST^{\widehat{\mathfrak{g}}-dom}_{2n}(\lambda,\mu)$$ as below

\begin{align}\label{dale}\daleth: 
LRS_{2n}(\lambda,\mu)\overset{\sim}{\underset{\lozenge}\longrightarrow }Rec_{2n}(\lambda/\mu)\overset{\sim}{\underset{\pi^{-1}}\longrightarrow }\{S^{H,\mu}\}\times Rec_{2n}(\lambda/\mu) \overset{\sim}{\underset{{\LRAII^{AII}}^{-1}}\longrightarrow }SST^{\mathfrak{k}-hw}_{2n}(\lambda,\mu)  \overset{\sim}{\underset{\Phi^{-1}}\longrightarrow}SST^{\widehat{\mathfrak{g}}-dom}_{2n}(\lambda,\mu)
\end{align}
 where $\pi$ is the natural projection, $(S^{H,\mu}, S)\rightarrow S$ for each $S\in Rec_{2n}(\lambda/\mu)$, and  $\lozenge$ is the bijection in \cite[Theorem 1]{azreco}, \cite[Lemma 8.3.2]{watanabe}.  We note that there is another recent bijection \cite{muniz} for the Naito-Sagaki conjecture which does not rely on those two others  that it is not considered here.

Take into account that the Berenstein-Gelfand-Zelevinsky LR model in \cite[Theorem 4.3]{BZphy} restricted to symplectic Gelfand-Tselin patterns is in bijection with  LR-Sundaram tableaux \cite{az26}, and that,  in turn, the set $LRS(\lambda/\mu, \nu)$, $\nu$ even, is in bijection with integer flagged hives whose boundary data is given by the same three partitions \cite{sathishtorres}.
Let $\Hive(\mu,\nu,\lambda,\varphi) $ denote the flagged hive polytope with boundary data $\mu,\nu,\lambda$  and  flag $\varphi$ \cite{krv21}.  Let
$\Hive_{\Z}(\mu,\nu,\lambda,\varphi) $ be its set of integer points, that is, the set of integer flagged hives with flag $\varphi$ whose boundary data is given by the same three partitions.
We  then have the  two chains of bijections
\begin{align}\label{bijectionsnsw}SST^{\widehat{\mathfrak{g}}-dom}_{2n}(\lambda,\mu)\overset{\sim}{\underset{\Phi}\rightarrow } SST^{\mathfrak{k}-hw}_{2n}(\lambda,\mu)  \overset{\sim}{\underset{ \lozenge^{-1}\circ \pi\circ \LRAII^{AII}} \rightarrow } LRS_{2n}(\lambda,\mu) \overset{\sim}\rightarrow
\bigsqcup_{\nu \mbox{ \scriptsize{even}}} Sp\mathcal{GT}^{rev(\lambda-\nu)}_{\mu,2n}\overset{\sim}\rightarrow \bigsqcup_{\nu \mbox{ \scriptsize{even}}} \Hive_{\Z}(\mu,\nu,\lambda,\varphi)
\end{align}

and

\begin{align}\label{bijectionschuto} SST^{\mathfrak{k}-hw}_{2n}(\lambda,\mu)\overset{\sim}{\underset{\Phi^{-1}}\rightarrow} SST^{\widehat{\mathfrak{g}}-dom}_{2n}(\lambda,\mu)\underset{\phi} \rightarrow  LRS_{2n}(\lambda,\mu) \overset{\sim}\rightarrow
\bigsqcup_{\nu \mbox{ \scriptsize{even}}} Sp\mathcal{GT}^{rev(\lambda-\nu)}_{\mu,2n}\overset{\sim}\rightarrow \bigsqcup_{\nu \mbox{ \scriptsize{even}}} \Hive_{\Z}(\mu,\nu,\lambda,\varphi)
\end{align}
where $Sp\mathcal{GT}^{rev(\lambda-\nu)}_{\mu,2n} $ is the set of  symplectic Gelfand-Tselin patterns of type $\mu$ and weight $rev(\lambda-\nu)$, the reverse of the vector $\lambda-\nu$ with $\nu$ an even partition, which are the left companions  of the  LR-Sundaram tableaux  in $LRS_{2n}(\lambda/\mu,\nu)$ \cite{az26} (we refer to \cite[Section 2.4]{akt16} for right and left Gelfand-Tsetlin companions of  an LR tableau); and $\Hive_{Z}(\mu,\nu,\lambda,\varphi)$ is the set of integral hives with boundary data $\mu,\nu,\lambda$  and  flag $\varphi$ in bijection with $LRS_{2n}(\lambda/\mu,\nu)$ studied by Sathish-Torres in \cite[Remark 4.8]{sathishtorres} in relation with  Kwon branching models \cite{kwon18}. Flagged hives were introduced by Kushwaha-Raghavan-Viswanath in \cite{krv21} as a generalization of Knutson-Tao hives \cite{knutson} (We also refer to \cite{alexanderssonpolytope} for integral points of a hive polytope whose boundary data is given by three partitions.)

 In conclusion, the  tableaux in $SST^{\widehat{\mathfrak{g}}-dom}_{2n}(\lambda,\mu)$ and the  tableaux in $SST^{\mathfrak{k}-hw}_{2n}(\lambda,\mu)$ in the quantum LR rule are each of which in bijection with the lattice points
in the (disjoint) union of the flagged hive polytopes $\Hive(\mu,\nu,\lambda,\varphi) $, where the
union is over all even partitions $\nu$. The lattice points of those flagged hive  polytopes are also in bijection with the integral triangles obtained by interlocking  the left and right companion Gelfand-Tstlin pattern  pair of each LRS tableau in the union of $LRS_{2n}(\lambda/\mu,\nu)$ where the union is over all   even partition $\nu$ \cite[Example 2.6]{akt16}.

The approach of  looking at the tableaux in $\domres(\lambda,\mu)$ (or $SST^{\widehat{\mathfrak{g}}-dom}_{2n}(\lambda,\mu)$) as lattice points of the LR cone under the light of the Pak-Vallejo work \cite{pakvallejocones} have  already been considered by Torres \cite{torres} and Schuman-Torres \cite[Theorem 57]{schumanntorres}  in the stable case, that is,  $\ell(\lambda),\ell(\mu)\le n $. In  the stable case, $LRS_{2n}(\lambda/\mu,\nu)=$ $LR_{2n}(\lambda/\mu,\nu)$ and $LR_{2n}(\lambda/\mu,\nu) \overset{\sim}\longleftrightarrow   \mathcal{H}(\mu,\nu,\lambda)$.

For a given positive integer $n$, we have previously explicitly characterized, by linear inequalities, the ${\mathfrak{k}}$-highest weight tableaux of shape length $2n$ or $2n-1$  produced by $1$-$0$-slack recording tableaux in the quantum Littlewood-Richardson (LR) rule \cite[Theorem 8, Theorem 9]{azreco} and their corollaries. Here, in Section \ref{subsec:less2n}, we extend that characterization for lower shape lengths in Theorem \ref{th:uk=n}, Theorem \ref{th:uk=n+1}, Theorem \ref{thm:nkeven} and  Theorem \ref{thm:nevenkodd}. We observe that it suffices to characterize the shape lengths $\ell(\lambda)=n+1$ when $n$ is odd, and $\ell(\lambda)=n+2$ when $n$ is even. See Remark \ref{enough} and Example \ref{ex:wordx}.

 In the cases $n=3$ and $n=4$, the characterization of ${\mathfrak{k}}$-highest weight tableaux of lower shape lengths is given in Section \ref{sec:illustration},
Theorem \ref{thm:n=3slack1} and Theorem \ref{thm:2}.

Given $l\in [0,2n]$ and $t\in[0,n]$ such that  $0\le t\le min\{l,2n-l\}$ and $l-t\in 2\Z$,
explicit procedures were given in \cite{azreduction}, as a  complement to \cite{watanabe}, to compute the inverse of the reduction map
\begin{align}\label{expintro}
\redu^{-1}_t:SpT_{2n}(\varpi_t)&\rightarrow SST_{2n}(\varpi_l).
\end{align}

We are here restricting our analysis to the subset of $1$-$0$ recording tableaux  ${Rec}_{2n}^{1,0}(\lambda/\mu)\subseteq Rec_{2n}(\lambda/\mu) $ \eqref{recording0-1}, and thus, in our case, $t\in \{0,1\}$, which means that

  \begin{align}l\notin 2\Z\Rightarrow t=1 \mbox{  and } l\in 2\Z\Rightarrow t=0,
  \end{align}
and the procedure to compute \eqref{expintro} is \cite[Theorem 2]{azreduction}.
As a consequence,  we have established Lemma \ref{lem:basic} which asserts that the  $\k$-highest weight tableaux of odd shape length   $l\le 2n-1$ produced  by $1$-$0$ slack recording tableaux  are obtained from those of shape length $l+1$ by deleting the columns of length $l+1$. The tableaux obtained in this way share the same linear inequalities. In Section \ref{sec:illustration}, this allows  to reduce the analysis for $n=3$ to the lengths $6$, $4$ and $2$, and for $n=4$ to the lengths $8$, $6$, $4$, and $2$. This is done in Theorem \ref{thm:n=3slack1} and Theorem \ref{thm:2}.

Using the promotion operators defining the bijection $\Phi$ between $SST^{\mathfrak{k}-hw}_{2n}(\lambda,\mu)$  and $SST^{\widehat{\mathfrak{g}}-dom}_{2n}(\lambda,\mu)$ in \cite{nsw}, we characterize in Section \ref{subse:dom2n}, Theorem \ref{thm:k}, Theorem \ref{thm:3} and Theorem \ref{thm:dominantneven}, the $\widehat{\mathfrak{g}}$-dominant tableaux, via $1$-$0$-slack recording tableaux, of shape length $2n$ or $2n-1$.
For $n=3,4$ in Section \ref{sec:illustration}, Theorem \ref{thm:1} respectively Theorem \ref{thm:2} this is done for lower shape lengths.

 We note that the cases $n=1,2$ to characterizing $SST^{\widehat{\mathfrak{g}}-dom}_{2n}(\lambda,\mu)$ have   been earlier considered by Torres \cite{torres} and Naito-Suzuki-Watanabe \cite{nsw}, and to characterizing $\k$-highest weight tableaux  have  also been considered by  the latter.
 We highlight Remarks \ref{ha} and \ref{haha} which convey much of the ideas behind the linear inequalities for $\widehat{\mathfrak{g}}$-dominant tableaux produced by $1$-$0$ slack recording tableaux.


Lastly, in the last Section \ref{sec:last}, we finish our manuscript by illustrating the bijections $\phi$ and $\daleth$ for $n=3$ (without restricting to $1$-$0$ recording tableaux). As for the   illustration we see that bijections $\phi$ \eqref{phi} and $\daleth$ \eqref{dale} coincide. In general, it rises the question on the coincidence of the bijections  $\phi$ \eqref{phi}, $\daleth$ \eqref{dale} and Muniz bijection \cite{muniz} for the Naito-Sagaki conjecture.

Our paper is organized in five sections.
\medskip

\section*{Acknowledgements}
The author acknowledges financial support by the Centre for Mathematics of the University of Coimbra (CMUC, https://doi.org/10.54499/UID/00324/2025) under the Portuguese Foundation for Science and Technology (FCT), Grants UID/00324/2025 and UID/PRR/00324/2025.

\section{Preliminaries}
This section follows the references \cite{torres, schumanntorres} and \cite{nsw} and we refer to them for more details.

Given $m\in \mathbb{N}$, the set of partitions with at most $m$ parts is denoted by $Par_{\le m}$. Given $\lambda\in Par_{\le m}$, the length of $\lambda$, $\ell(\lambda)$, is the number of  parts of $\lambda$ and $\ell(\lambda)\le m$.

Let $SSYT_m$ be the set of semi-standard Young tableaux of all shapes in $Par_{\le m}$ with entries in the
ordered alphabet $\cal{A}_m = \{a_1 < \cdots < a_m\}$. Let $T \in SSYT_m(\lambda)$, that is, $T \in SSYT_m$ of shape $\lambda\in Par_{\le m}$. The word $w(T)$ of the tableau
$T$ is obtained from it by reading its entries column-wise, top to bottom,  from right to left. The weight of $T$ (respectively the word  $w(T)$) is the vector $\wt(T)=(T[a_1],\dots,T[a_m])$ where $T[a_i]$ is the number of entries $a_i$ of $T$ (respectively  the word $w(T)$), for $i=1,\dots,m$.

\subsection{Dominant tableaux}

Let $SSYT_{C_n}$ be the set of all semi-standard Young tableaux of shapes the partitions in $Par_{\le 2n}$
with entries in the ordered alphabet
$C_n=\{1 < \cdots < n < \bar n <  \cdots< \bar 1\}$. Given $T\in SSYT_{C_n}$, the integral weight of $\mathfrak{sp}(2n,\C)$ is attached to  $T$ and to its word $w(T)$,  that is,  the $\mathfrak{sp}(2n,\C)$-weight of $T$, $\widehat{\wt}(T)$, is defined to be  $\widehat{\wt}(T)=(T[1]-T[\bar 1],\dots, T[n]-T[\bar n])\in \Z^n$. We say that $T\in SSYT_{C_n}$ is \emph{dominant} if the ($\mathfrak{sp}(2n,\C)$) weight of each prefix of its word $w(T)$ is a  partition \cite{schumanntorres,nsw}. In this case, $\wt(T)\in Par_{\le n}$.
For instance,

\begin{align}\label{dominantC}T=\YT{0.15in}{}{
{1,1,1,1,1,1,1,1,1,1,1},
{2,2,2,2,2,2,2,2,2},
{3,3,3,3,3,3,\overline 2,\overline 2},
{\overline 3,\overline 3,\overline 3,\overline 3,\overline 2,\overline 1,\overline 1},
{\overline 2,\overline 2,\overline 1,\overline 1,\overline 1},
{\overline 1}
}
\end{align}
is a dominant tableau either in $SST_{C_3}(11,8,7,6,4,1)$ with weight $\mu=(5,4,2)$ or in $SST_{C_6}(11,8,7,6,4,1,0^6)$ with weight $\mu=(5,4,2,0,0,0)$.

\begin{defi}\cite{schumanntorres} Let  $\lambda\in Par_{\le 2n}$ and $\mu\in Par_{\le n}$ such that $\mu\subseteq \lambda$.
The
set of tableaux in $SSYT_{C_n}$ of shape $\lambda$ and weight $\mu$ that have the dominance property is denoted by $\domres(\lambda,\mu)$.
\end{defi}

Let $\mathfrak{g}:={\mathfrak{gl}}_{2n}(\mathbb{C})$. The fixed point  Lie subalgebra  $\widehat {\mathfrak{g}}=\widehat{\mathfrak{gl}}_{2n}$, induced  by the  folding of the Dynkin diagram of type $A_{2n-1}$, along the middle vertex, into the Dynkin diagram of type $C_n$, is isomorphic to the symplectic Lie algebra $\mathfrak{sp}(2n,C)$ (we refer to \cite{nsw} for details).

Given a semi-standard
Young tableau $T\in \SST$ of shape $\lambda$
of at most $2n$ parts in the alphabet $\cal{A}_{2n}=\{1<2<\cdots<2n\}$, replace each letter $i>n$ by $\overline{2n-i+1}$
to produce a new tableau of the same shape $\lambda$, $\res(T)\in SSYT_{C_n}$ \cite{schumanntorres}. We write $\wt_{\widehat {\mathfrak{g}}}(T):=\widehat{\wt}(\res(T))$.

Read the word of this new tableau $res(T)$ column-wise from right to left and top
to bottom and recall that $\bar i=2n-i+1$. At each step $j$, let $\mu_i^j$ equal to the number of entries  $i$ minus the number of entries $\bar i$ in the word up to that point. Then $\res(T)$ belongs to the set $\domres(\lambda)$ if and only if $\mu_j=(\mu_1^j,\mu_2^j,\dots,\mu_n^j)$ is a partition for all steps $j$. In this case, if the last step returns the partition $\mu\in Par_{\le n}$, $\res(T)$ belongs to the set $\domres(\lambda,\mu)$.

\begin{defi} In \cite{nsw}, given  $\lambda\in Par_{\le 2n}$ and  $\mu\in Par_{\le n}$ such that $\mu\subseteq \lambda$, those tableaux $T$ in $\SST(\lambda)$ whose $\res(T)$ belong to $\domres(\lambda)$ are called $\widehat {\mathfrak{g}}$-\emph{dominant} or $\widehat{\mathfrak{gl}}_{2n}$-\emph{dominant}.
 The set of $\widehat {\mathfrak{g}}$-\emph{dominant} tableaux in  $\SST(\lambda)$ is denoted by
$\SST^{{\widehat{\mathfrak{g}}-dom}}(\lambda)$. The set of $\widehat {\mathfrak{g}}$-\emph{dominant} tableaux $T$ in  $\SST(\lambda)$  whose $\res(T)$ belong to $\domres(\lambda,\mu)$ is denoted by
$\SST^{{\widehat{\mathfrak{g}}-dom}}(\lambda,\mu)$.
\end{defi}

For example, below on LHS the tableau $S$ belongs to $\SST(11,8,7,6,4,1)$, with $n=3$,  and $\res(S)=T\in SST_{C_3}$ in \eqref{dominantC}; and on the $RHS$ the tableau $S'$ belongs to $\SST(11,8,7,6,4,1, 0^6)$ with $n=6$, and $\res(S')=T\in SST_{C_6}$ in \eqref{dominantC}, and their integral $\mathfrak{sp}(2n,C)$-weights are  respectively $\mu=(5,4,2)\in Par_{\le 3}$ and $\mu=(5,4,2,0,0,0)\in Par_{\le 6}$,
\begin{align}S=\YT{0.15in}{}{
{1,1,1,1,1,1,1,1,1,1,1},
{2,2,2,2,2,2,2,2,2},
{3,3,3,3,3,3,5,5},
{4,4,4,4,5,6,6},
{5,5,6,6,6},
{6}
}\in SST_6 \qquad
S'=\YT{0.15in}{}{
{1,1,1,1,1,1,1,1,1,1,1},
{2,2,2,2,2,2,2,2,2},
{3,3,3,3,3,3,11,11},
{10,10,10,10,11,12,12},
{11,11,12,12,12},
{12}
}\in SST_{12}\label{dominantA}
\end{align}

Then $\res(S)\in \domres(\lambda,\mu)$ with $\lambda=(11,8,7,6,4,1)$ and $\wt_{\widehat{\mathfrak{gl}}_{6}}(S)=\mu=(5,4,2)$;  and $\res(S')\in \domres(\lambda , \mu)$
with $\lambda=(11,8,7,6,4,1,0^6)$ and $\wt_{\widehat{\mathfrak{gl}}_{12}}(S')=\mu=(5,4,2,0^3)$. Equivalently $S$ is $\widehat{\mathfrak{gl}}_{6}$-dominant, and $S'$ is $\widehat{\mathfrak{gl}}_{12}$-dominant. In the later case, for $n=6$, $S'\in SST_{2n}$ is such that  $\ell(\lambda), \ell(\mu)\le n$.

\subsubsection{Inequalities for domres in the stable case} Let $\mu\subseteq \lambda$ partitions where $\ell(\mu)\le \ell(\lambda)\le n$. Then we say that $\lambda$ is \emph{stable}.
Since the  entries of a semi-standard Young tableau in $SST_{C_n}$ are weakly increasing along the rows, each semi-standard Young tableau $T $ in the stable case is
determined by the family of non negative numbers, parameterized by pairs in $(i,j)\in C_n\times \{1,\dots,n\}$,

$$\v(T):=\{(i,j):i\in C_n, j\in\{1,\dots,n\}\}$$

where for each $i\in C_n$,
$$(i, j):=\mbox{ number of entries $i\in C_n$ in row $j\in\{1,\dots,n\}$ of $T\in SST_{C_n}$}.$$

 In this case $\domres(\lambda,\mu)$ is the set of lattice points of a polytope \cite{schumanntorres}. This is done in the next lema by encoding each $ T \in \domres(\lambda,\mu)$ by the vector   of non-negative numbers $\v(T)$ as displayed below in \eqref{vector}.

\begin{lem}\cite[Lemma 49]{schumanntorres}\label{encode} Let $ T \in \domres(\lambda,\mu)$ for some $\mu\subseteq \lambda$ where $\ell(\mu)\le \ell(\lambda)\le n$ . Then,  the vector $\v(T)$  of non-negative numbers $(i,j)\in C_n\times \{1,\dots,n\}$  where for each $i\in C_n$,
$$(i, j):=\mbox{ number of entries $i$ in row $j\in\{1,\dots,n\}$ of $T\in SST_{C_n}(\lambda),$}$$
 is determined by the family of non-negative numbers, parameterized
by the pairs pictured in the array as below:
\begin{equation}\label{vector}\begin{array}{cccccccccccccc}
(1,1)&&&&&&&&\\
(2,2)\,&(\overline 1,2)&&&&&&&&\\
(3,3)\,&(\overline 2,3)&(\overline 1,3)&&&&&\\
\vdots&&&\vdots&&&&\\
(n,n)&(\overline {n-1},n)&\cdots&(\overline 3,n)&(\overline 2,n)&(\overline 1,n)&&
\end{array}
\end{equation}
\end{lem}
\begin{defi}\cite[Definition 50]{schumanntorres} \label{def:cancelation}A semi-standard Young tableau $T\in SST_{C_n}$
 satisfying the conclusions \eqref{vector} from \cite[Lemma 49]{schumanntorres}, here Lemma \ref{encode}, has the cancellation property
if, at each step of the word reading, every $\bar i$ in the word $w(T)$ (for $ i \in \{1,\dots , n\}$), can be
paired with one $i $  to its left which is not paired with another $\overline i$ at a previous step.
\end{defi}

\begin{prop}\label{prop:gineq}\cite[Proposition 5.6]{schumanntorres} Let $T$ be a semi-standard Young tableau of shape $\lambda$ and content $\mu$
with entries in $C_n$, satisfying the conclusions from \cite[Lemma 49]{schumanntorres} (here Lemma \ref{encode}) and such that the word of $T$, $w(T)$, has the
cancellation property. Then $T$ belongs to the set $\domres(\lambda, \mu)$ (that is, $w(T)$ has the dominance
property) if and only if for every  $1\le i\le n-1$, and every  $i\le l\le n$, the following inequality
holds for the vector $\v(T)$

\begin{align} \label{ineqdom}
(i,i)- (i+1,i+1)-\sum_{k=i}^l[(\overline i,k)-(\overline{i+1},k-1)] \ge 0,
\end{align}
where for $i\in C_n$, $(i, j):=\mbox{ number of entries $i$ in row $j$}$ of $T$.
\end{prop}

\begin{thm} \label{polytope}\cite[Theorem 57]{schumanntorres} Let $\mu\subseteq \lambda$ partitions where $\ell(\mu)\le \ell(\lambda)\le n$. Let $\mathcal{DR}(\lambda, \mu)
\subset \R^{n(n+1)/2}$ be the convex polytope of vectors
as in \eqref{vector} satisfying the inequalities \eqref{ineqdom}. Then
\begin{align}\domres(\lambda,\mu)\rightarrow \mathcal{DR}(\lambda, \mu), \quad T\rightarrow & \v(T)
\end{align}
is a bijection between the tableaux in $\domres(\lambda, \mu)$ and the lattice points of the polytope $\mathcal{DR}(\lambda, \mu)$.
\end{thm}
For example, $T\in SST_{C_6}$ in \eqref{dominantC}, equivalently $S'\in SST_{12}$ in \eqref{dominantA}, have $\ell(\mu)\le \ell(\lambda)\le n$ with $n=6$, and are in the conditions of the previous proposition. In particular,  $S'$ satisfies the conclusions from \cite[Lemma 49]{schumanntorres}, that is, $S'$ or $T\in SST_{C_6}$ is defined by the vector $\v(T)$ of non-negative numbers $\{(i,j): i\in C_6, j\in \{1,\dots,6\}\}$  where for $i\in C_6$, $(i, j):=\mbox{ number of entries $i$ in row $j\in\{1,\dots,6\}$}$ of $T\in \in SST_{C_6}$, as displayed below
$$\begin{array}{cccccccccccccc}
\,(1,1)=11&&&&&&&&\\
(2,2)=9\,&(\overline 1,2)=0&&&&&&&&\\
(3,3)=6\,&(\overline 2,3)=2&(\overline 1,3)=0&&&&&\\
(4,4)=0\,&(\overline 3,4)=4&(\overline 2,4)=1&(\overline 1,4)=2&&&&\\
(5,5)=0\,&(\overline 4,5)=0&(\overline 3,5)=0&(\overline 2,5)=2&(\overline 1,5)=3&&&\\
(6,6)=0\,&(\overline 5,6)=0&(\overline 4,6)=0&(\overline 3,6)=0&(\overline 2,6)=0&(\overline 1,6)=1&&
\end{array}
$$

On the other hand, for $n=3$, $T\in SST_{ C_3}$  or $S\in SST_6$ is not parameterized by a vector of non-negative numbers $(i,j)\in C_3\times \{1,2,3\}$ satisfying \eqref{vector} because $\ell(\lambda)=6>3$.
\subsection{Symplectic tableaux and k-highest weight tableaux}
Fix $n\in\mathbb{N}$ and let $SST_{2n}$ the set of all semi-standard tableaux in the alphabet $\mathcal{A}_{2n}$.

\begin{defi}\cite{king76} \label{def:symp}A semistandard tableau $T \in SST_{2n}(\lambda)$ is said to
be symplectic if
$$T(k, 1) \ge 2k-1, \mbox{  for all $  k \in  [ \ell(\lambda)]$}.$$
Let $SpT_{2n}(\lambda)\subset SST_{2n}(\lambda)$ denote the set of all symplectic tableaux of shape $\lambda$ on the alphabet $\mathcal{A}_{2n}$.
\end{defi}
It follows from the definition that if $T\in SpT_{2n}(\lambda)$ then $\ell(\lambda)\le n$.

Let $n\in\mathbb{N}$ and let us consider the two sequences of positive integers defined in \cite{nsw}.
For $ i=1,\dots,n$,
\begin{align}& u_i=2i-\frac{1+(-1)^i}{2}=\begin{cases}2i,&\text{ if } i \notin 2\Z,\\
2i-1,&\text{ if } i \in 2\Z,\label{numbers:u}
\end{cases}\\
& v_i=2i-\frac{1+(-1)^{i+1}}{2}=\begin{cases}2i,&\text{ if } i \in 2\Z,\\
2i-1,&\text{ if } i \notin 2\Z. \label{numbers:v}
\end{cases}
\end{align}

It is immediate from the definition of these numbers that any semi-standard tableau in $SST_{2n}$ whose first column is equal to $u_1\cdots u_n$ or $v_1\cdots v_n$ is symplectic.
These two sequences
\begin{align}\{u_i\}_{i=1}^n=\begin{cases}\{2,3,6,7,10,\dots, 2(n-1)-1,2n\},& n\notin 2\Z,\\
\{2,3,6,7,\dots, 2(n-1),2n-1\},& n\in 2\Z\end{cases}\label{numbers:uu}
\end{align}
and
\begin{align}\{v_i\}_{i=1}^n=\begin{cases}\{1,4,5,8,\dots, 2(n-1)-1,2n\},& n\in 2\Z\\
\{1,4,5,\dots, 2(n-1),2n-1\},& n\notin 2\Z,\end{cases}\label{numbers:vv}
\end{align}
 have no common values and its union  gives  $\{u_i\}_{i=1}^n\sqcup\{v_i\}_{i=1}^n=[1,2n]$.

\begin{ex} For $n=6$, $u_i=2,3,6,7,10,11$ and $v_i=1,4,5,8,9,12$, $1\le i\le 6$, and $\{u_i\}_{i=1}^6\cup \{v_i\}_{i=1}^6=\{1,2,\dots,12\}$; and for $n=7$, $u_i=2,3,6,7,10,11, 14$ and $v_i=1,4,5,8,9,12,13$, $1\le i\le 7$,
and $\{u_i\}_{i=1}^6\cup \{v_i\}_{i=1}^6=\{1,2,\dots,14\}$.
\end{ex}

A non-standard realization of the symplectic Lie algebra is  $\k$,  another  certain subalgebra of $\mathfrak{g}$, isomorphic to the symplectic Lie algebra $ \mathfrak{sp}(2n,C)$ and which  arises in $\imath$quantum symmetric pairs \cite{watanabe,nsw}. We now define $\k$-highest tableaux in $SST_{2n}(\lambda)$.
 For $S\in SST_{2n}(\lambda)$, its $\mathfrak{k}$-weight is \cite[Section 5.3]{nsw}
\begin{align}\wt_\mathfrak{k}(S)=(S[u_1]-S[v_1])\tilde\varepsilon_1+(S[u_2]-S[v_2])\tilde\varepsilon_2+\dots+(S[u_n]-S[v_n])\tilde\varepsilon_n.\label{kweight}
\end{align}

For example for $n=3$, $\wt_\mathfrak{k}(S)=(S[2]-S[1])\tilde\varepsilon_1+(S[3]-S[4])\tilde\varepsilon_2+(S[6]-S[5])\tilde\varepsilon_3$.

The set $\widetilde P^+=\{\mu_1 \tilde\varepsilon_1+\cdots+\mu_n \tilde\varepsilon_n\in\widetilde P:\mu_1\ge\cdots\ge \mu_n\ge 0\}$ can be identified with the set $Par_{\le n}$ \cite[Section 5.1]{nsw}.

 If $S$ is  the symplectic tableau below on the LHS 

 \begin{align}\label{symphw-hw}
\YT{0.2in}{}{
 {{u_1},{u_1},\cdots,\cdots,\cdots,\cdots,u_1},
 {{u_2},\cdots,\cdots,\cdots,\cdots,{u_2}},
 {{\vdots},\vdots,\vdots,{\vdots}},
 {u_n,\cdots,u_n},
} &\qquad\qquad
\YT{0.2in}{}{
 {{v_1},{v_1},\cdots,\cdots,\cdots,\cdots,v_1},
 {{v_2},\cdots,\cdots,\cdots,\cdots, {v_2}},
 {{\vdots},\vdots,\vdots,{\vdots}},
 {v_n,\cdots,v_n},
}\\
\mbox{ symplectic $\mathfrak{k}$-highest weight tableau } & \quad \mbox{ symplectic $\mathfrak{k}$-lowest weight tableau }\nonumber
\end{align}
 then the shape $(S[u_1]\ge S[u_2]\ge \dots\ge S[u_n])$ is equal to  its $\mathfrak{k}$-weight,
  $$\wt_\mathfrak{k}(S)=(S[u_1]\ge S[u_2]\ge \dots\ge S[u_n]).$$
 If $S$ is on the RHS \eqref{symphw-hw} then its $\mathfrak{k}$-weight is
   $-(S[v_1]\ge S[v_2]\ge \dots\ge S[v_n])$
   and the shape is
  $$w_0\wt_\mathfrak{k}(S)=w_0(-S[v_1],-S[v_2],\dots,-S[v_n])=(S[v_1]\ge S[v_2]\ge\dots\ge S[v_n])$$
  \noindent where $w_0$ is the longest element of the Weyl group of the Lie algebra $\mathfrak{k}$ isomorphic to the symplectic Lie algebra $\mathfrak{sp}_{2n}(\mathbb{C})$,

  For $n\in\mathbb{N}$, the symplectic $\mathfrak{k}$-highest weight  tableaux  in $SpT_{2n}$ comprise the symplectic columns
\begin{align}u_1\cdots u_{n-1}u_n\in SpT_{2n}(\varpi_n),~u_1\cdots u_{n-1}\in SpT_{2n}(\varpi_{n-1}),\dots, u_1u_2\in SpT_{2n}(\varpi_{2}), ~u_1\in SpT_{2n}(\varpi_1)
\end{align}

The \emph{symplectic} $\mathfrak{k}$-\emph{highest} \emph{weight} \emph{tableau} in $SpT_{2n}(\mu)$ is denoted by $S^{H,\mu}$ where
$$\mu=(S[u_1]\ge S[u_2]\ge \dots\ge S[u_n])$$ depicted on the LHS of \eqref{symphw-hw}.
\begin{ex} Let $n=4$, $S=(1,2,3,4)\in SST_{8}(\varpi_4)$. Then $\wt_\mathfrak{k}(S)=(S[2]-S[1],S[3]-S[4], S[6]-S[5], S[7]-S[8]  )$ $=(0,0,0,0)$.
Let $n=4$, $S=(5,6,7,8 )\in SpT_8(\varpi_4)$ then $$\wt_\mathfrak{k}(S)=(S[2]-S[1],S[3]-S[4], S[6]-S[5],S[7]-S[8])=(0,0,0,0)$$
\end{ex}

More generally, since the sequence $u_1,\dots,u_n$ nor the sequence $v_1,\dots,v_n$ contain an odd number followed with a consecutive even number, then the $\wt_\mathfrak{k}$-weight of any interval fixed by the parity involution  $\s$ is the null vector. That is, if $S=(a_1, a_1+1,\dots, a_1+t-1)\in\SST(\varpi_t)$ with $2\le t\in 2\Z$ and $a_1\notin 2\Z$, then for each $0\le j\le t-2$, either $a_1+j\in \{u_i\}_{i=1}^n$ and $a_1+j+1\in \{v_i\}_{i=1}^n$ or $a_1+j+1\in \{u_i\}_{i=1}^n$ and $a_1+j\in \{v_i\}_{i=1}^n$. Therefore,
$\wt_\mathfrak{k}(S)=0\in \Z_{\ge 0}^n.$

Recall the quantum Littlewood-Richardson rule of type $AII$ \cite{watanabe}
\begin{align} \label{lrsII}
 \LRAII^{AII}: SST_{2n}(\lambda) \overset{\sim}\longrightarrow
\bigsqcup_{\begin{smallmatrix}\mu\in Par_{\le n}\\
\mu\subseteq\lambda\\
Q\in Rec_{2n}(\lambda/\mu)\end{smallmatrix}}SpT_{2n}(\mu) \times \{Q\}.
\end{align}
where $LRS_{2n}(\lambda/\mu)\overset{\sim}{\underset{\lozenge}\longrightarrow} Rec_{2n}(\lambda/\mu)$,
and its inverse ${\LRAII^{AII}}^{-1}$ together with the inverse reduction $\redu^{-1}$ \cite{azreco, azreduction}.
Then,  for each $S\in \SpT(\mu)$,
$$T\in{\LRAII^{AII}}^{-1}(\{S \}\times Rec_{2n}(\lambda/\mu))\Rightarrow \wt_\mathfrak{k}(T)=\wt_\mathfrak{k}(S).$$

For each $\lambda\in Par_{\le 2n}$, let
\begin{align}SST^{\mathfrak{k}-hw}
_{2n}(\lambda):=\bigsqcup_{\begin{smallmatrix}\mu\in Par_{\le n}\\
\mu\subseteq\lambda\end{smallmatrix}} {\LRAII^{AII}}^{-1}(\{S^{H,\mu}\}\times Rec_{2n}(\lambda/\mu)).
\end{align}
be the set of all $\mathfrak{k}$-highest weight tableaux of shape $\lambda$ in $SST_{2n}$.
Then given $\lambda\in Par_{\le 2n}$ and  $\mu\in Par_{\le n}$ such that $\mu\subseteq \lambda $, \cite{azreco}

$$SST^{\mathfrak{k}-hw}
_{2n}(\lambda,\mu):={\LRAII^{AII}}^{-1}(\{S^{H,\mu}\}\times Rec_{2n}(\lambda/\mu)):=\{S\in    SST^{\mathfrak{k}-hw}_{2n}(\lambda)| \wt_{\mathfrak{k}}(S)=\mu\} \subseteq SST_{2n}(\lambda).$$
denotes the set of all $\mathfrak{k}$-highest weight tableaux of shape $\lambda$ in $SST_{2n}$ and $\k$-weight $\mu$.

 When $\ell(\lambda),\ell(\mu)\le n$ and $\mu=\lambda$, then $SST^{\mathfrak{k}-hw}
_{2n}(\lambda,\lambda):={\LRAII^{AII}}^{-1}(\{S^{H,\lambda}\}\times Rec_{2n}(\lambda/\lambda))=\{S^{H,\lambda}\}$, and $S^{H,\lambda}$ the $\k$-highest weight of $SpT_{2n}(\lambda)$.

\section{Linear inequalities for k-highest weight and dominant tableaux  via  the inverse quantum LR bijection for 1-0-slack recording tableaux }\label{sec:chracterization}
For a given positive integer $n$, in \cite[Theorem 8, Theorem 9]{azreco} and their corollaries for $n=3,4$, we have characterized the $\mathfrak{k}$-highest weight tableaux in $\SST$ with shape lengths $2n$ or $2n-1$ determined by $Rec^{1,0}_{2n}$, the quantum recording tableaux in $Rec_{2n}$ with $1$-$0$-slack sequences. In particular, we have exhibited all those tableaux for $n\le 4$ using the inverse quantum Littlewood-Richardson bijection ${\LRAII^{AII}}^{-1}$ via reverse Schensted insertion \cite{azreco,azslack}. In the cases of $n\le 2$, they were already exhibited in \cite[Section 6]{nsw} by establishing a bijection between them and corresponding $\widehat{\mathfrak{g}}$-dominant tableaux.  For $n=2$, $\widehat{\mathfrak{g}}$-dominant tableaux were also previously characterized in \cite{torres}.

\subsection{Promotion}

The following operation is the Sch\"utzenberger’s promotion operator $\pr$ \cite{promotion}. We use the formulation  in \cite{nsw}. Let $1\le n\in \mathbb{N}$.

\begin{defi}\cite{nsw} \label{def:promotion} Let $1\le a\le b\le 2n$ be integers, and $T\in \SST(\lambda)$. A new semi-standard tableau $\pr^{-1}_{a,b}(T)$ is defined by the following algorithm:

\begin{enumerate}
\item
Remove all boxes with entries not contained $[a,b]$ from $T$, and denote by $T'$ the remaining tableau.

\item
Decrease each entry of $T'$, except for $a$, by $1$, and replace the entries $a$ by $b$. \label{stepdecrease}
\item
Apply the \emph{jeu}-\emph{de}-\emph{taquin} slide \cite{fulton, stanley} to the boxes whose entry equals $b$, from right to left.
\item
Let $\pr^{-1}_{a,b}(T)$ be the tableau obtained by adding the boxes removed in (1) to the tableau obtained in (3).
\end{enumerate}
\end{defi}
\begin{prop}\cite{nsw} The operator $\pr^{-1}_{a,b}:\SST(\lambda)\rightarrow\SST(\lambda)$ is bijective. Its inverse can be constructed by applying steps $(1)--(4)$ in reverse order. The inverse operator  $\pr_{a,b}$ is called the promotion between $a$ and $b$.
\end{prop}

From  \cite[Definition 7.19.]{nsw}, given $1\le n\in \N$, the operator
$\Phi:\SST(\lambda)\rightarrow \SST(\lambda)$ with $\overline i=2n-i+1$, $1\le i\le n$, is a bijection defined according to the parity of $n$ by
\begin{align}\label{Phicases}\Phi=\begin{cases}\pr_{1,\overline 1}\circ\pr_{3,\overline 2}\circ\pr_{3,\overline 3}\circ\cdots\circ  \pr_{n-3,\overline{n-4}}\circ\pr_{n-3,\overline{n-3}}\circ \pr_{n-1,\overline{n-2}}\circ\pr_{n-1,\overline{n-1}},& \mbox{ if $n\in 2\Z$}\\
\pr_{1,\overline 1}\circ\pr_{3,\overline 2}\circ\pr_{3,\overline 3}\circ\cdots\circ \pr_{n-2,\overline{n-3}}\circ\pr_{n-2,\overline{n-2}}\circ \pr_{n,\overline{n-1}}\circ\pr_{n,\overline{n}},& \mbox{ if $n\notin 2\Z$}
\end{cases}
\end{align}

For $n=1$, $\Phi:=\pr_{1,2}$, for $n=2$,    $\Phi:=\pr_{1,4}$, for $n=3$,   $\Phi:=\pr_{1,6}\circ \pr_{3,5}\circ\pr_{3,4}$, and for $n=4$,  $\Phi:=\pr_{1,8}\circ \pr_{3,7}\circ\pr_{3,6}$.

\begin{prop}\label{special}Let $n\in\N$ and $\Phi$ in \eqref{Phicases}.
\begin{enumerate}
\item [(a)]Let $S^{H,\varpi_n}\in SpT_{2n}(\varpi_n)$. Then $S^{H,\varpi_n}=(u_1,\cdots, u_n)$ such that for $i=1,\dots,n$,
$$u_i=\begin{cases}2\times i-1, & i\in 2\Z\\
2\times i,& i\notin 2\Z
\end{cases}$$
\noindent and $\Phi^{-1}(u_1,\cdots, u_n)=(1,2,\dots,n)$.

\item [(b)]Let $S^{H,\mu}\in SpT_{2n}(\mu)$. Then $\Phi^{-1}(S^{H,\mu})=Y(\mu)\in \SST(\mu)$, where
$Y(\mu)$ is the Yamanouchi tableau of shape $\mu$, that is the tableau  of shape  $\mu$ such that for every $1\le i\le n$ the entries in row $i$ are equal to $i$.

\item [(c)] $ \Phi^{-1}(12\dots 2n)^M=(12\dots 2n)^M=(1\dots n\,\overline n\dots\overline 1)^M $, $M> 0$. \label{3}
\item [(d)] If $l\in 2\Z$ and $l=2q<2n$ then $\Phi^{-1}(12\dots 2q)=(12\dots q\,\overline q\dots\overline 1)$ $=(12\dots q,\, 2n-q+1, \dots,2n)\neq (12\dots l)$.
\end{enumerate}
\end{prop}
\begin{proof} $(a)$ In this case Definition \ref{def:promotion} reduces to step 2 \eqref{stepdecrease} and jeu de taquin is not needed. So  $(b)$ reduces  $(a)$.

We prove $( 1)$, for $n\notin 2\Z$. Let $n=2q+1$ where $q\ge 0$. In this case $(u_1,\dots, u_n)=(2,3,6,\dots,2n)$. Then from \eqref{Phicases},

$$\Phi^{-1}=\pr_{2q+1,\overline{2q+1}}\circ \pr_{2q+1,\overline{2q}}\circ\cdots\circ \pr^{-1}_{2k+1,\overline {2k+1}}\pr^{-1}_{2k+1,\overline 2k}\cdots \pr^{-1}_{3,\overline 3}\pr^{-1}_{3,\overline 2}\circ\pr^{-1}_{1,\overline 1}
$$
Let $0\le k$. We prove by induction on $k\ge 0$ that

\begin{align}&\pr^{-1}_{2k+1,\overline {2k+1}}\pr^{-1}_{2k+1,\overline 2k}\cdots \pr^{-1}_{3,\overline 3}\pr^{-1}_{3,\overline 2}\circ\pr^{-1}_{1,\overline 1}(u_1,\dots,u_n)=\nonumber\\
&=(1,2,\dots,2k+1,2k+2,u_{2k+3}-(2k+1),u_{2k+4}-(2k+1),\dots, u_n-(2k+1))
\end{align}
and when $k=q$ one obtains $\Phi^{-1}(u_1,\cdots, u_n)=(1,2,\dots,n)$.
For $k=0$ one has
\begin{align}&\pr^{-1}_{1,\overline 1}(u_1,\dots,u_n)=
(1,2,u_{3}-1,u_{4}-1,\dots, u_n-1).
\end{align}
and for $k=1$

\begin{align}& \pr^{-1}_{3,\overline 3}\pr^{-1}_{3,\overline 2}\circ\pr^{-1}_{1,\overline 1}(u_1,\dots,u_n)=(1,2,3,4,u_{5}-3,u_6-3,\dots, u_n-3).
\end{align}
By induction on $q>k\ge 0$, and noting that $\overline {2k+3}=2n-(2k+3)+1=2n-(2k+2)$, $\overline{2k+2}=2n-(2k+2)+1=2n-(2k+1)$, and $u_{2k+3}-(2k+1)=2(2k+3)-(2k+1)=2(2k+1)+4-(2k+1)=2k+5$, and
$u_{2k+3}-(2k+2)=2(2k+3)-(2k+2)=2(2k+1)+4-(2k+1)-1=2k+4$,

\begin{align}&\pr^{-1}_{2k+3,\overline {2k+3}}\pr^{-1}_{2k+3,\overline {2k+2}}\pr^{-1}_{2k+1,\overline {2k+1}}\pr^{-1}_{2k+1,\overline {2k}}\cdots \pr^{-1}_{3,\overline 3}\pr^{-1}_{3,\overline 2}\circ\pr^{-1}_{1,\overline 1}(u_1,\dots,u_n)=\nonumber\\
&=\pr^{-1}_{2k+3,\overline {2k+3}}\pr^{-1}_{2k+3,\overline {2k+2}}(1,2,\dots,2k+1,2k+2,u_{2k+3}-(2k+1),u_{2k+4}-(2k+1),\dots, u_n-(2k+1))\nonumber\\
&=\pr^{-1}_{2k+3,\overline {2k+3}} (1,2,\dots,2k+1,2k+2,u_{2k+3}-(2k+1)-1,u_{2k+4}-(2k+1)-1,\dots, u_n-(2k+1)-1)\nonumber\\
&=\pr^{-1}_{2k+3,\overline {2k+3}} (1,2,\dots,2k+1,2k+2,u_{2k+3}-(2k+2),u_{2k+4}-(2k+2),\dots, u_n-(2k+2))\nonumber\\
&=(1,2,\dots,2k+1,2k+2,2k+3,2k+4,u_{2k+5}-(2k+2)-1,\dots,u_n-(2k+2)-1).\nonumber
\end{align}
\end{proof}

\begin{thm}\cite[Theorem 5.18.]{nsw} For all $\lambda\in Par_{\le 2n}$ there exist two bijections

\begin{align}\Phi:SST^{\widehat{\mathfrak{g}}-dom}_{2n}(\lambda)\overset{\sim}\longrightarrow SST^{\mathfrak{k}-hw}_{2n}(\lambda) \quad \Psi:SST^{\widehat{\mathfrak{g}}-dom}_{2n}(\lambda)\overset{\sim}\longrightarrow SST^{\mathfrak{k}-lw}_{2n}(\lambda)
\end{align}
such that such that for each $T\in SST^{\widehat{\mathfrak{g}}-dom}_{2n}(\lambda)$
$$\wt_{\widehat{\mathfrak{g}}}(T)=\wt_{\widehat{\mathfrak{k}}}(\Phi(T))=\omega_0\wt_{\widehat{\mathfrak{k}}}(\Psi(T)),$$
where $w_0$ is the longest element of the Weyl group of the Lie subalgebra $\k$.
\end{thm}

In particular one has the following.

\begin{cor}\cite{nsw}\label{prop:promn3} Let $T\in SST^{\mathfrak{k}-hw}
_{2n}(\lambda)$. For $n=1$ let  $\Phi:=\pr_{1,2}$, for $n=2$,  let  $\Phi:=\pr_{1,4}$, for $n=3$, let  $\Phi:=\pr_{1,6}\circ \pr_{3,5}\circ\pr_{3,4}$, and for $n=4$, let $\Phi:=\pr_{1,8}\circ \pr_{3,7}\circ\pr_{3,6}$.
Then, for $n=1,2,3,4$, the following equalities hold

\begin{align} \SST^{{\widehat{\mathfrak{g}}-dom}}(\lambda)=\Phi^{-1}\SST^{\mathfrak{k}-hw}(\lambda),\quad \wt_{\mathfrak{k}}(T)=\wt_{\widehat{\mathfrak{g}}}(\Phi^{-1}(T)).
\end{align}
\end{cor}

\subsection{ k-highest weight and dominant tableaux via 1-0 slack recording tableaux}
Fix $n\in \N$ and let $\lambda\in Par_{\le 2n}$ and  $\mu\in Par_{\le n}$ be such that $\mu\subseteq \lambda $.
We set, from \cite{azreco},
\begin{align}\label{recording0-1}&{Rec}_{2n}^{1,0}(\lambda/\mu):=\{Q\in Rec_{2n}(\lambda/\mu) \mbox{ with $1$-$0$-slack sequence }\}\\
&{Rec}_{2n}^{1}(\lambda/\mu):=\{Q\in Rec_{2n}(\lambda/\mu) \mbox{ with $1$-slack sequence }\}.
\end{align}
\begin{obs} When $\lambda=\mu$, $Rec_{2n}(\mu/\mu)=\{\mu/\mu\}$ and ${LR^{AII}}^{-1}(\{S^{H,\mu}\}\times Rec_{2n}(\mu/\mu))=\{S^{H,\mu}\}$ and by Proposition \ref{special}, $\Phi^{-1}(S^{H,\mu})=Y(\mu)$. We consider from now on $\mu\subsetneq \lambda$.
\end{obs}
\begin{obs}\cite{azreco} \label{1slck} Let $Q\in {Rec}_{2n}^{1}(\lambda/\mu)$. Then $Q$ has a vertical strip decomposition $$\lambda=\mu^{(0)}\supset_{vert} \mu^{(1)}\supset_{vert} \cdots\supset_{vert}\mu^{(N-1)}\supset_{vert}\mu^{(N)} =\mu$$
where $N\ge 1$ such that, for all $1\le i\le N$,
\begin{enumerate}
\item
 $2n-1\ge Q[i]+1=\ell(\mu^{(i-1)}) \mbox{ and } \ell(\mu^{(i-1)})\notin 2\Z$ where $Q[i]$ is the length of the vertical strip $\mu^{(i-1)}/\mu^{(i)}$,

\item  $\ell(\lambda)\notin 2\Z$.

\end{enumerate}
\end{obs}

\begin{obs}\cite{azreco} \label{1-0slck} Let $Q\in {Rec}_{2n}^{1,0}(\lambda/\mu)$. Then $Q$ has a vertical strip decomposition  as in the previous Remark with a  tail of $M>0$ vertical strips of length $\ell(\lambda)$  if and only if $\ell(\lambda)\in 2Z$, that is,
$$\lambda=\mu^{(0)}\supset_{vert} \mu^{(1)}\supset_{vert}\cdots\supset_{vert} \mu^{(M-1)}\supset_{vert} \cdots\supset_{vert}\mu^{(N-1)}\supset_{vert}\mu^{(N)} =\mu$$

where $N\ge 1$ such that,
\begin{enumerate}
\item  for all $M\le i\le N$,
 $2n-1\ge Q[i]+1=\ell(\mu^{(i-1)}) \mbox{ and } \ell(\mu^{(i-1)})\notin 2\Z$ where $Q[i]$ is the length of the vertical strip $\mu^{(i-1)}/\mu^{(i)}$,

\item for all $0\le i\le M-1$, $2n\ge Q[i]=\ell(\lambda) \mbox{ and }  \ell(\lambda)\in 2\Z$ if $M>0$.

\item Let $Q^{\textsf{tail}}$ be of shape $\lambda/\mu^{(M)}$ consisting of the vertical strips of length $\ell(\lambda)$ in $Q$, $\lambda=\mu^{(0)}\supset_{vert} \mu^{(1)}\supset_{vert}\cdots\supset_{vert} \mu^{(M-1)}\supset_{vert} \mu^{(M)}$. Then $Q^{\textsf{tail}}\in  {Rec}_{2n}^{1,0}(\lambda/\mu^{(M)})$ has $0$-slack sequence and is called the tail of $Q$ whose length is $M$.
    We write
    $$Q=Q\setminus Q^{\textrm{tail}}\bigsqcup Q^{\textrm{tail}},$$
    where $Q\setminus Q^{\textrm{tail}}$ is defined by the vertical strip $\mu^{(M)}\supset_{vert} \cdots\supset_{vert}\mu^{(N-1)}\supset_{vert}\mu^{(N)} =\mu$ and $Q=Q^{\textrm{tail}}$ when $M=N$.
\end{enumerate}
\end{obs}

Let  $l\in [0,2n]$ and recall that the reduction map $\redu$ on $SST_{2n}(\varpi_l)$ is  the bijective assignment \cite{watanabe,azreduction}

\begin{align}\redu=\redu_l:SST_{2n}(\varpi_l)&\overset{\sim}\rightarrow\bigsqcup_{\begin{smallmatrix} 0\le t\le min\{l,2n-l\}\\
l-t\in 2\mathbb{Z}
\end{smallmatrix}}SpT_{2n}(\varpi_t),\quad \mathbf{a}\mapsto \redu( \mathbf{a}).
\label{redumap}
\end{align}

Given $l\in [0,2n]$ and $t\in[0,n]$ such that  $0\le t\le min\{l,2n-l\}$ and $l-t\in 2\Z$,
explicit procedures were given in \cite{azreduction} to compute the inverse of the reduction map
\begin{align}\label{exp}
\redu^{-1}_t:SpT_{2n}(\varpi_t)&\rightarrow SST_{2n}(\varpi_l).
\end{align}

Since we are restricting our analysis to the subset of recording tableaux  ${Rec}_{2n}^{1,0}(\lambda/\mu)\subseteq Rec_{2n}(\lambda/\mu) $ \eqref{recording0-1}, in our case, $t\in \{0,1\}$, which means that

  \begin{align}\label{0-1rule}l\notin 2\Z\Rightarrow t=1 \mbox{  and } l\in 2\Z\Rightarrow t=0.
  \end{align}
and the procedure to compute \eqref{exp} is \cite[Theorem 2]{azreduction}. Namely, for $l\notin 2\Z$ and $t=1$,
\begin{align}\label{expexp}
\redu^{-1}_1:SpT_{2n}(\varpi_1)&\rightarrow SST_{2n}(\varpi_l)\\
({a})&\mapsto \redu^{-1}_1({a})=\begin{cases}(1\dots l_1)a (a+1,\dots,a+l-l_1-1), & \mbox{ if } a\in 2\Z,\\
(1\dots l_1)a(a+2,\dots,a+l-l_1),& \mbox{ if } a\notin 2\Z,
\end{cases}
\nonumber
\end{align}

where $$l_1=\begin{cases}min\{a-2,l-1\}& a\in 2\Z\\
min\{a-1,l-1\}& a\notin 2\Z;
\end{cases}$$
and, for $l\in 2\Z$ and $t=0$,
$\redu^{-1}_0(())=(12\dots l)$.

\begin{obs} For $l=0$, $t=0$ and for $l=1$, $t=1$
\begin{align*}
\redu^{-1}_0:SpT_{2n}(\varpi_0)&\rightarrow SST_{2n}(\varpi_0)\\
()&\mapsto \redu^{-1}_0()=()
\end{align*}
and
\begin{align*}
\redu^{-1}_1:SpT_{2n}(\varpi_1)&\rightarrow SST_{2n}(\varpi_1)\\
({a})&\mapsto \redu^{-1}_1({a})=(a)
\end{align*}

\end{obs}

More precisely, recalling that ${\LRAII^{AII}}^{-1}(\{S^{H,\mu}\}\times Rec^{1,0}(\lambda/\mu))$ \cite{azreco} is obtained by reverse Schensted insertion followed with the inverse reduction map $\redu_1^{-1}$ or $\redu_0^{-1}$ \eqref{exp},
\begin{align}&SST^{\mathfrak{k}-hw}
_{2n}(\lambda,\mu)\supseteq\nonumber\\
&\supseteq{\LRAII^{AII}}^{-1}(\{S^{H,\mu}\}\times Rec^{0,1}_{2n}(\lambda/\mu)):=\{S\in    SST^{\mathfrak{k}-hw}_{2n}(\lambda)| Q^{AII}(S)\in Rec^{0,1}(\lambda/\mu)\wedge \wt_{\mathfrak{k}}(S)=\mu\}
\end{align}
and  in \eqref{redumap}, \eqref{exp},  $l\le \ell(\lambda)\le 2n$.

 In the next lemma, points $(1)$ and $(2)$, $(a)$ follow from the discussion above.
\begin{lem}\label{lem:basic} Given $\mu\in Par_{\le n}$ and $l\notin 2\Z$ such that $\ell (\mu)\le l\le 2n$, it holds
\begin{enumerate}
\item
$$\bigsqcup_{\begin{smallmatrix} \lambda\in Par_{2n}\\
\ell(\lambda)=l\\
\mu\subseteq \lambda
\end{smallmatrix}}Rec^{1,0}_{2n}(\lambda/\mu)=\big\{Q\setminus Q^{tail}:Q\in\bigsqcup_{\begin{smallmatrix} \lambda\in Par_{2n}\\
\ell(\lambda)=l+1\\
\mu\subseteq \lambda
\end{smallmatrix}}Rec^{1,0}_{2n}(\lambda/\mu)\big\}.$$

In other words,  the recording tableaux in
$\displaystyle \bigsqcup_{\begin{smallmatrix} \lambda\in Par_{2n}\\
\ell(\lambda)=l\\
\mu\subseteq \lambda
\end{smallmatrix}}\ Rec^{1,0}_{2n}(\lambda/\mu)$ are obtained from those in

 $\displaystyle\bigsqcup_{\begin{smallmatrix} \lambda\in Par_{2n}\\
\ell(\lambda)=l+1\\
\mu\subseteq \lambda
\end{smallmatrix}}\ Rec^{1,0}_{2n}(\lambda/\mu)$ by suppressing the \textsf{tails}.
\item  Let $Q\in Rec^{1,0}_{2n}(\lambda/\mu)$ with $\ell(\lambda)=l+1$ such that $\lambda-M\varpi_{l+1}$ has length $l$ for some $M>0$. Then
\begin{enumerate}
\item  $S={LR^{AII}}^{-1}(S^{H,\mu}, Q)$ if and only if  $S=(1,2,\dots l, \, l+1)^M S'$,  and

$S'={LR^{AII}}^{-1}(S^{H,\mu}, Q\setminus Q^{\textsf{tail}})$ where $Q^{\textsf{tail}}$.

In other words $S\in\SST^{\mathfrak{k}-hw}(\lambda,\mu)\Leftrightarrow
S'\in \SST^{\mathfrak{k}-hw}(\lambda-M\varpi_{l+1},\mu)$

\item When $l+1=2n$, $\Phi^{-1}(S)=\Phi^{-1}{LR^{AII}}^{-1}(S^{H,\mu}, Q)$ if and only if

$\Phi^{-1}(S)=(1,2,\dots  \, 2n)^M \Phi^{-1}(S')$ and

$\Phi^{-1}(S')=\Phi^{-1}{LR^{AII}}^{-1}(S^{H,\mu}, Q\setminus Q^{\textsf{tail}})$.
One has $$\wt_{\widehat{\mathfrak{g}}}(\Phi^{-1}(S'))=\wt_{\widehat{\mathfrak{k}}}(S')=\mu$$ and $$\wt_{\widehat{\mathfrak{g}}}\Phi^{-1}(1,2,\dots, 2n)=\wt_{\widehat{\mathfrak{g}}}(1,2,\dots, 2n)=0=\wt_{\widehat{\mathfrak{k}}}(1,2,\dots, 2n)=0$$
\end{enumerate}

\end{enumerate}
\end{lem}
\begin{proof} To justify point $(2)$, $(b)$, we notice that in the \emph{jeu} \emph{de} \emph{taquin} steps that $\Phi^{-1}$ requires, the entries in the first column have always the option to slide down instead of sliding to the right. Thus
$$\Phi^{-1}(S)=\Phi^{-1}((1,2,\dots  \, 2n)^M )\Phi^{-1}(S')=(1,2,\dots  \, 2n)^M \Phi^{-1}(S')=(1,2,\dots\, n\,\overline{n\dots 2\, 1})^M \Phi^{-1}(S').$$
\end{proof}

\begin{obs}\label{ha} Lemma \ref{lem:basic}, point $(2)$, $(b)$,  is not true for $(12\dots l)$ with $l<2n$. Let $n=3$ and $\lambda=(8,6,5,1)$, $l=\ell(\lambda)=4<6$, and

 $S=\YT{0.15in}{}{
 {1,{1},1,    {{ 2}},2,    2,   2,2},
 {2,{2},{2},3,  {3},  { 3}},
 {3, {3},4,  4,  6},
 {4},
}\in SST^{\mathfrak{k}-hw}_{6}(\lambda,\mu)$  of $\k$-weight $\mu=(5,2,1)$ where

 $P^{AII}(S)=S^{H,\mu}=\YT{0.15in}{}{
 {2,    2,   2,2,2},
 {3,  {3}},
 {6 },
 },$ and $Q^{AII}(S)\in Rec^{1,0}(\lambda,\mu)$ because all reduction maps $\redu$  used in the procedure to compute $P^{AII}(S)$  are of the form $\redu_l:SST_{2n}(\varpi_l)\rightarrow SpT_{2n}(\varpi_t)$ with $t\in \{0,1\}$.

Then, for $\Phi:=\pr_{1,6}\circ \pr_{3,5}\circ\pr_{3,4}$,

\begin{align}\label{Stypel=41example}
&\Phi^{-1}(S)=\Phi^{-1}\YT{0.15in}{}{
 {1,{1},1,    {{ 2}},2,    2,   2,2},
 {2,{2},{2},3,  {3},  { 3}},
 {3, {3},4,  4,  6},
 {4},
}
 =\YT{0.15in}{}{
 {1,\mathbf{1},1,    {{ 1}},1,    1,   1,1},
 {\mathbf{2},{2},{2},2,  {2},  { \overline 1}},
 {\mathbf{3}, \mathbf{\red{\overline 2}},{\overline 2},  {\overline 2},\overline 1},
 {\overline 1},
} \quad 
\end{align} and $\wt_{\widehat{\mathfrak{g}}}\Phi^{-1}(S)=\mu $.

Note $$\Phi^{-1}(1234)=(1256)=(1 2 \overline 2  \overline 1)\neq (123\overline 1) \mbox{ the first column of $\Phi^{-1}(S)$}$$
and $\wt_{\widehat{\mathfrak{g}}}\Phi^{-1}(1234)=\wt_{\widehat{\mathfrak{g}}}(1 2 3 \overline 1)=(0,1,1)$ while $\wt_{\widehat{\mathfrak{k}}}(1234)=0$. This means that if $S'$ is obtained from $S$ by deleting its first column, as below, $\Phi^{-1}(S')$ can not  be obtained from $\Phi^{-1}(S)$ by deleting the first column

However,  if $S'$ is obtained from $S$ by deleting its first column of $S$,
$$S'=\YT{0.15in}{}{
 {{1},1,    {{ 2}},2,    2,   2,2},
 {{2},{2},3,  {3},  { 3}},
 { {3},4,  4,  6},
}\in SST^{\mathfrak{k}-hw}_{6}(\lambda- \varpi_4,\mu)$$ has shape length $l-1=3$ and the  $\k$-weight $\mu=(5,2,1)$ as $S$.

Since $\Phi^{-1}(S')$ cannot be obtained from $\Phi^{-1}(S)$ just by deleting the first column $(1,2, 3, \overline 1)$ of $S$ because its $\widehat{\mathfrak{g}}$-weight is $(0,1,1) \neq 0$,
notably $\Phi^{-1}$  recovers the lost $\widehat{\mathfrak{g}}$-weight by changing the second  column $12\overline 2$ of $\Phi^{-1}(S)$ into $123$,
\begin{align}\label{typel=42example}&\Phi^{-1}(S')=
\Phi^{-1}\YT{0.15in}{}{
 {{1},1,    {{ 2}},2,    2,   2,2},
 {{2},{2},3,  {3},  { 3}},
 { {3},4,  4,  6},
}=\YT{0.15in}{}{
 {\mathbf{1},1,    {{ 1}},1,    1,   1,1},
 {\mathbf{2},2,  {2}, 2, { \overline 1}},
 {\mathbf{3},{\overline 2},  {\overline 2},\overline 1},
}
\end{align}
Now $$\wt_{\widehat{\mathfrak{g}}}\Phi^{-1}(S')=\wt_{{\mathfrak{\k}}}(S')=\mu=\wt_{\widehat{\mathfrak{g}}}\Phi^{-1}(S)=
\wt_{{\mathfrak{\k}}}(S)$$
Conclusion $\Phi^{-1}(S')$ is obtained from $\Phi^{-1}(S)$ by deleting the first column $(123\overline 1)$ and by replacing the column $12 \overline 2$ by $123$.

The  detailed computation of $\Phi^{-1}(S)$, $\Phi^{-1}= \pr_{3,4}^{-1}\circ \pr_{3,5}^{-1}\circ \pr_{1,6}^{-1}$, will clarify why  this is so

$$S=\YT{0.15in}{}{
 {1,{1},1,    {{ 2}},2,    2,   2,2},
 {2,{2},{2},3,  {3},  { 3}},
 {3, {3},4,  4,  6},
 {4},
}\underset{\pr_{1,6}^{-1}}\rightarrow \YT{0.15in}{}{
 {1,{1},1,    {{ 1}},1,    1,  1,1 },
 {2,{2},{2},2,  {2},  { 6}},
 {\mathbf{3}, {\bf 3},\bf 3,  \bf 5,  6},
 {6},
}\quad  \bf{3335} \underset{\pr_{3,5}^{-1}}\rightarrow \bf{555 4}\rightarrow \bf{4555}
$$
\begin{align*}&S=\YT{0.15in}{}{
 {1,{1},1,    {{ 2}},2,    2,   2,2},
 {2,{2},{2},3,  {3},  { 3}},
 {3, {3},4,  4,  6},
 {4},
}\underset{\pr_{1,6}^{-1}}\rightarrow \YT{0.15in}{}{
 {1,{1},1,    {{ 1}},1,    1,  1 ,1},
 {2,{2},{2},2,  {2},  { 6}},
 {\mathbf{3}, {\bf 3},\bf 3,  \bf 5,  6},
 {6},
}\underset{\pr_{3,5}^{-1}}\rightarrow \YT{0.15in}{}{
 {1,{1},1,    {{ 1}},1,    1,  1 ,1},
 {2,{2},{2},2,  {2},  { 6}},
 {\mathbf{4}, {\bf 5},\bf 5,  \bf 5,  6},
 {6},
}\\
&\underset{\pr_{3,4}^{-1}}\rightarrow \YT{0.15in}{}{
 {1,{1},1,    {{ 1}},1,    1,  1,1 },
 {2,{2},{2},2,  {2},  { 6}},
 {\mathbf{3}, { 5}, 5,   5,  6},
 {6},
}
=\YT{0.15in}{}{
 {1,{\bf 1},1,    {{ 1}},1,    1,  1 ,1},
 {\bf 2,{2},{2},2,  {2},  { \overline 1}},
 {\bf 3, {\overline 2},\overline 2, \overline 2,  \overline 1},
 {\overline 1},
}
\end{align*}

$$S'=\YT{0.15in}{}{
 {{1},1,    {{ 2}},2,    2,   2,2},
 {{2},{2},3,  {3},  { 3}},
 { {3},4,  4,  6},
}\underset{\pr_{1,6}^{-1}}\rightarrow \YT{0.15in}{}{
 {{1},1,    {{ 1}},1,    1,  1 ,1},
 {{2},{2},2,  {2},  { 6}},
 { {\bf 3},\bf 3,  \bf 5,  6},
}\quad  \bf{335} \underset{\pr_{3,5}^{-1}}\rightarrow \bf{55 4}\rightarrow \bf{455}
$$
\begin{align*}&S'=\YT{0.15in}{}{
 {{1},1,    {{ 2}},2,    2,   2,2},
 {{2},{2},3,  {3},  { 3}},
 { {3},4,  4,  6},
}\underset{\pr_{1,6}^{-1}}\rightarrow \YT{0.15in}{}{
 {{1},1,    {{ 1}},1,    1,  1 ,1},
 {{2},{2},2,  {2},  { 6}},
 { {\bf 3},\bf 3,  \bf 5,  6},
}\underset{\pr_{3,5}^{-1}}\rightarrow \YT{0.15in}{}{
 {{1},1,    {{ 1}},1,    1,  1 ,1},
 {{2},{2},2,  {2},  { 6}},
 {\mathbf{4}, {\bf 5},\bf 5,    6},
}\\
&\underset{\pr_{3,4}^{-1}}\rightarrow \YT{0.15in}{}{
 {{1},1,    {{ 1}},1,    1,  1,1 },
 {{2},{2},2,  {2},  { 6}},
 {\mathbf{3}, { 5}, 5,     6},
}
=\YT{0.15in}{}{
 {{\bf 1},    {{ 1}},1,    1,  1 ,1,1},
 {\bf 2,{2},2,  {2},  { \overline 1}},
 {\bf 3, \overline 2, \overline 2,  \overline 1},
}
\end{align*}

\end{obs}

\subsection{Characterization of dominant tableaux via 1-0 slack recording tableaux of shape length 2n or 2n-1}\label{subse:dom2n}
As earlier mentioned the $\k$-highest weight tableaux via 1-0 slack recording tableaux has been characterized for shape lengths $2n$ or $2n-1$ in \cite{azreco}. We now characterize the corresponding dominant tableaux for shape lengths $2n$ or $2n-1$. Recall, from \eqref{numbers:u},  for $i=1,\dots,n$,

\begin{align}\label{numbersagain:uu} u_i=2i-\frac{1+(-1)^i}{2}=\begin{cases}2i,&\text{ if } i \notin 2\Z,\\
2i-1,&\text{ if } i \in 2\Z.
\end{cases}
\end{align}

Given $T\in SST_{2n}$ and $\mathbf{b}$ a column in $SST_{2n}$ we write $m_\mathbf{b}$ to mean the multiplicity of $\mathbf{b}$ in $T$.
We also write  $\circledcirc_{q}^p$, to mean  the concatenation in the plactic monoid of $q-p+1$  words, say  $w_q\bigcdot\dots\bigcdot  w_p$, with the convention that when $q-p+1\le 0$ the concatenation is the empty word.

Let $1\le n\in \Z$. Then $u_n=2n, \;n \notin 2\Z$ and $u_n=2n-1,\; n \in 2\Z$. Recall, for $i=1,\dots,n$, $\redu_1^{-1}:SpT_{2n}(\varpi_1)\rightarrow SST_{2n}(\varpi_{2n-1})$  and

\begin{align}\label{redunumbers:uu}
\redu_1^{-1}(u_i)=\begin{cases}(12\cdots u_i-2 \;u_i=2i \;u_i+1\cdots 2n),&\text{ if } i \notin 2\Z,\, i\neq n\\
(12\cdots u_i-1 \;u_i=2i-1 \;u_i+2\cdots 2n),&\text{ if } i \in 2\Z,\,i\neq n\\
(12\cdots 2n-2\, 2n),&\text{ if } i=n\notin 2\Z\\
(12\cdots 2n-2\, 2n-1),&\text{ if } i=n\in 2\Z,
\end{cases}
\end{align}
and $\ell(\redu_1^{-1}(u_i))=2n-1$.
We recall Theorem 8 and Theorem 9 in \cite{azreco}. The reader is referred to \cite{azreco} for details on the inverse quantum LR rule.
\begin{thm}\label{thm:n=oddslack1}\cite{azreco} Let $1\le n\notin 2\Z$. Consider $S^{H,\mu}$  the $\mathfrak{k}$-highest weight tableau
in $SpT_{2n}(\mu)$ as in \eqref{symphw-hw}. Let  $Q\in  Rec^{1,0}_{2n}(\lambda/\mu)$, $\mu\subsetneq \lambda$,  with  $1$-$0$-slack sequence of the form $\underline \t=(1,\dots,1,0^M)$ and  corresponding  slack row index vector sequence   of the form
 $\underline \r=(1,\dots,1,2,\dots,2,3,\dots,3,\dots,n,\dots,n,2n-1,\dots,2n-1,()^M)$.  Let $u_1=2,\dots,u_n=2n$ be the numbers in \eqref{numbers:uu} , and $\mu'=\mu-\delta_{\underline\r^+}$ .
Then,
\begin{enumerate}
\item $2n-1\notin \underline\r^+$.

\item for some $0\le k\le n$,
  ${\LRAII^{AII}}^{-1}(S^{H,\mu},Q)$ returns   the $\mathfrak{k}$-highest  weight tableau in
$SST^{\mathfrak{k}-hw}_{2n}(\lambda,\mu)$ of the form

\begin{align}\label{generalfrog}
&Y(M^{2n})\bigcdot\nonumber\\
&\bigcdot\displaystyle\circledcirc_{j=n}^{k+1}\big(\redu_1^{-1}({u_j})\big)^{m_{\redu^{-1}(u_j)}}\bigcdot
\big(\redu_1^{-1}({u_{k}})\big)^{m_{\redu^{-1}(u_k)}}\bigcdot
{\displaystyle\circledcirc_{j=k}^{1}\big(\redu_1^{-1}({u_j})\setminus\{2n\}\big)^{m_{\redu^{-1}( u_j)\setminus \{2n\}}}}\bigcdot S^{H,\mu'},
\end{align}
where $S^{H,\mu'}$ is the $\k$-highest weight tableau in $SpT_{2n}(\mu')$,
$$\displaystyle \sum_{j=n}^{k+1} m_{\redu^{-1}(u_j)}+m_{\redu^{-1}(u_k)}+\sum_{j=1}^k m_{\redu^{-1}(u_i)\setminus \{2n\}}=|\delta_{\underline\r}|,$$
and satisfying restrictions on linear inequalities on the multiplicities of the columns

\begin{align*} &\redu^{-1}(u_i),~~  \redu^{-1}(u_i)\setminus \{2n\}, \mbox{ $i=2,\dots,n$, on the LHS of $S^{H,\mu'}$},\\
&\mbox{ and the symplectic columns  $u_1u_2\cdots u_i\in SpT_{2n}(\varpi_i)$,
$i=1,\dots,n$,}\nonumber\\
&u_1u_2\cdots u_n,~~ u_1u_2\cdots u_{n-1},\dots,u_1u_2,~~u_1
\end{align*} as follows

\begin{align}
&m_{\redu^{-1}(u_i)}+m_{\redu^{-1}(u_i)\setminus \{2n\}}\le m_{u_1\cdots u_{i-1}}, ~i=2,\dots,n,\\
&0\le m_{\redu^{-1}(2n)}-\sum_{i=1}^{n-1}m_{\redu^{-1}(u_i)\setminus\{2n\}}\le m_{u_1\cdots u_{n-1}}.
\end{align}

\end{enumerate}
\end{thm}
 \medskip
Consider  in the previous theorem
$$\underline \r=(1,\dots,1,2,\dots,2,3,\dots,3,\dots,n,\dots,n,()^M).$$
In this case, $m_{\redu^{-1}(u_i)\setminus\{u_n\}}=0$ for $i=1,\dots,n$.

\begin{cor}\label{cor:n=oddslack1}\cite{azreco} Let $1\le n\notin 2\Z$. Consider $S^{H,\mu}$  the $\mathfrak{k}$-highest weight tableau
in $SpT_{2n}(\mu)$ as in \eqref{symphw-hw} with $1\le\ell(\mu)\le n$. Let  $Q\in  Rec^{1,0}_{2n}(\lambda/\mu)$, $\mu\subsetneq \lambda$,  $1$-$0$-slack sequence of the form $\underline \t=(1,\dots,1,0^M)$ and  corresponding  slack row index vector sequence   of the form
 $\underline \r=(1,\dots,1,2,\dots,2,3,\dots,3,\dots,\ell(\mu),\dots,\ell(\mu),()^M)$.  Let $u_1=2,\dots,u_{\ell(\mu)},$ $\dots,u_n=2n$ be the numbers in \eqref{numbers:uu}, and $\mu'=\mu-\delta_{\underline\r^+}:=\mu-\sum_{k=1}^{|\underline \t|}\delta_{\r^{(k)}}$.
Then,
\begin{enumerate}

\item
  ${\LRAII^{AII}}^{-1}(S^{H,\mu},Q)$ returns   the $\mathfrak{k}$-highest  weight tableau in
$SST^{\mathfrak{k}-hw}_{2n}(\lambda,\mu)$   of the form

\begin{align}\label{frog}
&Y(M^{2n})\bigcdot\nonumber\\
&\big(\redu_1^{-1}(u_{\ell(\mu)})\big)^{m_{\redu^{-1}(u_{\ell(\mu)})}}\bigcdot \big(\redu_1^{-1}(u_{\ell(\mu)-1)})\big)^{m_{\redu^{-1}(u_{\ell(\mu)-1})}}\bigcdot\cdots\bigcdot
\big(\redu_1^{-1}(u_{1})\big)^{m_{\redu^{-1}(u_1)}}\bigcdot S^{H,\mu'},
\end{align}
where $S^{H,\mu'}$ is the $\k$-highest weight tableau in $SpT_{2n}(\mu')$ with
$$\displaystyle \sum_{j=\ell(\mu)}^{1} m_{\redu^{-1}(u_j)}=|\delta_{\underline\r^+}|,$$
satisfying restrictions on linear inequalities between  the multiplicities of the columns
\begin{align*} &\redu^{-1}(u_i),
\mbox{ $i=2,\dots,\ell(\mu)$, on the LHS of $S^{H,\mu'}$},\\
&\mbox{ and the symplectic columns  $u_1u_2\cdots u_i\in SpT_{2n}(\varpi_i)$,
$i=1,\dots,\ell(\mu)$,}\nonumber\\
&u_1u_2\cdots u_{\ell(\mu)},~~ u_1u_2\cdots u_{\ell(\mu)-1},\dots,u_1u_2,~~u_1
\end{align*} as follows
\begin{align}
&0\le m_{\redu^{-1}(u_i)}\le m_{u_1\cdots u_{i-1}}, ~i=2,\dots,\ell(\mu).
\end{align}

\end{enumerate}

\end{cor}

\begin{obs} The tableau in \eqref{frog} has shape length  $\ell(\lambda)=2n$ if $M>0$, otherwise $|\t|>0$ (because $\lambda\neq \mu$), has shape length $\ell(\lambda)=2n-1$  and, in this case, with rightmost column given by $\redu_1^{-1}(u_{i})$ for some $1\le i\le \ell(\mu)$, that is,
$1\le m_{\redu^{-1}(u_i)}$, for some  $1\le i\le \ell(\mu)$.
\end{obs}

Recall the definition of $\Phi$, \eqref{Phicases}, which depends on the parity of $n$.

\begin{thm}\label{thm:k} Assume the conditions of Corollary \ref{cor:n=oddslack1}. Let $T\in SST^{\mathfrak{k}-hw}_{2n}(\lambda,\mu)$  be the tableau  in \eqref{frog}. Let ${\bf{C}_0}=(12\cdots n\,\overline n\cdots\overline 2 \overline 1)$. Then $\Phi^{-1}(T)\in SST_{2n}^{{\widehat{\mathfrak{g}}-dom}}$ is equal to
\begin{align}\label{frogdom}C_0^M\bigcdot{(C_0\setminus\{\overline 1\})}^{m_{C_0\setminus\{\overline 1\}}}\bigcdot{(C_0\setminus\{\overline 2\})}^{m_{C_0\setminus\{\overline 2\}}}\bigcdot\dots\bigcdot {(C_0\setminus\{\overline{\ell(\mu)}\})}^{m_{C_0\setminus\{\overline {\ell(\mu)}\}}}\bigcdot Y(\mu')
\end{align}
satisfying the linear inequalities
\begin{align}\label{generaldomneq}
&m_{\{1{\oo 2}\cdots \ell(\mu)\,\cdots\,  n\,\overline n\, \cdots\overline{\ell(\mu)}\,\cdots\, \overline 1\}\setminus\{\overline 2\}}\le m_{ 1},\nonumber\\
&m_{\{1{ 2}{\red 3}\cdots \ell(\mu)\,\cdots\, n\,\overline n\,\cdots \overline{\ell(\mu)}\,\cdots\, \overline 1\}\setminus\{\overline 3\}}\le m_{ 12}\nonumber\\
&\vdots\nonumber\\
&m_{\{1{ 2}\cdots {\bro\ell(\mu)}\,\cdots\, { n}\,\overline n\cdots \overline{\ell(\mu)}\,\cdots\, \overline 1\}\setminus\{\overline {\ell(\mu)}\}}\le m_{ 12\cdots \ell(\mu)-1}
\end{align}

\end{thm}

\begin{obs} The tableau in \eqref{frogdom} has shape length  $\ell(\lambda)=2n$ if $M>0$, otherwise $|\t|>0$ (because $\lambda\neq \mu$), has shape length $\ell(\lambda)=2n-1$  and, in this case, with rightmost column given by
$$\{1{ 2}{\red 3}\cdots \ell(\mu)\,\cdots\, n\,\overline n\,\cdots \overline{\ell(\mu)}\,\cdots\, \overline 1\}\setminus\{\overline i\}$$ for some $1\le i\le \ell(\mu)$, that is,
$1\le m_{\{1{ 2}{\red 3}\cdots \ell(\mu)\,\cdots\, n\,\overline n\,\cdots \overline{\ell(\mu)}\,\cdots\, \overline 1\}\setminus\{\overline i\}}$, for some  $1\le i\le \ell(\mu)$.
\end{obs}

\begin{ex} For $n=3$, Theorem \ref{thm:1}, tableaux in $(1)$ and $(2)$ satisfy the conditions of this theorem.
\end{ex}

More generally,

 \begin{thm}\label{thm:3} Let $n\notin 2\Z$.  Assume the conditions of Theorem \ref{thm:n=oddslack1} where $1\le\ell (\mu)\le n$. Let $T\in SST^{\mathfrak{k}-hw}_{2n}(\lambda,\mu)$  be the tableau obtained in \eqref{generalfrog}. Let ${\bf{C}_0}=(12\cdots n\,\overline n\cdots\overline 2 \overline 1)$. Then $\Phi^{-1}(T)\in SST_{2n}^{{\widehat{\mathfrak{g}}-dom}}(\lambda,\mu)$ is of the form
\begin{align}\label{generalfrogdom}&C_0^M\bigcdot{(C_0\setminus\{\overline 1\})}^{m_{C_0\setminus\{\overline 1\}}}\bigcdot{(C_0\setminus\{\overline 2\})}^{m_{C_0\setminus\{\overline 2\}}}\bigcdot\dots\bigcdot {(C_0\setminus\{\overline{n}\})}^{m_{C_0\setminus\{\overline n\}}}\bigcdot {(C_0\setminus\{\overline{n},\overline{n-1}\})}^{m_{C_0\setminus\{\overline n,\overline{n-1}\}}}\bigcdot \nonumber\\
&\bigcdot(12\cdots n-1\overline{n-1\cdots 21})^{m_{12\cdots n-1\overline{n-1\cdots 21}}}\bigcdot(12\cdots n-1\overline{n-1})^{m_{12\cdots n-1\overline{n-1}}}\bigcdot\nonumber\\
&\bigcdot(12\cdots n-1)^{m_{12\cdots n-1}}\bigcdot\cdots \bigcdot(12)^{m_{12}}\bigcdot (1)^{m_1}
\end{align}
satisfying the linear inequalities on the multiplicity of the columns

\begin{align}\label{generalneq}
&m_{C_0\setminus\{\overline 2\}}\le m_{ 1},\\
&m_{C_0\setminus\{\overline 3\}}\le m_{ 12}\\
&\vdots\nonumber\\
&m_{C_0\setminus\{\overline {n-1}\}}\le m_{ 12\cdots n-2}\\
&m_{C_0\setminus\{\overline n\}}\le m_{ 12\cdots n-1}\\
&m_{C_0\setminus\{\overline{n},\overline{n-1}\}}\le m_{ 12\cdots  n-2\, \,n-1\,\overline {n-1}}
\end{align}
with $\mu=\widehat\wt(T)=(\mu_1,\mu_2,\dots,\mu_{n-1},\mu_n)$
where
\begin{align}&\mu_i=\sum_{k=i}^{n-1}m_{12\cdots k}+m_{12\cdots \scriptstyle n-1\,\scriptstyle\overline{ n-1}}
+m_{C_0\setminus \{\overline i\}},\quad 1\le i\le n-2\\
&\mu_{n-1}=m_{12\cdots n-1}+m_{C_0\setminus \{\overline {n-1}\}}+m_{C_0\setminus \{\overline{n},\overline{n-1}\}}\\
&\mu_{n}=m_{C_0\setminus \{\overline n\}}+m_{C_0\setminus \{\overline{n},\overline{n-1}\}}
\end{align}

\end{thm}
\begin{obs} The tableau in \eqref{generalfrogdom} has shape length  $\ell(\lambda)=2n$ if $M>0\Leftrightarrow m_{\{1{ 2}\cdots n-1\,  n\,\overline n\, \overline {n-1}\cdots \overline 1\}}\ge 1$, otherwise $|\t|>0$ (because $\lambda\neq \mu$), has shape length $\ell(\lambda)=2n-1$  and, in this case, with leftmost column given by
$$C_0\setminus\{\overline i\}$$ for some $1\le i\le n$, that is,
$1\le m_{C_0\setminus\{\overline i\}}$, for some  $1\le i\le n$.
\end{obs}

\begin{ex} For $n=3$, Theorem \ref{thm:1}, tableaux \eqref{ghtype33} and \eqref{ghtype34}, $(1)$ and $(2)$ satisfy the conditions of this theorem.
\end{ex}

Let $2\le n\in 2\Z$. Then $u_n=2n-1$.
\begin{thm}\cite{azreco} \label{thm:n=evenslack1}Let $2\le n\in 2\Z$. Consider $S^{H,\mu}$  the $\mathfrak{k}$-highest weight tableau
in $SpT_{2n}(\mu)$ as in \eqref{symphw-hw}. Let  $Q\in  Rec^{1,0}_{2n}(\lambda/\mu)$, $\mu\subsetneq \lambda$, with  $0$-$1$ slack sequence of the form  $\underline \t=(1,\dots,1,0^M)$ and  corresponding  slack row index vector sequence
 $\underline \r=(1,\dots,1,2,\dots,2,3,\dots,3,\dots,n,\dots,n,2n-1,\dots,2n-1,()^M)$. Let $u_1=2,\dots,u_n=2n-1$ be the numbers in \eqref{numbers:uu}, and $\mu'=\mu-\delta_{\underline \r^+}$.
Then,
\begin{enumerate}
\item $2n-1\notin \underline \r^+$

\item
 ${\LRAII^{AII}}^{-1}(S^{H,\mu},Q)$ returns   the $\mathfrak{k}$-highest  weight tableau in $SST^{\mathfrak{k}-hw}_{2n}(\lambda,\mu)$ of one of either  forms:
\begin{align}\label{evenwhale}
&Y(M^{2n})\bigcdot \big(\redu_1^{-1}(2n-1)\big)^{m_{\redu^{-1}(2n-1) }}
\bigcdot\big(12\cdots (2n-2)\cdot 2n\big)^{m_{12\cdots 2n-2\cdot 2n}}
\bigcdot\nonumber\\
&\bigcdot\displaystyle\circledcirc_{j=n-1}^{k+1}\big(\redu_1^{-1}({u_j})\big)^{m_{\redu^{-1}(u_j)}}
\bigcdot\big(\redu_1^{-1}({u_{k}})\big)^{m_{\redu^{-1}(u_k)}}\bigcdot\big(\redu_1^{-1}(u_k)\setminus\{2n\}\big)^{m_{\redu^{-1}(u_k)\setminus\{2n\} }}\nonumber\\
&\bigcdot
{\displaystyle\circledcirc_{j=k-1}^{1}\big(\redu_1^{-1}(u_j)\setminus\{2n\}\big)^{m_{\redu^{-1}(u_j)\setminus\{2n\} }}}
\bigcdot S^{H,\mu'}, \quad\mbox{for some $0\le k\le n$,}
\end{align}

or
\begin{align}\label{evendolphin}
&Y(M^{2n})\bigcdot \big(\redu_1^{-1}(2n-1))^{m_{\redu^{-1}(2n-1) }}\bigcdot
{\displaystyle\circledcirc_{j=n-1}^{1}\big(\redu_1^{-1}(u_j)\big)^{m_{\redu^{-1}(u_j) }}}
\bigcdot S^{H,\mu'}\nonumber\\
=&Y(M^{2n})\bigcdot \big((12\cdots (2n-3).(2n-2).(2n-1)\big)^{m_{12\cdots (2n-3).(2n-2).2n-1 }}\bigcdot
{\displaystyle\circledcirc_{j=n-1}^{1}\big(\redu_1^{-1}(u_j)\big)^{m_{\redu^{-1}(u_j) }}}
\bigcdot S^{H,\mu'},
\end{align}

or

\begin{align}\label{eventurtle}
&Y(M^{2n})\bigcdot \big(\redu_1^{-1}(2n-1)\big)^{m_{\redu^{-1}(2n-1) }}
\bigcdot\big(12\cdots (2n-2)\cdot 2n\big)^{m_{12\cdots 2n-2\cdot 2n}}
\bigcdot
\big(12\cdots (2n-2)\big)^{m_{(12\cdots (2n-2)) }}\bigcdot
\nonumber\\
&
\bigcdot
{\displaystyle\circledcirc_{j=n-1}^{1}\big(\redu_1^{-1}(u_j)\setminus\{2n\}\big)^{m_{\redu^{-1}(u_j)\setminus\{2n\} }}}
\bigcdot S^{H,\mu'}\nonumber\\
=&Y(M^{2n})
\bigcdot\big((12\cdots (2n-3).(2n-2).(2n-1)\big)^{m_{12\cdots (2n-3).(2n-2).2n-1 }}\bigcdot\big(12\cdots (2n-2)\cdot 2n\big)^{m_{12\cdots 2n-2\cdot 2n}}\nonumber\\
&\bigcdot
\big(12\cdots (2n-2)\big)^{m_{12\cdots (2n-2) }}\bigcdot
{\displaystyle\circledcirc_{j=n-1}^{1}\big(\redu_1^{-1}(u_j)\setminus\{2n\}\big)^{m_{\redu^{-1}(u_j)\setminus\{2n\} }}}
\bigcdot S^{H,\mu'}
\end{align}

where $\displaystyle \sum_{j=n}^{k+1} m_{\redu^{-1}(u_j)}+m_{\redu^{-1}(u_k)}+\sum_{j=k}^1 m_{\redu^{-1}(u_i)\setminus \{2n\}}=|\underline\delta_\r|-M$,
and
satisfying the restrictions on linear inequalities on the multiplicities of the columns

\begin{align} &\redu^{-1}(u_i),~~  \redu^{-1}(u_i)\setminus \{2n\}, \mbox{ $i=1,2,\dots,n$, } \mbox{ on LHS of $S^{H,\mu'},$}\nonumber\\
&\mbox{ and the symplectic columns  $u_1u_2\cdots u_i\in SpT_{2n}(\varpi_i)$, $i=1,\dots,n$,}\nonumber\\
&u_1u_2\cdots u_n,~ u_1u_2\cdots u_{n-1},\dots,u_1u_2,~~u_1\nonumber
\end{align} as follows

\begin{align*}
&m_{\redu^{-1}(u_i)}+m_{\redu^{-1}(u_i)\setminus \{2n\}}\le m_{u_1\cdots u_{i-1}}, ~i=2,\dots,n\\
 & m_{12\cdots (2n-2)2n}
 =\sum_{i=1}^{n-1} m_{\redu^{-1}(u_i)\setminus\{2n\}}.
 %
\end{align*}
\end{enumerate}
\end{thm}

\begin{thm}\label{thm:dominantneven} Let $n\in 2\Z$.  Assume the conditions of Theorem \ref{thm:n=evenslack1} where $1\le\ell (\mu)\le n$. Let $T\in SST^{\mathfrak{k}-hw}_{2n}(\lambda,\mu)$  be either the tableau  in \eqref{evenwhale} or in \eqref{evendolphin} or in \eqref{eventurtle}. Let ${\bf{C}_0}=(12\cdots n\,\overline n\cdots\overline 2 \overline 1)$. Then $\Phi^{-1}(T)\in SST_{2n}^{{\widehat{\mathfrak{g}}-dom}}(\lambda,\mu)$ is of the form
\begin{align}\label{generalnevendom}&C_0^M\bigcdot{(C_0\setminus\{\overline 1\})}^{m_{C_0\setminus\{\overline 1\}}}\bigcdot{(C_0\setminus\{\overline 2\})}^{m_{C_0\setminus\{\overline 2\}}}\bigcdot\dots\bigcdot {(C_0\setminus\{\overline{n}\})}^{m_{C_0\setminus\{\overline n\}}}\bigcdot {(C_0\setminus\{\overline{n},\overline{n-1}\})}^{m_{C_0\setminus\{\overline n,\overline{n-1}\}}}\bigcdot \nonumber\\
&\bigcdot(12\cdots n-1\overline{n-1\cdots 21})^{m_{12\cdots n-1\overline{n-1\cdots 21}}}\bigcdot(12\cdots n-1\overline{n-1})^{m_{12\cdots n-1\overline{n-1}}}\bigcdot\nonumber\\
&\bigcdot(12\cdots n-1)^{m_{12\cdots n-1}}\bigcdot\cdots \bigcdot(12)^{m_{12}}\bigcdot (1)^{m_1}
\end{align}
satisfying the linear inequalities on the multiplicity of the columns

\begin{align}\label{generalevenneq}
&m_{C_0\setminus\{\overline 2\}}\le m_{ 1},\\
&m_{C_0\setminus\{\overline 3\}}\le m_{ 12}\\
&\vdots\nonumber\\
&m_{C_0\setminus\{\overline {n-1}\}}\le m_{ 12\cdots n-2}\\
&m_{C_0\setminus\{\overline n\}}\le m_{ 12\cdots n-1}\\
&m_{C_0\setminus\{\overline{n},\overline{n-1}\}}\le m_{ 12\cdots  n-2\, \,n-1\,\overline {n-1}}
\end{align}
with $\mu=\widehat\wt(T)=(\mu_1,\mu_2,\dots,\mu_{n-1},\mu_n)$
where
\begin{align}&\mu_i=\sum_{k=i}^{n-1}m_{12\cdots k}+m_{12\cdots \scriptstyle n-1\,\scriptstyle\overline{ n-1}}
+m_{C_0\setminus \{\overline i\}},\quad 1\le i\le n-2\\
&\mu_{n-1}=m_{12\cdots n-1}+m_{C_0\setminus \{\overline {n-1}\}}+m_{C_0\setminus \{\overline{n},\overline{n-1}\}}\\
&\mu_{n}=m_{C_0\setminus \{\overline n\}}+m_{C_0\setminus \{\overline{n},\overline{n-1}\}}
\end{align}

\end{thm}
\begin{obs} The tableau in \eqref{generalnevendom} has shape length  $\ell(\lambda)=2n$ if $M>0\Leftrightarrow m_{\{1{ 2}\cdots n-1\,  n\,\overline n\, \overline {n-1}\cdots \overline 1\}}\ge 1$, otherwise $|\t|>0$ (because $\lambda\neq \mu$), has shape length $\ell(\lambda)=2n-1$  and, in this case, with leftmost column given by
$$C_0\setminus\{\overline i\}$$ for some $1\le i\le n$, that is,
$1\le m_{C_0\setminus\{\overline i\}}$, for some  $1\le i\le n$.
\end{obs}

\begin{ex} For $n=4$,  tableaux in Theorem \ref{thm:2},  $(1)$ and $(2)$, satisfy the conditions of this theorem.
\end{ex}

\subsection{Characterization of k-highest weight tableaux   via   1-0-slack recording tableaux of lower shape length  than 2n-1}\label{subsec:less2n}

\subsubsection{Case n odd} Consider the numbers defined in  \eqref{numbers:u}.

If $n\notin 2\Z$ and $n=2k-1$ with $k\in2\Z$ then $u_k=n$ and  the $\k$-highest symplectic  column in  $\SpT(\varpi_n)$ is
\begin{align}S^{H,\varpi_n}=(u_1,\dots,u_{k-1}=n-1, u_k=n,u_{k+1}=n+3,\dots,u_n=u_{k+k-1}=2n)\label{uk=n}\end{align}

If $n\notin 2\Z$ and $n=2k-1$ with $k\notin 2\Z$ then $u_k=n+1\in 2\Z$.
In this case,   the $\k$-highest symplectic  column in  $\SpT(\varpi_n)$ is
\begin{align}S^{H,\varpi_n}=(u_1,\dots,u_{k-1}=n-2, u_k=n+1,u_{k+1}=n+2,\dots,u_n=u_{k+k-1}=2n)\label{uk=n+1}
\end{align}

\begin{ex} For $ n=3=2\times 2-1$, $k=2$, and $ S^{H,\varpi_3}=(u_1=2, u_2=3, u_3=6)$. For $n=7=2\times 4-1$, $k=4$, and
$$S^{H,\varpi_7}=(u_1=2, u_2=3, u_3=6, u_4=7, u_5=10, u_6=13, u_7=14)$$

For $ n=1$, $ 1=2\times 1-1$, $k=1$, and $S^{H,\varpi_1}=(u_1=2)$. For $n=5$, $5=2\times 3-1$, $k=3$,  and
$$S^{H,\varpi_5}=(u_1=2,u_2=3, u_3=5+1=6,  u_4=7, u_5=10)$$
\end{ex}

Consider $n\notin 2\Z$ and the numbers in \eqref{uk=n}, $u_k=n\notin 2\Z$, $k\in2\Z$, $n+1\in 2\Z$. Let $l=n$ ($l-1=n-1\in 2\Z$) and recall \eqref{expexp}, for $i=1,\dots,k,\dots,n$, \begin{align}
\redu^{-1}_1:SpT_{2n}(u_i)&\rightarrow SST_{2n}(\varpi_{n})
\nonumber
\end{align}
where  $$\redu^{-1}_1(u_1)=(u_1=2,3,\dots,n,n+1)$$
$$\dots$$
$$\redu^{-1}_1(u_{k-1})=(12\cdots\,n-3,\, u_{k-1}= n-1,\,n,\,n+1)$$
$$\redu^{-1}_1(u_k)=(12\dots,n-1,u_k=n)$$
$$\redu^{-1}_1(u_{k+1})=(12\dots,n-1,u_{k+1}=n+3)$$
$$\dots$$
$$\redu^{-1}_1(u_{k+k-1})=(12\dots,n-1,u_{n}=2n)$$

\begin{thm}\label{th:uk=n}Let $ 3\le n\notin 2\Z$ and $S^{H,\varpi_n}$ as in \eqref{uk=n} where $u_k=n\notin 2\Z$, $k\in2\Z$. Consider $S^{H,\mu}$  the $\mathfrak{k}$-highest weight tableau
in $SpT_{2n}(\mu)$ as in \eqref{symphw-hw}. Let  $Q\in  Rec^{1,0}_{2n}(\lambda/\mu)$, $\mu\subsetneq \lambda$, and $\ell(\lambda)=n+1\in 2\Z$, with  $1$-$0$-slack sequence of the form $\underline \t=(1,\dots,1,0^M)$ and  corresponding  slack row index vector sequence   of the form
 $\underline \r=(1,\dots,1,2,\dots, 2,\dots,k,\dots, k,\dots,n,\dots,n,()^M)$.  Let  $\mu'=\mu-\delta_{\underline\r^+}$ .
Then, ${\LRAII^{AII}}^{-1}(S^{H,\mu},Q)$ returns   the $\mathfrak{k}$-highest  weight tableau in $SST^{\mathfrak{k}-hw}_{2n}(\lambda,\mu)$ of one of either  forms:
\begin{enumerate}
\item when $k+1,\dots,n\notin \underline \r$,

\begin{align}&(S^{H,\mu})^{1,\dots,1,\dots,k,\dots,k,()^M}=\nonumber\\
&=(12\cdots n\,n+1)^M\bigcdot(\redu^{-1}_1(u_k))^{m_{\redu^{-1}_1(u_k)}}\bigcdot
\circledcirc_{i=k-1}^{1}(\redu^{-1}_1(u_i))^{m_{\redu^{-1}_1(u_i)}}\bigcdot S^{H,\mu'}
\end{align}
where
$$m_{12\cdots n\,n+1}\ge 1$$
$$m_{\redu^{-1}_1(u_i)}\le m_{u_1\cdots u_{i-1}},\quad i=2,\dots,k.$$
\begin{ex} For $n=3$, $k=2$, see \eqref{typel=41} with $m_{124}=0$.
\end{ex}
\item when $n\in \underline \r$ and $k+1,\dots,n-1\notin \underline \r$ and $\ell(\mu)<n$
\begin{align}&(S^{H,\mu})^{1,\dots,1,\dots,k,\dots,k,n^\alpha,()^M}=\nonumber\\
&=(12\cdots n\,n+1)^M\bigcdot(\redu^{-1}_1(u_k))^{m_{\redu^{-1}_1(u_k)}}\bigcdot (12\cdots n-1\,n+1)^{\sum_{i=1}^{p}m_{\redu^{-1}_1(u_i)\setminus\{n+1\}}}
\bigcdot  \nonumber\\
&\bigcdot\circledcirc_{i=k-1}^{p}(\redu^{-1}_1(u_i))^{m_{\redu^{-1}_1(u_i)}}\nonumber \\
&\bigcdot\circledcirc_{i=p}^{1}(\redu^{-1}_1(u_i)\setminus\{n+1\})^{m_{\redu^{-1}_1(u_i)\setminus\{n+1\}}}\bigcdot
 S^{H,\mu'}, \quad \mbox{ for some $1\le p< k$}
 \end{align}

where
$$m_{12\cdots n\,n+1}\ge 1$$
$$m_{\redu^{-1}_1(u_i)}\le m_{u_1\cdots u_{i-1}},\quad i=p+1,\dots,k.$$
$$m_{\redu^{-1}_1(u_p)}+m_{\redu^{-1}_1(u_p)\setminus\{n+1\}}\le m_{u_1\cdots u_{p-1}}.$$
$$m_{\redu^{-1}_1(u_i)\setminus\{n+1\}}\le m_{u_1\cdots u_{i-1}}\quad i=1,\dots,p-1.$$
and
$$\alpha=\sum_{i=1}^{p}m_{\redu^{-1}_1(u_i)\setminus\{n+1\}}=m_{12\dots n-1\, n+1}=\begin{cases}\#\mbox{ columns of length $n-1$ },& \mbox{ if } n\ge 5\\
\le m_{23},& \mbox{ if } n=3\\
\end{cases}
$$

or
\begin{align}&(S^{H,\mu})^{1,\dots,1,\dots,k,\dots,k,n^\alpha,()^M}=\nonumber\\
&=(12\cdots n\,n+1)^M\bigcdot(\redu^{-1}_1(u_k))^{m_{\redu^{-1}_1(u_k)}}
\bigcdot (12\cdots n-1\,n+1)^{\sum_{i=1}^{p}m_{\redu^{-1}_1(u_i)\setminus\{n+1\}}}
\bigcdot  \nonumber\\
&\bigcdot (12\cdots n-1)^{m_{12\dots n-1}}\bigcdot 
\nonumber \\
&\bigcdot\circledcirc_{i=k-1}^{1}(\redu^{-1}_1(u_i)\setminus\{n+1\})^{m_{\redu^{-1}_1(u_i)\setminus\{n+1\}}}\bigcdot
 S^{H,\mu'},
 \end{align}
where
$$m_{12\cdots n\,n+1}\ge 1$$
$$m_{\redu^{-1}_1(u_k)}\le m_{u_1\cdots u_{k-1}}$$
$$m_{\redu^{-1}_1(u_i)\setminus\{n+1\}}\le m_{u_1\cdots u_{i-1}}\quad i=1,\dots,k-1.$$
and
$$\alpha=m_{12\dots n-1\, n+1}+m_{12\dots n-1}=
\begin{cases}\#\mbox{ columns of length $n-1$ },& \mbox{ if } n\ge 5\\
\le m_{23},& \mbox{ if } n=3\\
\end{cases}
$$

\begin{obs} For $n=3$, $k=2$, and $\ell(\mu)\le 2$, one has $p<2\Leftrightarrow p=1$. 
$$m_{1234}\ge 1$$
$$m_{123}\le m_2$$
$$m_{124}\le m_{23}$$
See \eqref{typel=41} with $p=1$ and $m_{236}=0$,  or \eqref{typel=43} with $m_{126}=0$.
\end{obs}

\item when some $k+1,\dots,n-1\in \underline \r$ and possibly $n\in\r$,   let $X$ be  a nonempty weakly increasing word  in the alphabet
$$\{u_{k+1}=n+3,\dots,u_n=u_{k+k-1}\}$$
and $I^X$ the weakly increasing word of the same length as X, in the alphabet $\{k+1,\dots,n-1, n\}$ defined by the indices of the letters $u_i$ in $X$ (taking into account their  multiplicities).
 \begin{itemize}

\item[(a)]  $n\notin\r$ and
 { $X=X_p\cdots X_1$ a factorization, possibly  some empty factors, with $X_i=x_1^{X_i}\cdots x_{|X_i|}^{X_i}$, $1\le i\le p$,
 for some $1\le p<k$,}

\begin{align}&(S^{H,\mu})^{1,\dots,1,\dots,k,\dots,k,I^X,()^M}=\nonumber\\
&(12\cdots n\,n+1)^M\bigcdot(\redu^{-1}_1(u_k))^{m_{\redu^{-1}_1(u_k)}}\bigcdot (12\cdots n-1\,n+1)^{\sum_{i=1}^{p}\sum_{j=k+1}^n m_{(\redu^{-1}_1(u_i)\setminus\{n+1\})\bigcdot u_j}}\nonumber\\
&\bigcdot \circledcirc_{i=k-1}^{p}(\redu^{-1}_1(u_i))^{m_{\redu^{-1}_1(u_i)}}\nonumber\\
&\bigcdot\circledcirc_{i=p}^{1}[\circledcirc_{j=1}^{|X_{i}|}(\redu^{-1}_1(u_i)\setminus\{n+1\})\bigcdot x_j^{X_{i}}]\bigcdot
 S^{H,\mu'},
\end{align}
where
$$m_{12\cdots n\,n+1}\ge 1$$
$$m_{\redu^{-1}_1(u_k)}\le m_{u_1\cdots u_{k-1}},$$
$$ m_{\redu^{-1}_1(u_i)}\le m_{u_1\cdots u_{i-1}},\quad i=k-1,\dots,p-1,$$
$$\sum_{j=k+1}^n m_{(\redu^{-1}_1(u_i)\setminus\{n+1\})\bigcdot u_j}
=|X_i|, \quad i=1,\dots,p$$
$$m_{\redu^{-1}_1(u_p)}+|X_p|\le m_{u_1\cdots u_{p-1}},$$
$$ |X_i|\le m_{u_1\cdots u_{i-1}},\quad i=2,\dots,p-1,$$
$$ \sum_{i=1}^{p} m_{(\redu^{-1}_1(u_i)\setminus\{n+1\})\bigcdot u_j}\le m_{u_1\cdots u_{j-1}},\,j=k+1,\dots,n$$

$$\sum_{i=1}^{p}\sum_{j=k+1}^n m_{(\redu^{-1}_1(u_i)\setminus\{n+1\})\bigcdot u_j}
=|X|$$

\item[(b)] $n\in \r$

\begin{itemize}

\item [(i)] $X=X_p\cdots X_q $ a factorization,  possibly  some empty factors, with $X_i=x_1^{X_i}\cdots x_{|X_i|}^{X_i}$, for some  $1< q\le i\le p<k$,


\begin{align}&(S^{H,\mu})^{1,\dots,1,\dots,k,\dots,k,I^X,(n)^\alpha,()^M}=\nonumber\\
   &=(12\cdots n\;n+1)^M\bigcdot(12\cdots n-1\;u_{k}=n)^{m_{\redu^{-1}_1(u_{k})}}\bigcdot \nonumber\\
   &\bigcdot(12\cdots n-1\,n+1)^{\sum_{i=q}^{p}\sum_{j=k+1}^n m_{(\redu^{-1}_1(u_i)\setminus\{n+1\})\bigcdot u_j}+\sum_{i=1}^{q}m_{(\redu^{-1}_1(u_i)\setminus\{n+1\})}}\bigcdot\nonumber\\
&\bigcdot \circledcirc_{i=k-1}^{p}(\redu^{-1}_1(u_i))^{m_{\redu^{-1}_1(u_i)}}\bigcdot\nonumber\\
&\bigcdot\circledcirc_{i=p}^{q}[\circledcirc_{j=1}^{|X_{i}|}(\redu^{-1}_1(u_i)\setminus\{n+1\})\bigcdot x_j^{X_{i}}] \bigcdot\nonumber\\
&\bigcdot\circledcirc_{i=q}^{1}(\redu^{-1}_1(u_i)\setminus\{n+1\})^{m_{\redu^{-1}_1(u_i)\setminus\{n+1\}}}\bigcdot
 S^{H,\mu'}, \nonumber
\end{align}
where
$$m_{12\cdots n\,n+1}\ge 1$$
$$ m_{(12\cdots n-1\,u_{k}=n)}\le m_{u_1\cdots u_{k-1}}$$
$$ m_{\redu^{-1}_1(u_i)}\le m_{u_1\cdots u_{i-1}},\quad i=p-1,\dots,k-1$$
$$ \sum_{j=k+1}^n m_{(\redu^{-1}_1(u_i)\setminus\{n+1\})\bigcdot u_j}
=|X_i|, \quad i=q,\dots,p$$
$$ m_{\redu^{-1}_1(u_p)}+|X_p|\le m_{u_1\cdots u_{p-1}},$$
$$
|X_i|\le m_{u_1\cdots u_{i-1}},\quad i=q+1,\dots,p-1,$$

$$m_{\redu^{-1}_1(u_q)\setminus\{n+1\}}+|X_q|\le m_{u_1\cdots u_{q-1}}$$
$$\redu^{-1}_1(u_i)\setminus\{n+1\}\le m_{u_1\cdots u_{i-1}},\quad i=2,\dots,q-1$$
$$ \sum_{i=q}^{p} m_{(\redu^{-1}_1(u_i)\setminus\{n+1\})\bigcdot u_j}\le m_{u_1\cdots u_{j-1}},\,j=k+1,\dots,n$$

$$\sum_{i=q}^{p}\sum_{j=k+1}^n m_{(\redu^{-1}_1(u_i)\setminus\{n+1\})\bigcdot u_j}
=|X|$$

and  $\alpha=\sum_{i=1}^{q}m_{(\redu^{-1}_1(u_i)\setminus\{n+1\})}$  the number of columns of length $n-1$.


\item[(ii)]{ $X=YX_{k-1}\cdots X_q$
 a factorization with $X_i=x_1^{X_i}\cdots x_{|X_i|}^{X_i}$, $q\le i\le k-1$,
 for some  $1\le q\le k-1$, and $Y$ a nonempty prefix increasing subword in the alphabet $\{u_{k+1},\dots, u_{k+t}\}$ for some $1\le t\le k-1$  such tat $u_{k+t}\le x_j^{X_i}$ for all $i,j$.}

\begin{align}&(S^{H,\mu})^{1,\dots,1,\dots,k,\dots k,I^X,(n)^\alpha,()^M}=\nonumber\\
&=(12\cdots n\,n+1)^M\bigcdot(\redu^{-1}_1(u_k))^{m_{\redu^{-1}_1(u_k)}}
\bigcdot  \nonumber\\
&(12\cdots n-1\,n+1)^{\sum_{i=q}^{k-1}\sum_{j=k+1}^n m_{(\redu^{-1}_1(u_i)\setminus\{n+1\})\bigcdot u_j}+\sum_{i=1}^{q}m_{\redu^{-1}_1(u_i)\setminus\{n+1\}}}\nonumber\\
&\bigcdot\circledcirc_{i=1}^{t}(12\cdots n-1\, u_{k+i})^{m_{12\cdots n-1\,u_{k+i}}}\nonumber\\
&\bigcdot\circledcirc_{i=k-1}^{q}[\circledcirc_{j=1}^{|X_{i}|}(\redu^{-1}_1(u_i)\setminus\{n+1\})\bigcdot x_j^{X_{i}}]\nonumber\\
&\bigcdot\circledcirc_{i=q}^{1}(\redu^{-1}_1(u_i)\setminus\{n+1\})^{m_{\redu^{-1}_1(u_i)\setminus\{n+1\}}}\bigcdot
 S^{H,\mu'},
 \nonumber
\end{align}
where
$$m_{12\cdots n\,n+1}\ge 1$$
$$m_{\redu^{-1}_1(u_k)}\le m_{u_1\cdots u_{k-1}},$$
$$m_{\redu^{-1}_1(u_i)\setminus\{n+1\}}\le m_{u_1\cdots u_{i-1}},\quad i=2,\dots,q-1,$$
$$  \sum_{j=k+t}^n m_{(\redu^{-1}_1(u_i)\setminus\{n+1\})\bigcdot u_j}=|X_i|,\quad i=k-1,\dots,q,$$
$$ m_{\redu^{-1}_1(u_q)\setminus\{n+1\}}+|X_q|\le m_{u_1\cdots u_{q-1}},$$
$$ |X_i|\le m_{u_1\cdots u_{i-1}},\quad i=k-1,\dots,q-1,$$
$$\sum_{i=q}^{k-1}m_{(\redu^{-1}_1(u_i)\setminus\{n+2\})\bigcdot u_{j}}\le m_{u_1\cdots u_{j-1}},\quad j=k+t+1,\dots,n,$$
$${m_{12\cdots n-1\,u_{k+i}}}\le {m_{u_1\cdots \,u_{k+i-1}}},\quad i=1,\dots,t-1$$
$${m_{12\cdots n-1\,u_{k+t}}}+\sum_{i=q}^{k-1}m_{(\redu^{-1}_1(u_i)\setminus\{n+1\})\bigcdot u_{k+t}}\le m_{u_1\cdots u_{k+t-1}}$$
\medskip
and $\alpha=\sum_{i=1}^{q}m_{(\redu^{-1}_1(u_i)\setminus\{n+1\})}$ the number of columns of length $n-1$.


\item[(iii)]   $X=Y$
\begin{align}&(S^{H,\mu})^{1,\dots,1,\dots,k,\dots,k,I^X,(n)^\alpha,()^M}=\nonumber\\
&(12\cdots n\,n+1)^M\bigcdot({\redu^{-1}_1(u_k)})^{m_{\redu^{-1}_1(u_k)}}\nonumber\\
&\bigcdot (12\cdots n-1\,n+1)^{\sum_{i=1}^{k-1}m_{\redu^{-1}_1(u_i)\setminus\{n+1\}}}
\bigcdot \nonumber\\
&\bigcdot\circledcirc_{i=1}^{k-1}(12\cdots n-1\, u_{k+i})^{m_{12\cdots n-1\,u_{k+i}}}
\bigcdot (12\cdots n-1)^{m_{12\cdots n-1}}\nonumber\\
&\bigcdot\circledcirc_{i=k-1}^{1}(\redu^{-1}_1(u_i)\setminus\{n+1\})^{m_{\redu^{-1}_1(u_i)\setminus\{n+1\}}}\bigcdot\nonumber\\
&\bigcdot S^{H,\mu'}
\end{align}

where
$$m_{12\cdots n\,n+1}\ge 1$$
$$m_{\redu^{-1}_1(u_k)}\le m_{u_1\cdots u_{k-1}},$$
$$m_{\redu^{-1}_1(u_i)\setminus\{n+1\}}\le m_{u_1\cdots u_{i-1}},\quad i=2,\dots,k-1,$$
$$m_{12\cdots n-1\,u_{k+i}}\le m_{u_1\cdots u_{k+i-1}},\quad i=1,\dots,k-1.$$
and $\alpha=\sum_{i=1}^{q}m_{(\redu^{-1}_1(u_i)\setminus\{n+1\})}+m_{12\cdots n-1}$ the number of columns of length $n-1$.
\end{itemize}
\end{itemize}
\end{enumerate}
\quad \qquad\qquad\qquad\qquad\qquad\qquad\qquad\qquad\qquad\qquad\qquad\qquad\qquad\qquad\qquad\qquad \qquad\qquad\qquad\qquad\qquad\qquad$\square$
\end{thm}

\bigskip
Consider now $n\notin 2\Z$ and the numbers in \eqref{uk=n+1}, $u_k=n+1\in 2\Z$, $k\notin 2\Z$. Let $l=n\notin 2\Z$ ($l-1\in 2\Z$) and for $i=1,\dots,k,\dots,n$, \begin{align}
\redu^{-1}_1:SpT_{2n}(u_i)&\rightarrow SST_{2n}(\varpi_{n})
\nonumber
\end{align}
where  
$$\redu^{-1}_1(u_1)=(u_1=2,3,\dots,n,n+1)$$
$$\dots$$
$$\redu^{-1}_1(u_{k-1})=(12\cdots,\,n-3,\, u_{k-1}= n-2,\,n,\,n+1)$$
 $$\redu^{-1}_1(u_k)=(12\dots,n-1,u_k=n+1)$$
 $$\redu^{-1}_1(u_{k+1})=(12\dots,n-1,u_{k+1}=n+2)$$
 $$\dots$$
 $$\redu^{-1}_1(u_{n})=(12\dots,n-1,u_{n}=2n)$$

\begin{obs} The case $n=1$ has been considered in \cite{torres} and \cite[Corollary 5]{azreco}.
\end{obs}
\begin{thm}\label{th:uk=n+1}Let $5\le n\notin 2\Z$ and $S^{H,\varpi_n}$ as in \eqref{uk=n+1} where $u_k=n+1\in 2\Z$, $k\notin 2\Z$. Consider $S^{H,\mu}$  the $\mathfrak{k}$-highest weight tableau
in $SpT_{2n}(\mu)$ as in \eqref{symphw-hw}. Let  $Q\in  Rec^{1,0}_{2n}(\lambda/\mu)$, $\mu\subsetneq \lambda$, and $\ell(\lambda)=n+1\in 2\Z$, with  $1$-$0$-slack sequence of the form $\underline \t=(1,\dots,1,0^M)$ and  corresponding  slack row index vector sequence   of the form
 $\underline \r=(1,\dots,1,2,\dots, 2,\dots,k,\dots, k,\dots,n,\dots,n,()^M)$.  Let  $\mu'=\mu-\delta_{\underline\r^+}$ .
Then, ${\LRAII^{AII}}^{-1}(S^{H,\mu},Q)$ returns   the $\mathfrak{k}$-highest  weight tableau in $SST^{\mathfrak{k}-hw}_{2n}(\lambda,\mu)$ of one of either  forms:
\begin{enumerate}
\item when $k+1,\dots,n\notin \underline \r$,
\begin{align}&(S^{H,\mu})^{1,\dots,1,\dots,k,\dots,k,()^M}=\nonumber\\
&=(12\cdots n\,n+1)^M\bigcdot(12\dots\,n-1\,n+1)^{m_{\redu^{-1}_1(u_k)}}\bigcdot
\circledcirc_{i=k-1}^{1}(\redu^{-1}_1(u_i))^{m_{\redu^{-1}_1(u_i)}}\bigcdot S^{H,\mu'}
\end{align}
where
$$m_{\redu^{-1}_1(u_i)}\le m_{u_1\cdots u_{i-1}},\quad i=2,\dots,k.$$

\item when $n\in \underline \r$ and $k+1,\dots,n-1\notin \underline \r$ and $\ell(\mu)<n$

\begin{align}&(S^{H,\mu})^{1,\dots,1,\dots,k,\dots,k,n^\alpha,()^M}=\nonumber\\
&=(12\cdots n\,n+1)^M\bigcdot(12\dots\,n-1\,n+1)^{m_{\redu^{-1}_1(u_k)}}\bigcdot (12\cdots n-1\,n+1)^{\sum_{i=1}^{p}m_{\redu^{-1}_1(u_i)\setminus\{n+1\}}}
\bigcdot  \nonumber\\
&\bigcdot\circledcirc_{i=k-1}^{p}(\redu^{-1}_1(u_i))^{m_{\redu^{-1}_1(u_i)}} \bigcdot\circledcirc_{i=p}^{1}(\redu^{-1}_1(u_i)\setminus\{n+1\})^{m_{\redu^{-1}_1(u_i)\setminus\{n+1\}}}\bigcdot
 S^{H,\mu'}, \mbox{ for some $p<k$}
 \end{align}

where
$$(12\dots\,n-1\,n+1)^{m_{1\,2\,\dots\,n-1\,n+1}}-\sum_{i=1}^{p}m_{\redu^{-1}_1(u_i)\setminus\{n+1\}}\le m_{u_1\cdots u_{k-1}}$$
$$m_{\redu^{-1}_1(u_i)}\le m_{u_1\cdots u_{i-1}},\quad i=p+1,\dots,k-1.$$
$$m_{\redu^{-1}_1(u_p)}+m_{\redu^{-1}_1(u_p)\setminus\{n+1\}}\le m_{u_1\cdots u_{p-1}}.$$
$$m_{\redu^{-1}_1(u_i)\setminus\{n+1\}}\le m_{u_1\cdots u_{i-1}}\quad i=1,\dots,p-1.$$
and $\alpha=\sum_{i=1}^{p}m_{\redu^{-1}_1(u_i)\setminus\{n+1\}}$ the number of columns of length $n-1$.

\item { when some $k+1,\dots,n-1\in \underline \r$ and possibly $n\in\r$,   let $X$ be  a nonempty weakly increasing word  in the alphabet
$$\{u_{k+1}=n+2,\dots,u_n=u_{k+k-1}\}$$
and $I^X$ the weakly increasing word of the same length as X, in the alphabet $\{k+1,\dots,n\}$ defined by the indices of the letters $u_i$ in $X$ (taking into account their  multiplicities).}

\begin{itemize}
\item [(a)]  $n\notin\r$ and
  $X=X_p\cdots X_1$ a factorization, possibly  some empty factors, with $X_i=x_1^{X_i}\cdots x_{|X_i|}^{X_i}$, $1\le i\le p$,
 for some $1\le p<k$,

\begin{align}& (S^{H,\mu})^{1,\dots,1,\dots,k,\dots,k,I^X,()^M}=\nonumber\\
&=(12\cdots n\,n+1)^M\bigcdot(12\cdots n-1\,n+1)^{m_{\redu^{-1}_1(u_k)}}\bigcdot \nonumber\\
&\bigcdot (12\cdots n-1\,n+1)^{\sum_{i=1}^{p}\sum_{j=k+1}^n m_{(\redu^{-1}_1(u_i)\setminus\{n+1\})\bigcdot u_j}}\nonumber\\
&\bigcdot \circledcirc_{i=k-1}^{p}(\redu^{-1}_1(u_i))^{m_{\redu^{-1}_1(u_i)}}\nonumber\\
&\bigcdot\circledcirc_{i=p}^{1}[\circledcirc_{j=1}^{|X_{i}|}(\redu^{-1}_1(u_i)\setminus\{n+1\})\bigcdot x_j^{X_{i}}]\bigcdot
 S^{H,\mu'}, \nonumber\\
\end{align}

where
$$ m_{(12\cdots n-1\,n+1)}-\sum_{i=1}^{p}\sum_{j=k+1}^n m_{(\redu^{-1}_1(u_i)\setminus\{n+1\})\bigcdot u_j}\le m_{u_1\cdots u_{k-1}},$$
$$ m_{\redu^{-1}_1(u_i)}\le m_{u_1\cdots u_{i-1}},\quad i=k-1,\dots,p-1,$$
$$\sum_{j=k+1}^n m_{(\redu^{-1}_1(u_i)\setminus\{n+1\})\bigcdot u_j}
=|X_i|, \quad i=1,\dots,p$$
$$ m_{\redu^{-1}_1(u_p)}+|X_p|\le m_{u_1\cdots u_{p-1}},$$
$$ |X_i|\le m_{u_1\cdots u_{i-1}},\quad i=2,\dots,p-1,$$
$$ \sum_{i=1}^{p} m_{(\redu^{-1}_1(u_i)\setminus\{n+1\})\bigcdot u_j}\le m_{u_1\cdots u_{j-1}},\,j=k+1,\dots,n$$

$$ \sum_{i=1}^{p}\sum_{j=k+1}^n m_{(\redu^{-1}_1(u_i)\setminus\{n+1\})\bigcdot u_j}
=|X|$$

\item[(b)]  $n\in\r$
\begin{itemize}
\item[(i)]$X=X_p\cdots X_q $ a factorization,  possibly  some empty factors, with $X_i=x_1^{X_i}\cdots x_{|X_i|}^{X_i}$, for some  $1< q\le i\le p<k$,

\begin{align}&(S^{H,\mu})^{1,\dots,1,\dots,k,\dots,k,I^X,(n)^\alpha,()^M}=\nonumber\\
   &=(12\cdots n\;n+1)^M\bigcdot(12\cdots n-1\;u_{k}=n+1)^{m_{\redu^{-1}_1(u_{k})}}\bigcdot \nonumber\\
   &\bigcdot(12\cdots n-1\,n+1)^{\sum_{i=q}^{p}\sum_{j=k+1}^n m_{(\redu^{-1}_1(u_i)\setminus\{n+1\})\bigcdot u_j}+\sum_{i=1}^{q}m_{(\redu^{-1}_1(u_i)\setminus\{n+1\})}}\bigcdot\nonumber\\
&\bigcdot \circledcirc_{i=k-1}^{p}(\redu^{-1}_1(u_i))^{m_{\redu^{-1}_1(u_i)}}\bigcdot\nonumber\\
&\bigcdot\circledcirc_{i=p}^{q}[\circledcirc_{j=1}^{|X_{i}|}(\redu^{-1}_1(u_i)\setminus\{n+1\})\bigcdot x_j^{X_{i}}] \bigcdot\nonumber\\
&\bigcdot\circledcirc_{i=q}^{1}(\redu^{-1}_1(u_i)\setminus\{n+1\})^{m_{\redu^{-1}_1(u_i)\setminus\{n+1\}}}\bigcdot
 S^{H,\mu'}, \nonumber
\end{align}
where
$$ m_{(12\cdots n-1\,n+1)}-\sum_{i=1}^{q}m_{\redu^{-1}_1(u_i)\setminus\{n+1\}}-
\sum_{i=p}^{k-1}\sum_{j=k+1}^n m_{(\redu^{-1}_1(u_i)\setminus\{n+1\})\bigcdot u_j}\le m_{u_1\cdots u_{k-1}},$$
$$ m_{\redu^{-1}_1(u_i)}\le m_{u_1\cdots u_{i-1}},\quad i=p-1,\dots,k-1$$
$$ \sum_{j=k+1}^n m_{(\redu^{-1}_1(u_i)\setminus\{n+1\})\bigcdot u_j}
=|X_i|, \quad i=q,\dots,p$$
$$ m_{\redu^{-1}_1(u_p)}+|X_p|\le m_{u_1\cdots u_{p-1}},$$
$$
|X_i|\le m_{u_1\cdots u_{i-1}},\quad i=q+1,\dots,p-1,$$

$$m_{\redu^{-1}_1(u_q)\setminus\{n+1\}}+|X_q|\le m_{u_1\cdots u_{q-1}}$$
$$\redu^{-1}_1(u_i)\setminus\{n+1\}\le m_{u_1\cdots u_{i-1}},\quad i=2,\dots,q-1$$
$$ \sum_{i=q}^{p} m_{(\redu^{-1}_1(u_i)\setminus\{n+1\})\bigcdot u_j}\le m_{u_1\cdots u_{j-1}},\,j=k+1,\dots,n$$

$$\sum_{i=q}^{p}\sum_{j=k+1}^n m_{(\redu^{-1}_1(u_i)\setminus\{n+1\})\bigcdot u_j}
=|X|$$

and  $\alpha=\sum_{i=1}^{q}m_{(\redu^{-1}_1(u_i)\setminus\{n+1\})}$  the number of columns of length $n-1$.

\item[(ii)]
$X=Y X_{k-1}\cdots X_p$ a factorization with $X_i=x_1^{X_i}\cdots x_{|X_i|}^{X_i}$, $1\le i\le p$,
 for some $1\le p<k$,
 {and $Y$ a nonempty prefix increasing subword in the alphabet $\{u_{k+1},\dots, u_{k+t}\}$ for some $1\le t\le k-1$  such tat $u_{k+t}\le x_j^{X_i}$ for all $i,j$.}

\begin{align}&(S^{H,\mu})^{1,\dots,1,\dots,k,\dots,k,I^X,(n)^\alpha,()^M}=\nonumber\\
&(12\cdots n\,n+1)^M\bigcdot (12\cdots n-1\,n+1)^{m_{\redu^{-1}_1(u_k)}}\bigcdot\nonumber\\
& \bigcdot(12\cdots n-1\,n+1)^{\sum_{i=1}^{p}m_{\redu^{-1}_1(u_i)\setminus\{n+1\}}}
\bigcdot  \nonumber\\
&\bigcdot(12\cdots n-1\,n+1)^{\sum_{i=p}^{k-1}\sum_{j=k+1}^n m_{(\redu^{-1}_1(u_i)\setminus\{n+1\})\bigcdot u_j}}\nonumber\\
&
\bigcdot\circledcirc_{i=1}^{t}(12\cdots n-1\, u_{k+i})^{m_{12\cdots n-1\,u_{k+i}}}\nonumber\\
&\bigcdot\circledcirc_{i=k-1}^{p}[\circledcirc_{j=1}^{|X_{i}|}(\redu^{-1}_1(u_i)\setminus\{n+1\})\bigcdot x_j^{X_{i}}]\bigcdot\nonumber\\
&\bigcdot\circledcirc_{i=p}^{1}(\redu^{-1}_1(u_i)\setminus\{n+1\})^{m_{\redu^{-1}_1(u_i)\setminus\{n+1\}}}\bigcdot
 S^{H,\mu'},
 \nonumber
\end{align}

where

$$  m_{(12\cdots n-1\,n+1)}-\sum_{i=1}^{p}m_{\redu^{-1}_1(u_i)\setminus\{n+1\}}-
\sum_{i=p}^{k-1}\sum_{j=k+1}^n m_{(\redu^{-1}_1(u_i)\setminus\{n+1\})\bigcdot u_j}\le m_{u_1\cdots u_{k-1}},$$
$$m_{\redu^{-1}_1(u_i)\setminus\{n+1\}}\le m_{u_1\cdots u_{i-1}},\quad i=2,\dots,p-1,$$
$$  \sum_{j=k+t}^n m_{(\redu^{-1}_1(u_i)\setminus\{n+1\})\bigcdot u_j}=|X_i|,\quad i=k-1,\dots,p,$$
$$ m_{\redu^{-1}_1(u_p)\setminus\{n+1\}}+|X_p|\le m_{u_1\cdots u_{p-1}},$$
$$ |X_i|\le m_{u_1\cdots u_{i-1}},\quad i=k-1,\dots,p-1,$$
$$ \sum_{i=p}^{k-1}m_{(\redu^{-1}_1(u_i)\setminus\{n+2\})\bigcdot u_{j}}\le m_{u_1\cdots u_{j-1}},\quad j=k+t+1,\dots,n,$$
$${m_{12\cdots n-1\,u_{k+i}}}\le {m_{u_1\cdots \,u_{k+i-1}}},\quad i=1,\dots,t-1$$
$${m_{12\cdots n-1\,u_{k+t}}}+\sum_{i=p}^{k-1}m_{(\redu^{-1}_1(u_i)\setminus\{n+1\})\bigcdot u_{k+t}}\le m_{u_1\cdots u_{k+t-1}}$$
and $\alpha=\sum_{i=1}^{q}m_{(\redu^{-1}_1(u_i)\setminus\{n+1\})}$ the number of columns of length $n-1$.

\item[(iii)]   $X=Y$
\begin{align}&(S^{H,\mu})^{1,\dots,1,\dots,k,\dots,k,I^X,(n)^\alpha,()^M}=\nonumber\\
&=(12\cdots n\,n+1)^M\bigcdot(12\dots,n-1,n+1)^{m_{\redu^{-1}_1(u_k)}}\bigcdot\nonumber\\
&\bigcdot (12\cdots n-1\,n+1)^{\sum_{i=1}^{k-1}m_{\redu^{-1}_1(u_i)\setminus\{n+1\}}}
\bigcdot \nonumber\\
&\bigcdot\circledcirc_{i=1}^{k-1}(12\cdots n-1\, u_{k+i})^{m_{12\cdots n-1\,u_{k+i}}}
\bigcdot (12\cdots n-1)^{m_{12\cdots n-1}}\bigcdot\nonumber\\
&\bigcdot\circledcirc_{i=k-1}^{1}(\redu^{-1}_1(u_i)\setminus\{n+1\})^{m_{\redu^{-1}_1(u_i)\setminus\{n+1\}}}\bigcdot\nonumber\\
&\bigcdot S^{H,\mu'}
\end{align}

where

$$  m_{(12\cdots n-1\,n+1)}-\sum_{i=1}^{k-1}m_{\redu^{-1}_1(u_i)\setminus\{n+1\}}\le m_{u_1\cdots u_{k-1}},$$
$$m_{\redu^{-1}_1(u_i)\setminus\{n+1\}}\le m_{u_1\cdots u_{i-1}},\quad i=2,\dots,k-1,$$
$$m_{12\cdots n-1\,u_{k+i}}\le m_{u_1\cdots u_{k+i-1}},\quad i=1,\dots,k-1.\quad \square$$
and $\alpha=\sum_{i=1}^{q}m_{(\redu^{-1}_1(u_i)\setminus\{n+1\})}+m_{12\cdots n-1}$ the number of columns of length $n-1$.

\end{itemize}
\end{itemize}
\end{enumerate}
\end{thm}
\subsubsection{Case $n$ even}Consider again the numbers defined in  \eqref{numbers:u}.

If $n\in 2\Z$ and $n=2k$ with $k\in 2\Z$ then $u_k=n-1\notin 2\Z$ and  the $\k$-highest weight symplectic  column in  $\SpT(\varpi_n)$ is
\begin{align}S^{H,\varpi_n}=(u_1,\dots,u_{k-1}=n-2, u_k=n-1,u_{k+1}=n+2,\dots,u_n=u_{k+k}=2n-1)\label{uk=oddkeven}\end{align}

If $n\in 2\Z$ and $n=2k$ with $k\notin 2\Z$ then $u_k=n\in 2\Z$ and  the $\k$-highest weight symplectic  column in  $\SpT(\varpi_n)$ is

\begin{align}S^{H,\varpi_n}=(u_1,\dots,u_{k-1}=n-3, u_k=n,u_{k+1}=n+1,\dots,u_n=u_{k+k}=2n-1)\label{uk=nevenkodd}\end{align}

\begin{ex} For $n=4=2\times 2$, $k=2$,
$$S^{H,\varpi_4}=(u_1=2, u_2=3, u_3=6,u_4=7),$$
and for $n=2$, $n=2\times 1$, $\oo k=1$, and for $n=6=2\times 3$, $k=3$,
$$S^{H,\varpi_6}=(u_1=2, u_2=3, u_3=6, u_4=7, u_5=10, u_6=11)$$
\end{ex}

Consider $n\in 2\Z$ and the numbers in \eqref{uk=oddkeven}, where $k\in 2\Z$ and $u_k=n-1\notin 2\Z$. Let  $l=n+1$ ($l-1\in 2\Z$) and recall the definition of the inverse reduction map \eqref{expexp}. For $i=1,\dots,k,\dots,n$, let \begin{align}
\redu^{-1}_1:SpT_{2n}(u_i)&\rightarrow SST_{2n}(\varpi_{n+1})
\nonumber
\end{align}
 for $l=n+1\notin 2\Z$ and $t=1$.

where $$\redu^{-1}_1(u_1)=(u_1=2,3,\dots,n+1,n+2)$$
$$\dots$$
$$\redu^{-1}_1(u_{k-1})=(12\cdots\,n-4,\, u_{k-1}= n-2,\,n-1,n,\,n+1,n+2)$$
 $$\redu^{-1}_1(u_k)=(12\dots,n-2,u_k=n-1,n+1,n+2)$$
$$\redu^{-1}_1(u_{k+1})=(12\dots,n-2,n-1,n,u_{k+1}=n+2)$$
$$\redu^{-1}_1(u_{k+2})=(12\dots,n-2,n-1,n,u_{k+2}=n+3)$$
$$\dots$$
$$\redu^{-1}_1(u_{k+k})=(12\dots,n-2,n-1,n,u_{k+k}=2n-1)$$

\begin{thm} \label{thm:nkeven}Let $4\le n\in 2\Z$ and $S^{H,\varpi_n}$ as in \eqref{uk=oddkeven} where $u_k=n-1\notin 2\Z$, $u_{k+1}=n+2$,  and $k\in 2\Z$. Consider $S^{H,\mu}$  the $\mathfrak{k}$-highest weight tableau
in $SpT_{2n}(\mu)$ as in \eqref{symphw-hw}. Let  $Q\in  Rec^{1,0}_{2n}(\lambda/\mu)$, $\mu\subsetneq \lambda$, and $\ell(\lambda)=n+2\in 2\Z$, with  $1$-$0$-slack sequence of the form $\underline \t=(1,\dots,1,0^M)$ and  corresponding  slack row index vector sequence   of the form
 $$\underline \r=(1,\dots,1,2,\dots, 2,\dots,k+1,\dots, k+1,\dots,n,\dots,n,n+1,\dots,n+1,()^M).$$  Let  $\mu'=\mu-\delta_{\underline\r^+}$ .
Then, ${\LRAII^{AII}}^{-1}(S^{H,\mu},Q)$ returns   the $\mathfrak{k}$-highest  weight tableau in $SST^{\mathfrak{k}-hw}_{2n}(\lambda,\mu)$ of one of either  forms:
\begin{enumerate}
\item when $k+2,\dots,n,n+1\notin \underline \r$,

\begin{align}&(S^{H,\mu})^{1,\dots,1,\dots,k+1,\dots,k+1,()^M}=\nonumber\\
&=(12\cdots n\,n+1,n+2)^M\bigcdot(12\dots,n-1,n,u_{k+1}=n+2)^{m_{\redu^{-1}_1(u_{k+1})}}\bigcdot\nonumber\\
&\circledcirc_{i=k}^{1}(\redu^{-1}_1(u_i))^{m_{\redu^{-1}_1(u_i)}}\bigcdot S^{H,\mu'}
\end{align}
where
$$m_{\redu^{-1}_1(u_i)}\le m_{u_1\cdots u_{i-1}},\quad i=2,\dots,k+1.$$


\item  when $n+1\in \underline \r$ and $k+2,\dots,n\notin \underline \r$ ($\Leftrightarrow $ some $1,\dots,k+1\in\r$),
\begin{itemize}
\item[(a)]
\begin{align}&(S^{H,\mu})^{1,\dots,1,\dots,k+1,\dots,k+1,(n+1)^\alpha,()^M}=\nonumber\\
&(12\cdots n\,n+1,n+2)^M\bigcdot(12\dots,n-1,n,u_{k+1}=n+2)^{m_{\redu^{-1}_1(u_{k+1})}}\bigcdot\nonumber\\
&\bigcdot(12\dots,n-1,n,n+2)^{\sum_{i=1}^p{m_{\redu^{-1}_1(u_i)\setminus\{n+2\}}}}\bigcdot\nonumber\\
&\bigcdot\circledcirc_{i=k}^{p}(\redu^{-1}_1(u_i))^{m_{\redu^{-1}_1(u_i)}}\bigcdot \nonumber\\
&\bigcdot\circledcirc_{i=p}^{1}(\redu^{-1}_1(u_i)\setminus\{n+2\})^{m_{\redu^{-1}_1(u_i)\setminus\{n+2\}}}\bigcdot
S^{H,\mu'},  \mbox{ for some $1\le p\le k$},
\end{align}
where
$$ m_{(12\cdots n\,n+2)}-\sum_{i=1}^p{m_{\redu^{-1}_1(u_i)\setminus\{n+2\}}}\le m_{u_1\cdots u_k}$$
$$m_{\redu^{-1}_1(u_i)\setminus\{n+2\}}\le m_{u_1\cdots u_{i-1}},\quad i=2,\dots,p.$$
$$m_{\redu^{-1}_1(u_i)}\le m_{u_1\cdots u_{i-1}},\quad i=p,\dots,k.$$
$$m_{\redu^{-1}_1(u_p)}+m_{\redu^{-1}_1(u_p)\setminus\{n+2\}}\le m_{u_1\cdots u_{p-1}},$$
and  $\alpha= \sum_{i=1}^p{m_{\redu^{-1}_1(u_i)\setminus\{n+2\}}}$ the number of columns of length $n$.

\item [(b)]
\begin{align}&(S^{H,\mu})^{1,\dots,1,\dots,k+1,\dots,k+1,(n+1)^\alpha,()^M}=\nonumber\\
&(12\cdots n\,n+1,n+2)^M\bigcdot(12\dots,n-1,n,u_{k+1}=n+2)^{m_{\redu^{-1}_1(u_{k+1})}}\bigcdot\nonumber\\
&\bigcdot(12\dots,n-1,n,n+2)^{\sum_{i=1}^k{m_{\redu^{-1}_1(u_i)\setminus\{n+2\}}}}\bigcdot (12\cdots n)^{m_{12\cdots n}}\bigcdot\nonumber\\
&\bigcdot\circledcirc_{i=k}^{1}(\redu^{-1}_1(u_i)\setminus\{n+2\})^{m_{\redu^{-1}_1(u_i)\setminus\{n+2\}}}\bigcdot
S^{H,\mu'}, 
\end{align}
where
$$ m_{(12\cdots n\,n+2)}-\sum_{i=1}^p{m_{\redu^{-1}_1(u_i)\setminus\{n+2\}}}\le m_{u_1\cdots u_k}$$
$$m_{\redu^{-1}_1(u_i)\setminus\{n+2\}}\le m_{u_1\cdots u_{i-1}},\quad i=2,\dots,k.$$
\end{itemize}
and $\alpha= \sum_{i=1}^k{m_{\redu^{-1}_1(u_i)\setminus\{n+2\}}}+m_{12\cdots n}$ the number of columns of length $n$.


\item when  $1,\dots,k+1\notin \underline \r$, some $k+ 2,\dots,n\in \underline \r$ and possibly  $n+1\in\r$ 


\begin{align}&(S^{H,\mu})^{k+2,\dots,k+2,\dots,n,(n+1)^\alpha,()^M}=\nonumber\\
&=(12\cdots n\,n+1,n+2)^M\bigcdot \circledcirc_{i=k+2}^{n}(12\dots n\, u_i)^{m_{\redu^{-1}_1(u_i)}}\bigcdot (12\dots n)^{m_{12\dots n}}
\bigcdot
 S^{H,\mu'}, \nonumber
 \end{align}

where $$m_{(12\dots n\, u_i)}\le m_{u_1\cdots u_{i-1}},\quad i=k+2,\dots,n$$
and $\alpha =m_{12\dots n}$ the number of columns of length $n$.

\begin{ex}For $n=4$, $k=2$,  $u_4=7$, \eqref{n=4type4160x}
 \end{ex}

\item \label{wordX2} when some $1,\dots,k+1\in\r$ and $k+2,\dots,n\in \r$, let $X$  be a nonempty weakly increasing word  in the alphabet
$$\{u_{k+2}=n+3,\dots,u_n=u_{k+k}=2n-1\}$$
and $I^X$ the weakly increasing word of the same length as X, in the alphabet $\{k+2,\dots,n\}$ defined by the indices of the letters $u_i$ in $X$ (taking into account their  multiplicities).
\begin{itemize}
   \item[(a)] $n+1\notin\r$ and $X=X_p\cdots X_1$ a factorization, possibly  some empty factors, with $X_i=x_1^{X_i}\cdots x_{|X_i|}^{X_i}$, $1\le i\le p$,
 for some $1\le p<k+1$,

   \begin{align}&(S^{H,\mu})^{1,\dots,1,\dots,n,\dots,n,I^X,()^M}=\nonumber\\
   &(12\cdots n\;n+1\, n+2)^M\bigcdot(12\cdots n\;u_{k+1}=n+2)^{m_{\redu^{-1}_1(u_{k+1})}}\bigcdot \nonumber\\
   &\bigcdot (12\cdots n\,n+2)^{\sum_{i=1}^{p}\sum_{j=k+2}^n m_{(\redu^{-1}_1(u_i)\setminus\{n+2\})\bigcdot u_j}}\nonumber\\
&\bigcdot \circledcirc_{i=k}^{p}(\redu^{-1}_1(u_i))^{m_{\redu^{-1}_1(u_i)}}\bigcdot\nonumber\\
&\bigcdot\circledcirc_{i=p}^{1}[\circledcirc_{j=1}^{|X_{i}|}(\redu^{-1}_1(u_i)\setminus\{n+2\})\bigcdot x_j^{X_{i}}]\bigcdot
 S^{H,\mu'}, \nonumber
\end{align}

  \bigskip
 where
$$ m_{(12\cdots n\,n+2)}-\sum_{i=1}^{p}\sum_{j=k+2}^n m_{(\redu^{-1}_1(u_i)\setminus\{n+2\})\bigcdot u_j}\le m_{u_1\cdots u_{k}},$$
$$ m_{\redu^{-1}_1(u_i)}\le m_{u_1\cdots u_{i-1}},\quad i=k,\dots,p-1,$$
$$ m_{\redu^{-1}_1(u_p)}+|X_p|\le m_{u_1\cdots u_{p-1}},$$
$$ |X_i|\le m_{u_1\cdots u_{i}},\quad i=2,\dots,p-1,$$
$$ \sum_{i=1}^{p} m_{(\redu^{-1}_1(u_i)\setminus\{n+2\})\bigcdot u_j}\le m_{u_1\cdots u_{j-1}},\,j=k+2,\dots,n$$

$$\sum_{j=k+2}^n m_{(\redu^{-1}_1(u_i)\setminus\{n+2\})\bigcdot u_j}
=|X_i|, \quad i=1,\dots,p$$
$$\sum_{i=1}^{p}\sum_{j=k+2}^n m_{(\redu^{-1}_1(u_i)\setminus\{n+2\})\bigcdot u_j}
=|X|$$


\medskip

\item[(b)] $n+1\in\r$,

\begin{itemize}

\item [(i)] $X=X_p\cdots X_q $ a factorization,  possibly  some empty factors, with $X_i=x_1^{X_i}\cdots x_{|X_i|}^{X_i}$, $q\le i\le p$,
 for some $1\le q\le  p<k+1$,
 \begin{align} &(S^{H,\mu})^{1,\dots,1,\dots,n,\dots,n,I^X,(n+1)^\alpha,()^M}=\nonumber\\
   &=(12\cdots n\;n+1\, n+2)^M\bigcdot(12\cdots n\;u_{k+1}=n+2)^{m_{\redu^{-1}_1(u_{k+1})}}\bigcdot \nonumber\\
   &\bigcdot(12\cdots n\,n+2)^{\sum_{i=q}^{p}\sum_{j=k+2}^n m_{(\redu^{-1}_1(u_i)\setminus\{n+2\})\bigcdot u_j}+\sum_{i=1}^{q}m_{(\redu^{-1}_1(u_i)\setminus\{n+2\})}}\bigcdot\nonumber\\
&\bigcdot \circledcirc_{i=k}^{p}(\redu^{-1}_1(u_i))^{m_{\redu^{-1}_1(u_i)}}\bigcdot\nonumber\\
&\bigcdot\circledcirc_{i=p}^{q}[\circledcirc_{j=1}^{|X_{i}|}(\redu^{-1}_1(u_i)\setminus\{n+2\})\bigcdot x_j^{X_{i}}] \bigcdot\circledcirc_{i=q}^{1}(\redu^{-1}_1(u_i)\setminus\{n+2\})^{m_{\redu^{-1}_1(u_i)\setminus\{n+2\}}}\bigcdot
 S^{H,\mu'}, \nonumber
\end{align}
 where

$$ m_{(12\cdots n\,n+2)}-\sum_{i=1}^{p}m_{\redu^{-1}_1(u_i)\setminus\{n+2\}}-
\sum_{i=p}^{k}\sum_{j=k+2}^n m_{(\redu^{-1}_1(u_i)\setminus\{n+2\})\bigcdot u_j}\le m_{u_1\cdots u_{k}}$$
$$ m_{\redu^{-1}_1(u_i)}\le m_{u_1\cdots u_{i-1}},\quad i=p-1,\dots,k$$
$$ m_{\redu^{-1}_1(u_p)}+|X_p|\le m_{u_1\cdots u_{p-1}},$$
$$
|X_i|\le m_{u_1\cdots u_{i-1}},\quad i=q+1,\dots,p-1,$$

$$m_{\redu^{-1}_1(u_q)\setminus\{n+2\}}+|X_q|\le m_{u_1\cdots u_{q-1}}$$
$$\redu^{-1}_1(u_i)\setminus\{n+2\}\le m_{u_1\cdots u_{i-1}},\quad i=2,\dots,q-1$$
$$ \sum_{i=q}^{p} m_{(\redu^{-1}_1(u_i)\setminus\{n+2\})\bigcdot u_j}\le m_{u_1\cdots u_{j-1}},\,j=k+2,\dots,n$$

$$ \sum_{j=k+2}^n m_{(\redu^{-1}_1(u_i)\setminus\{n+2\})\bigcdot u_j}
=|X_i|, \quad i=q,\dots,p$$
$$\sum_{i=q}^{p}\sum_{j=k+2}^n m_{(\redu^{-1}_1(u_i)\setminus\{n+2\})\bigcdot u_j}
=|X|$$

and $\alpha=\sum_{i=1}^{q}m_{(\redu^{-1}_1(u_i)\setminus\{n+2\})}$ the number of columns of length $n$.
 \bigskip

 \item[(ii)]$X=YX_{k}\cdots X_q$
 a factorization with $X_i=x_1^{X_i}\cdots x_{|X_i|}^{X_i}$, $q\le i\le k$,
 for some  $1\le q\le k$, and $Y$ a prefix increasing subword in the alphabet $\{u_{k+2},\dots, u_{k+t}\}$ for some $2\le t\le k$  such tat $u_{k+t}\le x_j^{X_i}$ for all $i,j$,

 \begin{align} &(S^{H,\mu})^{1,\dots,1,\dots,n,\dots,n,I^X,(n+1)^\alpha,()^M}=\nonumber\\
 &(12\cdots n+1\,n+2)^M\bigcdot (12\cdots n\, \,u_{k+1}=n+2)^{m_{\redu^{-1}_1(u_{k+1})}}
\bigcdot  \nonumber\\
&(12\cdots n\, \,n+2)^{\sum_{i=q}^{k}\sum_{j=k+2}^n m_{(\redu^{-1}_1(u_i)\setminus\{n+2\})\bigcdot u_j}+\sum_{i=1}^{q}m_{\redu^{-1}_1(u_i)\setminus\{n+2\}}}
\nonumber\\
&
\bigcdot\circledcirc_{i=2}^{t}(12\cdots n\, u_{k+i})^{m_{12\cdots n\,u_{k+i}}}\nonumber\\
&\bigcdot\circledcirc_{i=k}^{q}[\circledcirc_{j=1}^{|X_{i}|}(\redu^{-1}_1(u_i)\setminus\{n+2\})\bigcdot x_j^{X_{i}}]\bigcdot\nonumber\\
&\bigcdot \circledcirc_{i=q}^{1}(\redu^{-1}_1(u_i)\setminus\{n+2\})^{m_{\redu^{-1}_1(u_i)\setminus\{n+2\}}}\bigcdot
 S^{H,\mu'},
 \nonumber
\end{align}

where

$$  m_{(12\cdots n\,n+2)}-\sum_{i=1}^{q}m_{\redu^{-1}_1(u_i)\setminus\{n+2\}}-
\sum_{i=q}^{k}\sum_{j=k+2}^n m_{(\redu^{-1}_1(u_i)\setminus\{n+2\})\bigcdot u_j}\le m_{u_1\cdots u_{k}},$$
$$m_{\redu^{-1}_1(u_i)\setminus\{n+2\}}\le m_{u_1\cdots u_{i-1}},\quad i=2,\dots,q-1,$$
$$ m_{\redu^{-1}_1(u_q)\setminus\{n+2\}}+|X_q|\le m_{u_1\cdots u_{q-1}},$$
$$  \sum_{j=k+t}^n m_{(\redu^{-1}_1(u_i)\setminus\{n+2\})\bigcdot u_j}=|X_i|\le m_{u_1\cdots u_{i}},\quad i=k,\dots,q-1,$$
$$\sum_{i=q}^km_{(\redu^{-1}_1(u_i)\setminus\{n+2\})\bigcdot u_{j}}\le m_{u_1\cdots u_{j-1}},\quad j=k+t+1,\dots,n,$$
$${m_{12\cdots n-1\,u_{k+i}}}\le {m_{u_1\cdots \,u_{k+i-1}}},\quad i=1,\dots,t-1$$
$${m_{12\cdots n\,u_{k+t}}}+\sum_{i=p}^{k}m_{(\redu^{-1}_1(u_i)\setminus\{n+2\})\bigcdot u_{k+t}}\le m_{u_1\cdots u_{k+t-1}}$$
$$|Y|=\sum_{i=1}^{q}m_{\redu^{-1}_1(u_i)\setminus\{n+2\}}=\sum_{i=2}^tm_{12\cdots n\,u_{k+i}}$$

and $\alpha=\sum_{i=1}^{q}m_{(\redu^{-1}_1(u_i)\setminus\{n+2\})}$ the number of columns of length $n$.

\begin{ex} For $n=4$, $k=2$, $u_2=2$, $u_3=6$, $u_4=7$ \eqref{l=6n=4type31XXXY}
\begin{align}&m_{12357}\le m_2,\quad m_{12347}+m_{12357}+m_{23457}\le m_{236},\quad\\
& 0\le m_{12346}- m_{12357}-m_{23457}-m_{2345}\le m_{23}
\end{align}
\end{ex}

\item[(iii)] $X=Y$
\begin{align}&(S^{H,\mu})^{1,\dots,1,\dots,n,\dots,n,I^X,(n+1)^\alpha,()^M}=\nonumber\\
&=(12\cdots n\,n+1\,n+2)^M\bigcdot(12\dots,n-1,n,\,u_{k+1}=n+2)^{m_{\redu^{-1}_1(u_{k+1})}}\bigcdot \nonumber\\
& \bigcdot(12\cdots n\,n+2)^{\sum_{i=1}^{k-1}m_{\redu^{-1}_1(u_i)\setminus\{n+2\}}}
\bigcdot \nonumber\\
&\bigcdot\circledcirc_{i=1}^{k}(12\cdots n\, u_{k+i})^{m_{12\cdots n\,u_{k+i}}}
\bigcdot (12\cdots n)^{m_{12\cdots n-1}}\bigcdot\nonumber
\\
&\bigcdot\circledcirc_{i=k-1}^{1}(\redu^{-1}_1(u_i)\setminus\{n+2\})^{m_{\redu^{-1}_1(u_i)\setminus\{n+2\}}}\bigcdot\nonumber\\
&\bigcdot S^{H,\mu'}
\end{align}

where

$$  m_{(12\cdots n\,n+2)}-\sum_{i=1}^{p-1}m_{\redu^{-1}_1(u_i)\setminus\{n+2\}}\le m_{u_1\cdots u_{k}},$$
$$m_{\redu^{-1}_1(u_i)\setminus\{n+2\}}\le m_{u_1\cdots u_{i-1}},\quad i=2,\dots,k,$$
$$m_{12\cdots n\,u_{k+i}}\le m_{u_1\cdots u_{k+i-1}},\quad i=1,\dots,k.$$

and $\alpha=\sum_{i=1}^{q}m_{(\redu^{-1}_1(u_i)\setminus\{n+2\})}+m_{12\cdots n}$  the number of columns of length $n$.

\end{itemize}

 \end{itemize}
$$\qquad \qquad\qquad\qquad\qquad\qquad\qquad\qquad\qquad\qquad\qquad\qquad \qquad\qquad\qquad\qquad\square$$

\end{enumerate}
\end{thm}

\bigskip
Consider $n\in 2\Z$ and the numbers in \eqref{uk=nevenkodd}, where $k\notin 2\Z$ and $u_k=n$. Let  $l=n+1$ ($l-1\in 2\Z$) and recall \eqref{expexp}. For $i=1,\dots,k,\dots,n$, let \begin{align}
\redu^{-1}_1:SpT_{2n}(u_i)&\rightarrow SST_{2n}(\varpi_{n+1})
\nonumber
\end{align}

 for $l\notin 2\Z$ and $t=1$.

where $$\redu^{-1}_1(u_1)=(u_1=2,3,\dots,n,\,n+1,n+2)$$
$$\dots$$
$$\redu^{-1}_1(u_{k-1})=(12\cdots\,n-4,\, u_{k-1}= n-3,n-1,n\,\,n+1,n+2)$$
 $$\redu^{-1}_1(u_k)=(12\dots,n-2,u_k=n,\,n+1\,n+2)$$
$$\redu^{-1}_1(u_{k+1})=(12\dots,n-2,n-1,n,\,u_{k+1}=n+1)$$
$$\redu^{-1}_1(u_{k+2})=(12\dots,n-2,n-1,n,\,u_{k+2}=n+4)$$
$$\dots$$
$$\redu^{-1}_1(u_{k+k})=(12\dots,n-2,n-1,n,u_{k+k}=2n-1)$$

\begin{obs} The case $n=2$ has been considered in \cite{nsw} and \cite[Corollary 7]{azreco} and for $\ell(\lambda)=2$ reduces to the case $n=1$.
\end{obs}

\medskip
\begin{thm} \label{thm:nevenkodd} Let $ 6\le n\in 2\Z$ and $S^{H,\varpi_n}$ as in \eqref{uk=oddkeven} where $u_k=n\in 2\Z$, $u_{k+1}=n+1$, and $k\notin 2\Z$. Consider $S^{H,\mu}$  the $\mathfrak{k}$-highest weight tableau
in $SpT_{2n}(\mu)$ as in \eqref{symphw-hw}. Let  $Q\in  Rec^{1,0}_{2n}(\lambda/\mu)$, $\mu\subsetneq \lambda$, and $\ell(\lambda)=n+2\in 2\Z$, with  $1$-$0$-slack sequence of the form $\underline \t=(1,\dots,1,0^M)$ and  corresponding  slack row index vector sequence   of the form
 $$\underline \r=(1,\dots,1,2,\dots, 2,\dots,k+1,\dots, k+1,\dots,n,\dots,n,n+1,\dots,n+1,()^M).$$  Let  $\mu'=\mu-\delta_{\underline\r^+}$ .
Then, ${\LRAII^{AII}}^{-1}(S^{H,\mu},Q)$ returns   the $\mathfrak{k}$-highest  weight tableau in $SST^{\mathfrak{k}-hw}_{2n}(\lambda,\mu)$ of one of either  forms:
\begin{enumerate}
\item when $k+2,\dots,n,n+1\notin \underline \r$,

\begin{align}&(S^{H,\mu})^{1,\dots,1,\dots,k+1,\dots,k+1,()^M}=\nonumber\\
&=(12\cdots  n\,n+1,n+2)^M\bigcdot(12\dots,n-1,n,u_{k+1}=n+1)^{m_{\redu^{-1}_1(u_{k+1})}}\bigcdot\nonumber\\
&\circledcirc_{i=k}^{1}(\redu^{-1}_1(u_i))^{m_{\redu^{-1}_1(u_i)}}\bigcdot
\bigcdot S^{H,\mu'}
\end{align}
where
$$m_{\redu^{-1}_1(u_i)}\le m_{u_1\cdots u_{i-1}},\quad i=2,\dots,k+1.$$


\item  when $n+1\in \underline \r$ and $k+2,\dots,n\notin \underline \r$ ($\Leftrightarrow $ some $1,\dots,k+1\in\r$)
\begin{itemize}
\item[(a)]
\begin{align}&(S^{H,\mu})^{1,\dots,1,\dots,k+1,\dots,k+1,(n+1)^\alpha,()^M}=\nonumber\\
&(12\cdots n\,n+1,n+2)^M\bigcdot(12\dots,n-1,n,u_{k+1}=n+1)^{m_{\redu^{-1}_1(u_{k+1})}}\bigcdot\nonumber\\
&\bigcdot(12\dots,n-1,n,n+2)^{\sum_{i=1}^p{m_{\redu^{-1}_1(u_i)\setminus\{n+2\}}}}\nonumber\\
&\bigcdot\circledcirc_{i=k}^{p}(\redu^{-1}_1(u_i))^{m_{\redu^{-1}_1(u_i)}}\bigcdot\nonumber\\
&\bigcdot\circledcirc_{i=p}^{1}(\redu^{-1}_1(u_i)\setminus\{n+2\})^{m_{\redu^{-1}_1(u_i)\setminus\{n+2\}}}\bigcdot
S^{H,\mu'},\quad \mbox{ for some $1\le p\le k$},
\end{align}
where
$$ m_{(12\cdots n\,n+1)}\le 
m_{u_1\cdots u_k}$$
$$m_{\redu^{-1}_1(u_i)\setminus\{n+2\}}\le m_{u_1\cdots u_{i-1}},\quad i=2,\dots,p-1.$$
$$m_{\redu^{-1}_1(u_i)}\le m_{u_1\cdots u_{i-1}},\quad i=p-1,\dots,k.$$
$$m_{\redu^{-1}_1(u_p)}+m_{\redu^{-1}_1(u_p)\setminus\{n+2\}}\le m_{u_1\cdots u_{p-1}},$$
and $\alpha= \sum_{i=1}^p{m_{\redu^{-1}_1(u_i)\setminus\{n+2\}}}$  the number of columns of length $n$.
\item[(b)]
\begin{align}&(S^{H,\mu})^{1,\dots,1,\dots,k+1,\dots,k+1,(n+1)^\alpha,()^M}=\nonumber\\
&(12\cdots n\,n+1,n+2)^M\bigcdot(12\dots,n-1,n,u_{k+1}=n+1)^{m_{\redu^{-1}_1(u_{k+1})}}\bigcdot\nonumber\\
&\bigcdot(12\dots,n-1,n,n+2)^{\sum_{i=1}^k{m_{\redu^{-1}_1(u_i)\setminus\{n+2\}}}}\bigcdot (12\cdots n)^{m_{12\cdots n}}\bigcdot\nonumber\\
&\bigcdot\circledcirc_{i=k}^{1}(\redu^{-1}_1(u_i)\setminus\{n+2\})^{m_{\redu^{-1}_1(u_i)\setminus\{n+2\}}}\bigcdot
S^{H,\mu'},
\end{align}

where
$$ m_{(12\cdots n\,n+1)}\le 
m_{u_1\cdots u_k}$$
$$m_{\redu^{-1}_1(u_i)\setminus\{n+2\}}\le m_{u_1\cdots u_{i-1}},\quad i=2,\dots,k,$$
and $\alpha= \sum_{i=1}^k{m_{\redu^{-1}_1(u_i)\setminus\{n+2\}}}+m_{12\cdots n}$  the number of columns of length $n$.
\end{itemize}
\item  when  $1,\dots,k+1\notin \underline \r$ and some $k+2,\dots,n\in \underline \r$   and possibly  $n+1\in\r$

\begin{align}&(S^{H,\mu})^{k+2,\dots,n,(n+1)^\alpha,()^M}=\nonumber\\
&(12\cdots n\,n+1,n+2)^M\bigcdot \circledcirc_{i=k+2}^{n}(12\dots n\, u_i)^{m_{\redu^{-1}_1(u_i)}}\bigcdot (12\dots n)^{m_{12\dots n}}
\bigcdot
 S^{H,\mu'}, \nonumber
 \end{align}

where $$m_{(12\dots n\, u_i)}\le m_{u_1\cdots u_{i-1}},\quad i=k+2,\dots,n$$

and $\alpha=m_{12\dots n}$  the number of columns of length $n$.

\item  \label{wordX} when some $1,\dots,k+1\in\r$ and some $k+2,\dots,n\in \r$, let $X$  be a nonempty weakly increasing word  in the alphabet
$$\{u_{k+2}=n+4,\dots,u_n=u_{k+k}=2n-1\}$$
and $I^X$ the weakly increasing word of the same length as X, in the alphabet $\{k+2,\dots,n\}$ defined by the indices of the letters $u_i$ in $X$ (taking into account their  multiplicities).
\begin{itemize}
\medskip
   \item[(a)] $n+1\notin\r$ and $X=X_p\cdots X_1$ a factorization,  possibly  some empty factors, with $X_i=x_1^{X_i}\cdots x_{|X_i|}^{X_i}$, $1\le i\le p$,
 for some $1\le p<k+1$,

   \begin{align}&(S^{H,\mu})^{1,\dots,1,\dots,n,\dots,n,I^X,()^M}=\nonumber\\
   &(12\cdots n\;n+1\, n+2)^M\bigcdot(12\cdots n\;u_{k+1}=n+1)^{m_{\redu^{-1}_1(u_{k+1})}}\bigcdot \nonumber\\
   &\bigcdot(12\cdots n\,n+2)^{\sum_{i=1}^{p}\sum_{j=k+2}^n m_{(\redu^{-1}_1(u_i)\setminus\{n+2\})\bigcdot u_j}}\bigcdot\nonumber\\
&\bigcdot \circledcirc_{i=k}^{p}(\redu^{-1}_1(u_i))^{m_{\redu^{-1}_1(u_i)}}\bigcdot\nonumber\\
&\bigcdot\circledcirc_{i=p}^{1}[\circledcirc_{j=1}^{|X_{i}|}(\redu^{-1}_1(u_i)\setminus\{n+2\})\bigcdot x_j^{X_{i}}]\bigcdot
 S^{H,\mu'}, \nonumber
\end{align}

 where
$$ m_{(12\cdots n\,u_{k+1}=n+1)}\le m_{u_1\cdots u_{k}}$$
$$ m_{\redu^{-1}_1(u_i)}\le m_{u_1\cdots u_{i-1}},\quad i=p-1,\dots,k$$
$$ m_{\redu^{-1}_1(u_p)}+|X_p|\le m_{u_1\cdots u_{p-1}},$$
$$
|X_i|\le m_{u_1\cdots u_{i-1}},\quad i=2,\dots,p-1,$$
$$ \sum_{i=1}^{p} m_{(\redu^{-1}_1(u_i)\setminus\{n+2\})\bigcdot u_j}\le m_{u_1\cdots u_{j-1}},\,j=k+2,\dots,n$$

$$\sum_{j=k+2}^n m_{(\redu^{-1}_1(u_i)\setminus\{n+2\})\bigcdot u_j}
=|X_i|, \quad i=1,\dots,p$$
$$\sum_{i=1}^{p}\sum_{j=k+2}^n m_{(\redu^{-1}_1(u_i)\setminus\{n+2\})\bigcdot u_j}
=|X|$$

\medskip
\item[(b)]  $n+1\in\r$,
\begin{itemize}
\item [(i)] $X=X_p\cdots X_q $ a factorization,  possibly  some empty factors, with $X_i=x_1^{X_i}\cdots x_{|X_i|}^{X_i}$, $q\le i\le p$,
 for some $1\le q\le  p<k+1$,

\begin{align}&(S^{H,\mu})^{1,\dots,1,\dots,n,\dots,n,I^X,(n+1)^\alpha,()^M}=\nonumber\\
   &=(12\cdots n\;n+1\, n+2)^M\bigcdot(12\cdots n\;u_{k+1}=n+1)^{m_{\redu^{-1}_1(u_{k+1})}}\bigcdot \nonumber\\
   &\bigcdot(12\cdots n\,n+2)^{\sum_{i=q}^{p}\sum_{j=k+2}^n m_{(\redu^{-1}_1(u_i)\setminus\{n+2\})\bigcdot u_j}+\sum_{i=1}^{q}m_{(\redu^{-1}_1(u_i)\setminus\{n+2\})}}\bigcdot\nonumber\\
&\bigcdot \circledcirc_{i=k}^{p}(\redu^{-1}_1(u_i))^{m_{\redu^{-1}_1(u_i)}}\bigcdot\nonumber\\
&\bigcdot\circledcirc_{i=p}^{q}[\circledcirc_{j=1}^{|X_{i}|}(\redu^{-1}_1(u_i)\setminus\{n+2\})\bigcdot x_j^{X_{i}}] \bigcdot\circledcirc_{i=q}^{1}(\redu^{-1}_1(u_i)\setminus\{n+2\})^{m_{\redu^{-1}_1(u_i)\setminus\{n+2\}}}\bigcdot
 S^{H,\mu'}, \nonumber
\end{align}
where
$$ m_{(12\cdots n\,u_{k+1}=n+1)}\le m_{u_1\cdots u_{k}}$$
$$ m_{\redu^{-1}_1(u_i)}\le m_{u_1\cdots u_{i-1}},\quad i=p-1,\dots,k$$
$$ m_{\redu^{-1}_1(u_p)}+|X_p|\le m_{u_1\cdots u_{p-1}},$$
$$
|X_i|\le m_{u_1\cdots u_{i-1}},\quad i=q+1,\dots,p-1,$$

$$m_{\redu^{-1}_1(u_q)\setminus\{n+2\}}+|X_q|\le m_{u_1\cdots u_{q-1}}$$
$$\redu^{-1}_1(u_i)\setminus\{n+2\}\le m_{u_1\cdots u_{i-1}},\quad i=2,\dots,q-1$$
$$ \sum_{i=q}^{p} m_{(\redu^{-1}_1(u_i)\setminus\{n+2\})\bigcdot u_j}\le m_{u_1\cdots u_{j-1}},\,j=k+2,\dots,n$$

$$ \sum_{j=k+2}^n m_{(\redu^{-1}_1(u_i)\setminus\{n+2\})\bigcdot u_j}
=|X_i|, \quad i=q,\dots,p$$
$$\sum_{i=q}^{p}\sum_{j=k+2}^n m_{(\redu^{-1}_1(u_i)\setminus\{n+2\})\bigcdot u_j}
=|X|$$

and  $\alpha=\sum_{i=1}^{q}m_{(\redu^{-1}_1(u_i)\setminus\{n+2\})}$  the number of columns of length $n$.

\bigskip
\item[(ii)]
$X=YX_{k}\cdots X_q$
a factorization with $X_i=x_1^{X_i}\cdots x_{|X_i|}^{X_i}$, $q\le i\le k$,
 for some  $1\le q\le k$, and $Y$ a prefix increasing subword in the alphabet $\{u_{k+2},\dots, u_{k+t}\}$ for some $2\le t\le k$  such tat $u_{k+t}\le x_j^{X_i}$ for all $i,j$,

\begin{align}&(S^{H,\mu})^{1,\dots,1,\dots,n,\dots,n,(n+1)^\alpha,()^M}=\nonumber\\
 &=(12\cdots n+1\,n+2)^M\bigcdot (12\cdots n\, \,u_{k+1}=n+1)^{m_{\redu^{-1}_1(u_{k+1})}}
\bigcdot  \nonumber\\
&(12\cdots n\, \,n+2)^{\sum_{i=p}^{k}\sum_{j=k+2}^n m_{(\redu^{-1}_1(u_i)\setminus\{n+2\})\bigcdot u_j}+\sum_{i=1}^{q}m_{\redu^{-1}_1(u_i)\setminus\{n+2\}}}
\nonumber\\
&
\bigcdot\circledcirc_{i=2}^{t}(12\cdots n\, u_{k+i})^{m_{12\cdots n\,u_{k+i}}}\nonumber\\
&\bigcdot\circledcirc_{i=k}^{q}[\circledcirc_{j=1}^{|X_{i}|}(\redu^{-1}_1(u_i)\setminus\{n+2\})\bigcdot x_j^{X_{i}}]\bigcdot\nonumber\\
&\bigcdot \circledcirc_{i=q}^{1}(\redu^{-1}_1(u_i)\setminus\{n+2\})^{m_{\redu^{-1}_1(u_i)\setminus\{n+2\}}}\bigcdot
 S^{H,\mu'},
 \nonumber
\end{align}
where

$$ m_{(12\cdots n\,u_{k+1}=n+1)}\le m_{u_1\cdots u_{k}}$$
$$|Y|=\sum_{i=1}^{q}m_{\redu^{-1}_1(u_i)\setminus\{n+2\}}=\sum_{i=2}^tm_{12\cdots n\,u_{k+i}}$$
$$m_{\redu^{-1}_1(u_i)\setminus\{n+2\}}\le m_{u_1\cdots u_{i-1}},\quad i=2,\dots,q-1,$$
$$ m_{\redu^{-1}_1(u_q)\setminus\{n+2\}}+|X_q|\le m_{u_1\cdots u_{q-1}},$$
$$ \sum_{j=k+t}^n m_{(\redu^{-1}_1(u_i)\setminus\{n+2\})\bigcdot u_j}
=|X_i|\le m_{u_1\cdots u_{i}},\quad i=k,\dots,q-1,$$

$$\sum_{i=q}^km_{(\redu^{-1}_1(u_i)\setminus\{n+2\})\bigcdot u_{j}}\le m_{u_1\cdots u_{j-1}},\quad j=k+t+1,\dots,n,$$

$${m_{12\cdots n-1\,u_{k+i}}}\le {m_{u_1\cdots \,u_{k+i-1}}},\quad i=1,\dots,t-1$$
$${m_{12\cdots n\,u_{k+t}}}+\sum_{i=q}^{k}m_{(\redu^{-1}_1(u_i)\setminus\{n+2\})\bigcdot u_{k+t}}\le m_{u_1\cdots u_{k+t-1}}$$

and $\alpha=\sum_{i=1}^{q}m_{(\redu^{-1}_1(u_i)\setminus\{n+2\})}$  the number of columns of length $n$.

\item[(iii)]  $X=Y$
\begin{align}&(S^{H,\mu})^{1,\dots,1,\dots,n,\dots,n,(n+1)^\alpha,()^M}=\nonumber\\
&=(12\cdots n\,n+1\,n+2)^M\bigcdot(12\dots,n-1,n,\,u_{k+1}=n+1)^{m_{\redu^{-1}_1(u_{k+1})}}\bigcdot\\ &\bigcdot(12\cdots n\,n+2)^{\sum_{i=1}^{k}m_{\redu^{-1}_1(u_i)\setminus\{n+2\}}}
\bigcdot \nonumber\\
&\bigcdot\circledcirc_{i=2}^{k}(12\cdots n\, u_{k+i})^{m_{12\cdots n\,u_{k+i}}}
\bigcdot (12\cdots n)^{m_{12\cdots n}}\bigcdot\nonumber
\\
&\bigcdot\circledcirc_{i=k}^{1}(\redu^{-1}_1(u_i)\setminus\{n+2\})^{m_{\redu^{-1}_1(u_i)\setminus\{n+2\}}}\bigcdot\nonumber\\
&\bigcdot S^{H,\mu'}
\end{align}

where

$$m_{\redu^{-1}_1(u_i)\setminus\{n+2\}}\le m_{u_1\cdots u_{i-1}},\quad i=2,\dots,k,$$
$$m_{12\cdots n\,u_{k+i}}\le m_{u_1\cdots u_{k+i-1}},\quad i=1,\dots,k.$$

and $\alpha=\sum_{i=1}^{q}m_{(\redu^{-1}_1(u_i)\setminus\{n+2\})}+m_{12\cdots n}$  the number of columns of length $n$.
\end{itemize}

 \end{itemize}
\end{enumerate}
$$\qquad \qquad\qquad\qquad\qquad\qquad\qquad\qquad\qquad\qquad\qquad\qquad \qquad\qquad\qquad\qquad\qquad\qquad\qquad\qquad\qquad\square$$

\end{thm}

\begin{obs}\label{enough}For $n\in 2\Z$, it is enough to study the case when $\ell(\lambda)=n+2$, and for $n\notin 2\Z$ it is enough to study the case when $\ell(\lambda)=n+1$. The other even shape lengths less or equal than $2n$ are obtained from these cases by adapting the inverse reduction map to the new length $l$ and conveniently reduce the length of the word $X$ which appears in the item \eqref{wordX2}, \eqref{wordX}of the previous theorems of this subsection. We illustrate this in the next example.

\end{obs}

\begin{ex}\label{ex:wordx}Example for \eqref{wordX} $n=2\times 5=10$, $k=5$,
\begin{enumerate}
\item Let $\ell(\lambda)=n+2=12$. See Figure \ref{fig:figura12}

\begin{align}S^{H,\varpi_{10}}=(u_1=2,\dots,u_4=7,u_{5}=10, u_6=11,u_{7}=14,u_8=15,u_9=18,u_{10}=19=2n-1)\label{uk=10}\end{align}

\begin{figure}[htbp] \centering \includegraphics[width=0.9\textwidth]{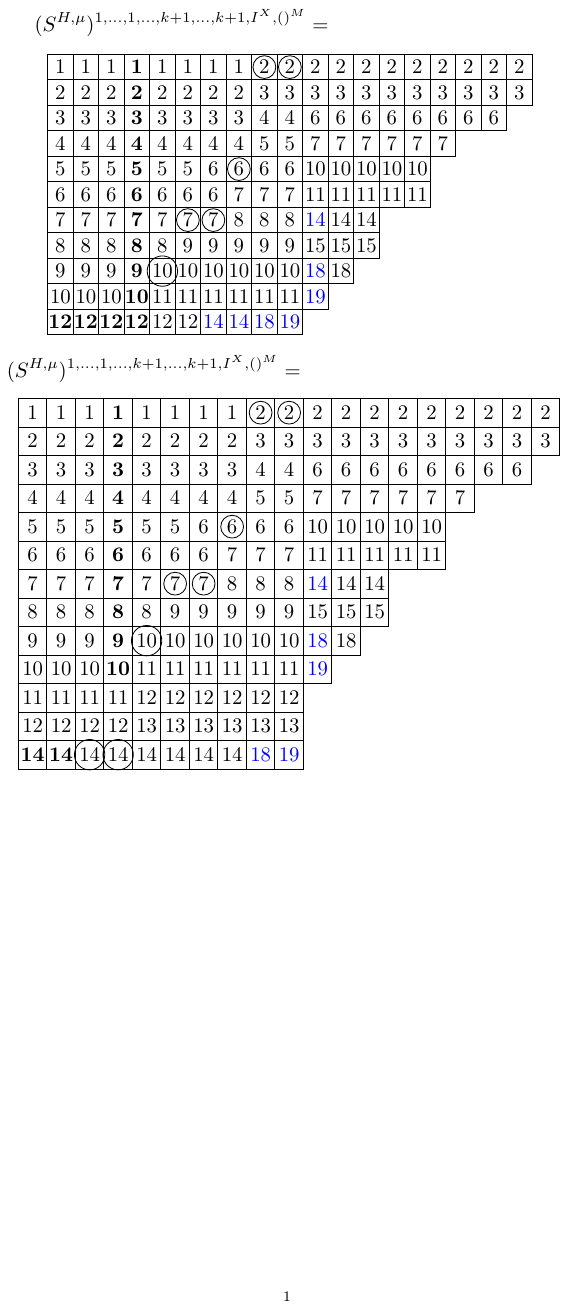}
\caption{$\ell(\lambda)=12$ and $14$.}
\label{fig:figura12}
 \end{figure}
 \begin{figure}[htbp] \centering \includegraphics[width=0.9\textwidth]{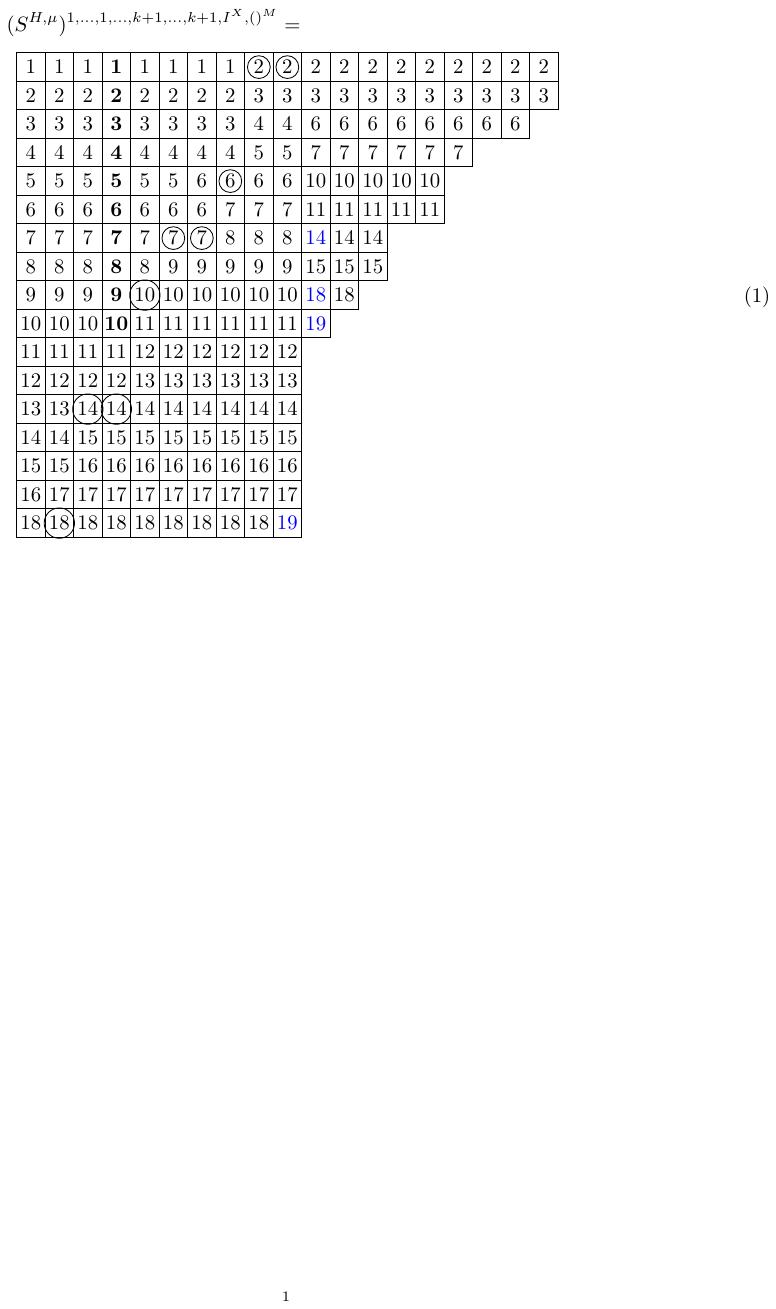}
\caption{$\ell(\lambda)=18$}
\label{fig:figura18}
 \end{figure}

$$X=(u_7=14,u_7=14,u_9=18,u_{10}=19)$$
is a word in $u_7=14,u_8=15,u_9=18,u_{10}=19$
$I^X=(7,7,9,10)$, $p=4$, $X_4=X_3=(14), X_2=\emptyset, X_1=(18,19)$
$$ m_{(12\cdots 10\,u_{5+1}=10+1)}=0$$
$$ m_{\redu^{-1}_1(u_i)}\le m_{u_1\cdots u_{i-1}},\quad i=3,4,5$$
$$|X_4|=m_{\redu^{-1}_1 (u_{4})\setminus\{12\}}=1$$
$$ m_{\redu^{-1}_1(u_4)}+|X_4|\le m_{u_1u_2 u_{3}},$$
$$m_{\redu^{-1}_1(u_1)\setminus\{12\}}=|X_1|=2$$
$$m_{\redu^{-1}_1(u_2)\setminus\{12\}}=|X_2|=0\le m_{u_1}$$
$$m_{\redu^{-1}_1(u_3)\setminus\{12\}}=|X_3|=1\le m_{u_1 u_{2}},$$
$$ \sum_{i=1}^{4} m_{(\redu^{-1}_1(u_i)\setminus\{12\})\bigcdot u_j}\le m_{u_1\cdots u_{j-1}},\,j=7,\dots,10$$
$$ \sum_{i=1}^{4} m_{(\redu^{-1}_1(u_i)\setminus\{12\})\bigcdot u_7=14}=2\le m_{u_1\cdots u_{6}}$$
$$ \sum_{i=1}^{4} m_{(\redu^{-1}_1(u_i)\setminus\{12\})\bigcdot u_8}=0\le m_{u_1\cdots u_{7}}$$
$$ \sum_{i=1}^{4} m_{(\redu^{-1}_1(u_i)\setminus\{12\})\bigcdot u_9=18}=1\le m_{u_1\cdots u_{8}}$$
$$ \sum_{i=1}^{4} m_{(\redu^{-1}_1(u_i)\setminus\{12\})\bigcdot u_{10}=19}=1\le m_{u_1\cdots u_{9}}$$
$$m_{12\dots u_n=10,12}=\sum_{i=1}^{4}\sum_{j=7}^{10} m_{(\redu^{-1}_1(u_i)\setminus\{n+2\}\bigcdot u_j} =\sum_{i=1}^{4} m_{(\redu^{-1}_1(u_i)\setminus\{12\})\bigcdot u_7=14}$$
$$+\sum_{i=1}^{4} m_{(\redu^{-1}_1(u_i)\setminus\{12\})\bigcdot u_8}+\sum_{i=1}^{4} m_{(\redu^{-1}_1(u_i)\setminus\{12\})\bigcdot u_9=18}
+\sum_{i=1}^{4} m_{(\redu^{-1}_1(u_i)\setminus\{12\})\bigcdot u_{10}=19}=2+0+1+1=4$$

\item Let $\ell(\lambda)=n+4=14$. See Figure \ref{fig:figura12}
$$X=(u_9=18,u_{10}=19)$$
is a word in $u_9=18,u_{10}=19$
$I^X=(9,10)$, $p=2$, $ X_2=\emptyset, X_1=(18,19)$

\item $\ell(\lambda)=n+8=18$. See Figure \ref{fig:figura18}
\end{enumerate}
\end{ex}

\section{Illustration of k-highest weight and dominant tableaux  via   1-0-slack recording tableaux for n equal 3 and  4}\label{sec:illustration}
\subsection{Case n=3}
Let $n=3$. Thanks to Lemma \ref{lem:basic}, $(1)$, $(2)$, the following is a generalization of \cite[Corollary 6]{azreco} for any $\ell(\mu)\le \ell(\lambda)\le 2n$ with $\ell(\mu)\le n$.

\begin{thm}\label{thm:n=3slack1} Let $n=3$. Consider $S^{H,\mu}$  the $\mathfrak{k}$-highest weight tableau
in $SpT_6(\mu)$. Let  $Q\in  Rec^{1,0}_{6}(\lambda/\mu)$.
Then,
${\LRAII^{AII}}^{-1}(S^{H,\mu},Q)$ returns   the $\mathfrak{k}$-highest  weight tableau in
$SST_{6}(\lambda)$ with ${\mathfrak{k}}$-weight  $\mu$    in either  form below, with the following caption, circled elements indicate bumped entries from $S^{H,\mu}$, blue circle indicates the column $\redu_1^{-1}(2)$, orange circle indicates the column $\redu_1^{-1}(3)$ and brown circle indicates the column $\redu_1^{-1}(6)$.

\begin{enumerate}
\item  For $\ell(\lambda)=6$ the $\k$-highest tableaux are the ones in \cite[Corollary 6]{azreco}.
These tableaux satisfy linear inequalities  on the multiplicities 
of the following columns: if $m_{12356}$ is the multiplicity of $\redu^{-1}(3)=(12356)$, $m_{12346}$ the multiplicity of $\redu^{-1}(6)=(12346)$,
  $m_{23}$ the multiplicity of $(2,3)\in SpT_6(\varpi_2)$,  and $m_{2}$ the multiplicity of $(2)\in SpT_6(\varpi_1) $,
 then

\begin{align}1&\le m_{123456}\nonumber\\
&m_{12{\bf 3}56}+m_{1235}\le m_{\bf 2} \mbox{ and }0\le m_{1234\bf 6}-m_{1235}-m_{2345}\le m_{2\bf 3}.\label{Yineqn=3}
\end{align}

\item If $\ell(\lambda)=5$,  Lemma \ref{lem:basic} guarantees that the  $\k$-highest weight tableaux are the same as for $\ell(\lambda)=6$ with  $M=0$ in  $(123456)^M$ and the  linear inequalities that do satisfy are the same \eqref{Yineqn=3} plus the condition on the length five.

\item If $\ell(\lambda)=4$, one has

\begin{align}\label{typel=41}
&\YT{0.15in}{}{
 {1,\cdots,1,{1},\cdots,1,1,\cdots,1,    \blue{\circled{\bf 2}},\cdots, \blue{\circled{\bf 2}},{2},\cdots,2,    2,\cdots, 2,   2,\cdots,{ 2}},
 {2,\cdots,2,{2},\cdots,2,{2},\cdots,2,3,\cdots,3,  {3},\cdots,3,  { 3},\cdots, 3},
 {3,\cdots,3, \red{\circled{\bf{3}}},\cdots, \red{\circled{\bf 3}},4,\cdots,4,  4,\cdots,4,\ora 6,\cdots,\ora 6},
 {4,\cdots,4},
}
\end{align}

or

\begin{align}\label{typel=42}
&\YT{0.15in}{}{
 {1,\cdots,1,{1},\cdots,1,{1},\cdots,1,1,\cdots,1,  2,\cdots, 2, 2,\cdots, 2, 2,\cdots,{2}},
 {2,\cdots,2,{2},\cdots,2,{2},\cdots,2,2,\cdots,2,  {3},\cdots,3,   {3},\cdots,3},
 {3,\cdots,3,\red{\circled{\bf{3}}},\cdots,\red{\circled{\bf{3}}},4,\cdots,4,  \ora 6,\cdots,\ora 6,\ora 6,\cdots,\ora 6},
 {4,\cdots,4},
}
\end{align}

or
\begin{align}\label{typel=43}
&\YT{0.15in}{}{
 {1,\cdots,1,{1},\cdots,1,{1},\cdots,1,1,\cdots,1,{1},\cdots,1,    2,\cdots, 2, 2,\cdots, 2},
 {2,\cdots,2,{2},\cdots,2,{2},\cdots,2,2,\cdots,2,  {2},\cdots,2,   {3},\cdots,3},
 {3,\cdots, 3,\red{3},\cdots,\red 3,4,\cdots,4,  \ora 6,\cdots,\ora 6},
 {4,\cdots,4},
}
\end{align}
and all these tableaux satisfy the linear inequalities
\begin{align}\label{horseineq=64}&1\le m_{1234}\nonumber\\
&m_{123}\le m_2, \quad m_{124}+m_ {126}\le m_{23}
\end{align}

  \item If $\ell(\lambda)=3$, Lemma \ref{lem:basic} guarantees that the  $\k$-highest weight tableaux are the same tableaux as immediately above with $M=0$ in  $(1234)^M$. The linear inequalities that they do satisfy are the same \eqref{horseineq=64} plus the condition on length three.

  \item If $\ell(\lambda)=2$, either $\ell(\mu)=2$  and
\begin{align}\label{typel=2}
&\YT{0.15in}{}{
 {{1},\cdots,1,    2,\cdots, 2, 2,\cdots, 2},
 {  {2},\cdots,2,   {3},\cdots,3},
}
\end{align} or $\ell(\mu)=1$ and respectively \begin{align}\label{typel=2+}
&\YT{0.15in}{}{
 {{1},\cdots,1,     2,\cdots, 2},
 {  {2},\cdots,2},
}
\end{align}
with no constraints on the multiplicity of the columns besides the condition on the length.

      \item If $\ell(\lambda)=1$, then $\ell(\mu)=\ell(\lambda)=1$, and we are in the previous case \eqref{typel=2+} with $M=0$ in $(12)^M$. Then $Q=\emptyset$ and $S^{H,\mu}=2^M$, $M\ge 0$, ${\LRAII^{AII}}^{-1}(S^{H,\mu},\emptyset)=S^{H,\mu}.$
\end{enumerate}
\end{thm}
\begin{proof} All these tableaux $T$ above can be checked to satisfying $P^{AII}(T)=S^{H,\mu}\in SpT_6(\mu)$ with $\mu=\wt_\k(S)$, using the linear inequalities, and that the reduce map in this computation is always of the type \begin{align}
\redu_l:SST_{2n}(\varpi_l)&\rightarrow SpT_{2n}(\varpi_t)
\end{align} with $t\in \{0,1\}$ which guarantees that the recording tableaux belong to $Rec^{1,0}$.

$(1)$,  $(2)$ From inequalities \eqref{Yineqn=3}, we may check that $ P^{AII}(T)=S^{H,\mu}$ where $\mu$ is as below. For instance,   the tableaux in \cite[Corollary 6]{azreco}, $(108)-(111)$, satisfy
$$\mu=\wt_\k(S)=(m_2+m_{23}+m_{236}+m_{2345},m_{23}+m_{236}+m_{1235}+m_{12356},m_{236}+m_{1234\bf 6}-m_{1235}-m_{2345})$$
Inequalities \eqref{Yineqn=3} ensure that

$$m_2+m_{23}+m_{236}+m_{2345}\ge m_{23}+m_{236}+m_{1235}+m_{12356}\ge m_{236}+m_{1234\bf 6}-m_{1235}-m_{2345}$$
 since $m_{1235}+m_{12356}\le m_2$ and $0\le m_{1234\bf 6}-m_{1235}-m_{2345}\le m_{23} $.

Note that for $(2)$, $\wt_\k(123456)=(0,\dots,0)$ and $\mu=\wt_\k(S)=$$\wt_\k(S\setminus (123456)^M)$, and $ m_{123456}$ has no contribution for the inequalities in \eqref{Yineqn=3}

$(3)$, $(4)$ Similarly, we may check that $ P^{AII}(S)=S^{H,\mu}$ providing inequalities \eqref{horseineq=64} hold .  For  \eqref{typel=41},
$$\mu=\wt_\k(S)=(m_2+m_{236}+m_{23}-m_{124},m_{23}-m_{124}+m_{236}+m_{123},m_{236})\in P_{\le 3}$$
where $m_2\ge m_{123}$.

For \eqref{typel=42},
$$\mu=\wt_\k(S)=(m_2+m_{236}+m_{23}-m_{124},m_{23}-m_{124}+m_{236}+m_{123},m_{236}+m_{126})\in P_{\le 3}$$
where $m_{124}+m_ {126}\le m_{23}\Leftrightarrow m_{23}-m_{124}\ge m_{126}$.

For \eqref{typel=43},
$$\mu=\wt_\k(S)=(m_2+m_{23},m_{23}-m_{124}+m_{123},m_{126})\in P_{\le 3}$$

where $m_2\ge m_{123}$ and $m_{23}-m_{124}\ge m_{126}$.

Similarly for the remaining cases $(5)$, $(6)$.

\end{proof}

We now compute for $n=3$,
$$\Phi^{-1}({\LRAII^{AII}}^{-1}(\{S^{H,\mu}\}\times {Rec}_{6}^{1,0}(\lambda/\mu))\subset SST_6^{{\widehat{\mathfrak{g}}-dom}}(\lambda,\mu),$$
where $\Phi:=\pr_{1,6}\circ \pr_{3,5}\circ\pr_{3,4}$. Unfortunately Lemma \ref{lem:basic} $(2)$, $(a)$ is not preserved, in general, when we apply $\Phi^{-1}$ but this is what should be because all these tableaux are such that $P^{AII}(T)=S^{H,\mu}$. Hence $\wt_{{\mathfrak{\k}}}(T)=\mu=\wt_{\widehat{\mathfrak{\g}}}(\Phi^{-1}(T))$.

\begin{obs}\label{haha}This example clarifies why we  impose the linear inequality $m_{123\overline 1}\le m_{12\overline 2}$ in the following $\k$-highest weight tableau. In other words, it explains that the linear inequality imposes that the recording  tableau belongs to $Rec^{1,0}$.
Let $n=3$, $\lambda=(6,4,3,1)$, $\mu=(3,2,1)$  and

 $S=\YT{0.15in}{}{
 {1,\mathbf{1},1,    {{ 1}},1 ,1   },
 {\mathbf{2},{2},{\bf 2}, { \overline 1}},
 {\mathbf{3},  {\overline 2},\overline 1},
 {\overline 1},
}\in SST_6^{{\widehat{\mathfrak{g}}-dom}}(\lambda,\mu)
$ that satisfies the linear inequalities
$$m_{123\overline 1}\le m_{12\overline 2},\quad m_{12\overline 1}\le m_{1}$$

Then
$$\Phi(S)=\YT{0.15in}{}{
 {1,1,    {{ 1}},2, 2,2   },
 {{2},{2},{2}, { 3}},
 {{3},  { 3}, 6},
 {4},
}\in SST_6^{{{\mathfrak{k}}}}(\lambda,\mu),
$$

and $Q^{AII}(\Phi(S))\in Rec_6^{1,0}(\lambda,\mu)$.

However if we ignore the inequality $m_{123\overline 1}\le m_{12\overline 2}$, then we may consider

$R=\YT{0.15in}{}{
 {1,\mathbf{1},1,    {{ 1}},1    },
 {\mathbf{2},{2}, { \overline 1}},
 {\mathbf{3}, \overline 1},
 {\overline 1},
}\in SST_6^{{\widehat{\mathfrak{g}}-dom}}(\lambda-\varpi_3,\mu')
$ with $\mu'=(2,2,1)$ and linear inequalities $m_{123\overline 1}+m_{12\overline 1}\le m_{1}$
where
$$\Phi(R)=\YT{0.15in}{}{
 {1,1,    {{ 1}},2, 2  },
 {{2},{2},{2}},
 {{3},  { 3}},
 {6},
}\in SST_6^{{{\mathfrak{k}}}}(\lambda-\varpi_3,\mu'),
$$ but $Q^{AII}\Phi(R)\notin  Rec_6^{1,0}(\lambda-\varpi_3,\mu')$ because

$\redu(1,2,3,6)=(3,6)\Rightarrow$ that the recording tableau is not $1$-$0$- slack. More precisely,
$$\Phi(R)=\YT{0.15in}{}{
 {1,1,    {{ 1}},2, 2  },
 {{2},{2},{2}},
 {{3},  { 3}},
 {6},
}\underset{P^{AII}}\rightarrow
\YT{0.15in}{}{
 {1,    {{ 1}},2, 2  },
 {{2},{2}},
 {{3},  { 3}},
 {6},
}\underset{P^{AII}}\rightarrow
\YT{0.15in}{}{
 {1,    2, 2  },
 {{2},{3}},
 {{3}},
 {6},
}\underset{P^{AII}}\rightarrow
\YT{0.15in}{}{
 {   2, 2  },
 {{3},  { 3}},
 {6},
}=S^{H,\mu'}=P^{AII}(\Phi(R)
$$
\end{obs}

\begin{thm}\label{thm:1} Let $n=3$ and $S^{H,\mu}$ the $\mathfrak{k}$-highest weight symplectic tableau  in $ SpT_6(\mu)$. Consider the $\mathfrak{k}$-highest weight tableaux in  ${\LRAII^{AII}}^{-1}(\{S^{H,\mu}\}\times {Rec}_{6}^{1,0}(\lambda/\mu))\subseteq \SST^{\mathfrak{k}-hw}(\lambda,\mu)\subseteq SST_6(\lambda)$   as displayed in Theorem \ref{thm:n=3slack1}. Then
 for $T$ an $\mathfrak{k}$-highest weight tableau in ${\LRAII^{AII}}^{-1}(\{S^{H,\mu}\}\times {Rec}_{6}^{1,0}(\lambda/\mu))$,

$$\Phi^{-1}(T)\in SST_6^{{\widehat{\mathfrak{g}}-dom}}(\lambda,\mu)\subseteq SST_6(\lambda)$$
 is a $\widehat{\mathfrak{g}}$-dominant tableau of one of either forms below (recall $\overline i=2n-i+1$ and $n=3$) satisfying the linear inequalities as follows:

 \begin{enumerate}
\item For $\ell(\lambda)=6$,

\begin{align}&1\le m_{1{ 2}3\overline 3\overline 2\overline 1}\nonumber\\
&m_{1{\oo 2}3\overline 3\overline 1}\le m_{\oo 1}, \quad m_{12{\red 3}\overline 2\overline 1}\le m_{\red 12},\quad m_{1\bro 2{3}\overline 1}\le m_{\bro 12\overline 2}\label{ghwineq}
\end{align}

where
\begin{align}\Phi^{-1}(T)=
\YT{0.15in}{}{
 {1,\cdots,1,               1,\cdots,1,  {1},\cdots,{1},1,\cdots, 1, 1,\cdots, 1,\red 1,\cdots,{\red 1},\oo 1,\cdots,\oo 1},
 {2,\cdots,2,               {2},\cdots,2,     { \oo 2},\cdots,\oo 2,{2},\cdots,{2},{2},\cdots,{2},\red 2,\cdots,\red 2},
 {3,\cdots,3,               {3},\cdots,{3},   3,\cdots, 3,{\red 3},\cdots,{\red 3},{3},\cdots,{3}},
 {\overline 3,\cdots,\overline 3, \overline 3,\cdots,\overline 3,  \overline 3,\cdots,\overline 3,\overline 2,\cdots,\overline 2},
 {\overline 2,\cdots,\overline 2, \overline 2,\cdots,\overline 2, \overline 1 ,\cdots,\overline1,\overline 1,\cdots,\overline 1},
 {\overline 1,\cdots,\overline 1},
}\label{ghtype31}
\end{align}

or
\begin{align}\Phi^{-1}(
T)=\YT{0.15in}{}{
 {1,\cdots,1,{1},\cdots,1,{1},\cdots,1, {1},\cdots,1,{1},\cdots,{1},1, \cdots,1,\bro 1,\cdots, \bro 1,\red 1,\cdots,\red 1, \oo 1,\cdots,{\oo 1}},
 {2,\cdots,2,{2},\cdots,2, {\oo 2},\cdots,\oo 2, {2},\cdots,2,2,\cdots,2,{\bro2},\cdots,\bro 2,\bro 2,\cdots,{\bro 2},\red 2,\cdots,\red 2},
 {3,\cdots,3,{3},\cdots,3,{3},\cdots,3,  \red 3,\cdots,\red 3,\red 3,\cdots,\red 3, {\bro 3},\cdots,{ \bro 3},\bro \overline  2,\cdots,{\bro\overline 2}},
 {\overline 3,\cdots,\overline 3,\overline 3,\cdots,\overline 3, {\overline 3},\cdots,\overline 3, {\overline 2},\cdots,\overline 2,\overline 2,\cdots,\overline 2, \overline 1,\cdots,\overline 1},
 {\overline 2,\cdots,\overline 2,\overline 2,\cdots,\overline 2,\overline 1,\cdots,\overline 1, \overline 1 ,\cdots,\overline 1,\overline 1,\cdots,\overline 1},
 {\overline 1,\cdots,\overline 1},
 }\label{ghtype33}
 \end{align}

or

 \begin{align}\Phi^{-1}(T)
=\YT{0.15in}{}{
 {1,\cdots,1,{1},\cdots,1, {1},\cdots,1,{1},\cdots,1,1,\cdots,1,1,\cdots,1, \bro 1,\cdots, \bro 1, \red 1,\cdots,{\red 1},\oo 1,\cdots,{\oo 1}},
 {2,\cdots,2,{2},\cdots,2,{\oo 2},\cdots,\oo 2, { 2},\cdots, 2, \bro 2,\cdots,\bro 2,2,\cdots,2,\bro 2,\cdots,{\bro 2},\red 2,\cdots,\red 2},
 {3,\cdots,3,{3},\cdots,3, 3,\cdots,3,\red 3,\cdots,\red 3,   \bro 3,\cdots,\bro 3,\overline 2,\cdots,\overline 2, \bro \overline 2,\cdots,\bro\overline 2},
 {\overline 3,\cdots,\overline 3,\overline 3,\cdots,\overline 3, \overline 3,\cdots,\overline 3,\overline 2,\cdots,\overline 2,\overline 1,\cdots,\overline 1,\overline 1,\cdots ,\overline 1},
 {\overline 2,\cdots,\overline 2,\overline 2,\cdots,\overline 2, {\overline 1},\cdots,\overline 1,{\overline 1},\cdots,\overline 1},
 {\overline 1,\cdots,\overline 1},
}\label{ghtype34}
\end{align}

\item If $\ell (\lambda)=5$, Lemma \ref{lem:basic}, $(2)$, $(b)$ guarantees that the  tableaux are the same as above with $M=0$ in $(123\overline{321})^M$ and the same linear inequalities \eqref{ghwineq} since $m_{123\overline{321}}\ge 1$ does not impose { other constraints in the previous tableaux besides the positivity of the multiplicity}.
    { Here similarly we have the restriction of length five.}

\item If $\ell (\lambda)=4$, the following linear inequalities hold

\begin{align}\label{gineql=4}m_{12\overline 1}\le m_1,\quad 1\le m_{123\overline 1}\le m_{12\overline 2}
\end{align}

where
\begin{align}\label{gl=41}\Phi^{-1}(T)=
&\YT{0.15in}{}{
 {1,\cdots,1,1,\cdots,1,  1,\cdots, 1,   1,\cdots, 1,{1},\cdots,1,       1,\cdots,{ 1}},
 {2,\cdots,2,2,\cdots,2,{2},\cdots,2,2,\cdots,2,   { \overline 1},\cdots, \overline 1},
 {3,\cdots,3, \overline 2,\cdots, \overline 2, \overline 2,\cdots,\overline 2,  \overline 1,\cdots,\overline 1 },
 {\overline 1,\cdots,\overline 1},
}
\end{align}

or

\begin{align}\Phi^{-1}(T)=&\YT{0.15in}{}{
 {1,\cdots,1,1,\cdots,1, 1,\cdots,1,    1,\cdots, 1,{1},\cdots,1,    1,\cdots, 1, 1,\cdots, 1},
 {2,\cdots,2,{2},\cdots,2,{2},\cdots,2,2,\cdots,2, 2,\cdots,2, { \overline 1},\cdots, \overline 1},
 {3,\cdots,3,  3,\cdots,3,\overline 2,\cdots,\overline 2,  \overline 2,\cdots,\overline 2, \overline 1,\cdots, \overline 1},
 {\overline 1,\cdots,\overline 1},
}
\end{align}

\item If $\ell (\lambda)=3$, one has the linear inequality below { plus the length condition on $\lambda$}

\begin{align}m_{12\overline 1}\le m_1
\end{align}

and

\begin{align}\label{gl=33}\Phi^{-1}(T)=
&\YT{0.15in}{}{
 {1,\cdots,1,1,\cdots,1,     1,\cdots, 1,    1,\cdots, 1,   1,\cdots,{ 1}},
 {2,\cdots,2,{2},\cdots,2,2,\cdots,2,    { \overline 1},\cdots, \overline 1},
 {3,\cdots,3, \overline 2,\cdots,\overline 2,  \overline 1,\cdots,\overline 1},
}
\end{align}

\item $\ell (\lambda)=2$,

\begin{align}\label{gl=2}\Phi^{-1}(\eqref{typel=2})=
&\YT{0.15in}{}{
 {1,\cdots,1,1,\cdots,1,     1,\cdots, 1 },
 {2,\cdots,2,    { \overline 1},\cdots, \overline 1},
}
\end{align}
with no constraints on the multiplicity of the columns  besides the length condition on $\ell (\lambda)$.

\item $\ell(\lambda)=1$ ($\Rightarrow \mu=1$)

$$\Phi^{-1}(S^{H,\mu})=\Phi^{-1}(2^M)=(1^M), \quad M\ge 0.$$
\end{enumerate}
\end{thm}

\begin{proof} $(1)$
By Corollary \ref{prop:promn3}, it is an easy exercise with the promotion operations in Definition \ref{def:promotion}  that applying $\Phi^{-1}=\pr^{-1}_{3,4}\circ \pr_{3,5}^{-1}\circ \pr^{-1}_{1,6}$ to ${\LRAII^{AII}}^{-1}\{(S^{H,\mu},Q): Q \mbox { $1$-$0$  slack recording tableau } \} \subseteq SST_6$ consisting of the $\mathfrak{k}$-highest weight tableaux \cite[Corollary 6]{azreco}
we get the  tableaux
 \eqref{ghtype31}, \eqref{ghtype33}, \eqref{ghtype34}. We want to check that they belong to
 $SST_6^{{\widehat{\mathfrak{g}}-dom}}(\lambda,\mu)=\domres(\lambda,\mu)$ by showing that they are characterized by the linear inequalities \eqref{ghwineq}.

  Due to the embedding ${\mathfrak{gl}}_{6}\subset {\mathfrak{gl}}_{12}$,  we show that
 they are $\widehat{\mathfrak{gl}}_{12}$-dominant. Let us consider those tableaux in $SSTY_{C_6}$. By Definition \ref{def:cancelation} (\cite[Defintion 50]{schumanntorres} ) the obtained tableaux have the cancellation property and satisfy the assumptions of \cite[Lemma 49]{schumanntorres}, here Proposition \ref{prop:gineq}. Using the linear inequalities \eqref{ghwineq}

 \begin{align}m_{1{\oo 2}3\overline 3\overline 1}\le m_{\oo 1}, \quad m_{12{\red 3}\overline 2\overline 1}\le m_{\red 12},\quad m_{1\bro 2{3}\overline 1}\le m_{\bro 12\overline 2}\nonumber
 \end{align}
 we can verify that the inequalities \eqref{ineqdom} in Proposition \ref{prop:gineq} (\cite[ Proposition 56]{schumanntorres})  hold, for $1\le i\le 5$, $1\le l\le 6$.

 Note the entries $\i$ appear after row $i$  and rows one and two have  just entries $1$ respectively   $2$. There are no entries $4\le i\le 6$ nor $\overline 6\le\i\le \overline 4$. We show the inequalities \eqref{ineqdom} for tableaux
 \eqref{ghtype31}, \eqref{ghtype33}.
  The other case is similar.

Tableau \eqref{ghtype31}:
For $i=1$
$$(1,1)-(2,2)=m_1$$
$$(1,1)-(2,2)-((\overline 1,5)-(\overline 2,4))=m_1- m_{1{\oo 2}3\overline 3\overline 1}\ge 0 , \mbox{ because $m_{1{\oo 2}3\overline 3\overline 1}\le m_{\oo 1}$}$$

$$(1,1)-(2,2)-((\overline 1,5)-(\overline 2,4))-((\overline 1,6)-(\overline 2,5))=m_1- m_{1{\oo 2}3\overline 3\overline 1}-(-m_{123\overline 3\overline 2})\ge 0$$

For $i=2$
$$(2,2)-(3,3)=m_{12}$$
$$(2,2)-(3,3)-((\overline 2,4)-(\overline 3,3))=m_{12}-m_{1{ 2}{\red 3}\overline 2\overline 1}\ge 0, \mbox{ because $ m_{1{ 2}{\red 3}\overline 2\overline 1}\le m_{\red12}$}$$

$$(2,2)-(3,3)-((\overline 2,4)-(\overline 3,3))- ((\overline 2,5)-(\overline 3,4))-((\overline 2,6)-(\overline 3,5))=$$
$$=m_{12}-m_{1{ 2}{\red 3}\overline 2\overline 1}-(-m_{123\overline 3\overline 1})-(-(\overline 3,5)))\ge 0$$

For $i=3$,
$$(3,3)-(4,4)-((\overline 3,3)-(\overline 4,2))=(3,3)\ge 0$$
$$(3,3)-(4,4)-((\overline 3,3)-(\overline 4,2))-((\overline 3,4)-(\overline 4,3))=(3,3)-(\overline 3,4)\ge 0$$

Tableau \eqref{ghtype33}:
For $i=1$
$$(1,1)-(2,2)=m_1$$
$$(1,1)-(2,2)-((\overline 1,4)-(\overline 2,3))=m_1- (m_{1{2}3\overline 1}-m_{12\overline 2})=$$
$$=m_1+ (m_{12\overline 2}-m_{1{2}3\overline 1})\ge 0 , \mbox{ because $m_{12\overline 2}\ge m_{1{2}3\overline 1}$}$$

$$(1,1)-(2,2)-((\overline 1,4)-(\overline 2,3))-((\overline 1,5)-(\overline 2,4)=$$
$$=m_1+ (m_{12\overline 2}-m_{1{2}3\overline 1})-m_{123\overline 3\overline 1}=$$
$$=m_1-m_{123\overline 3\overline 1}+ (m_{12\overline 2}-m_{1{2}3\overline 1})\ge 0, \mbox{ because $m_{1}\ge m_{123\overline 3\overline 1} $}$$

$$(1,1)-(2,2)-((\overline 1,4)-(\overline 2,3))-((\overline 1,5)-(\overline 2,4)-((\overline 1,6)-(\overline 2,5))=$$
$$=m_1-m_{123\overline 3\overline 1}+ (m_{12\overline 2}-m_{1{2}3\overline 1})-(-m_{123\overline 3\overline 2})\ge 0$$

For $i=2$
$$(2,2)-(3,3)=m_{12}+m_{12\overline 2}$$
$$(2,2)-(3,3)-((\overline 2,3)-(\overline 3,2))=m_{12}+m_{12\overline 2}-(\overline 2,3) =m_{12}
+m_{12\overline 2}-m_{12\overline 2}=m_{12}$$

$$(2,2)-(3,3)-((\overline 2,3)-(\overline 3,2))-((\overline 2,4)-(\overline 3,3))=$$
$$=m_{12}-((\overline 2,4)-(\overline 3,3))=m_{12}-(\overline 2,4)=$$
$$=m_{12}-m_{123\overline 2\overline 1}\ge 0, \mbox{ because $m_{12}\ge m_{1{2}3\overline 2 \overline  1}$}$$

$$(2,2)-(3,3)-((\overline 2,3)-(\overline 3,2))-((\overline 2,4)-(\overline 3,3))-((\overline 2,5)-(\overline 3,4))=$$
$$ =m_{12}-m_{123\overline 2\overline 1}-((\overline 2,5)-(\overline 3,4))=m_{12}-m_{123\overline 2\overline 1}-(-m_{123\overline 3\overline 1})\ge 0$$

$$(2,2)-(3,3)-((\overline 2,3)-(\overline 3,2))-((\overline 2,4)-(\overline 3,3))-((\overline 2,5)-(\overline 3,4))-(\overline 2,6)-(\overline 3,5))=$$
$$=m_{12}-m_{123\overline 2\overline 1}+m_{123\overline 3\overline 1}-(\overline 2,6)-(\overline 3,5))=$$
$$=m_{12}-m_{123\overline 2\overline 1}+m_{123\overline 3\overline 1}-(-(\overline 3,5))\ge 0$$

For $i=3$
$$(3,3)-(4,4)=(3,3)$$

$$(3,3)-(4,4)-((\overline 3,3)-(\overline 4,2))=(3,3)=3,3)-(4,4)-((\overline 3,3)-(\overline 4,2))=(3,3)\ge 0$$

$$(3,3)-(4,4)-((\overline 3,3)-(\overline 4,2))-((\overline 3,4)-(\overline 4,3))=(3,3)-(\overline 3,4)\ge 0$$

$$(3,3)-(4,4)-((\overline 3,3)-(\overline 4,2))-((\overline 3,4)-(\overline 4,3))-((\overline 3,5)-(\overline 4,4))=$$
$$=(3,3)-(\overline 3,4)\ge 0$$

$$(3,3)-(4,4)-((\overline 3,3)-(\overline 4,2))-((\overline 3,4)-(\overline 4,3))-((\overline 3,5)-(\overline 4,4))-((\overline 3,6)-(\overline 4,5))=(3,3)-(\overline 3,4)\ge 0$$

\end{proof}

\subsection{Case n=4}

Let $n=4$. Thanks to Lemma \ref{lem:basic}, $(1)$, $(2)(a)$, we now generalize \cite[Corollary 8]{azreco} for any $\ell(\mu)\le \ell(\lambda)< 2n-1$ and $\ell(\mu)\le n$, and then
we compute

$$\Phi^{-1}({\LRAII^{AII}}^{-1}(\{S^{H,\mu}\}\times {Rec}_{8}^{1,0}(\lambda/\mu))\subset SST_8^{{\widehat{\mathfrak{g}}-dom}}(\lambda,\mu).$$
 where $\Phi=\pr_{1,8}\circ \pr_{3,7}\circ\pr_{3,6}$.

\begin{thm}\label{thm:2} Let $n=4$ and $S^{H,\mu}\in SpT_8(\mu)$.
The set of $\mathfrak{k}$-highest weight tableaux defined by
$${\LRAII^{AII}}^{-1}(\{S^{H,\mu}\}\times {Rec}_{8}^{1-0}(\lambda/\mu))\subseteq SST_8^{\mathfrak{k}-hw}(\lambda,\mu)\subseteq SST_8(\lambda)$$ is displayed below for any $\ell(\mu)\le \ell(\lambda)< 2n-1$ and $\ell(\mu)\le n$.
For each
$T\in{\LRAII^{AII}}^{-1}(\{S^{H,\mu}\}\times  {Rec}_{8}^{1,0}(\lambda/\mu)),$
$$ \Phi^{-1}(T)\in SST_8^{{\widehat{\mathfrak{g}}-dom}}(\lambda,\mu)\subseteq SST_8(\lambda)$$ is a $\widehat{\mathfrak{g}}$-dominant tableau   as displayed below (recall $\overline i=2n-i+1$ and $n=4$) satisfying the corresponding linear inequalities.
\begin{enumerate}
\item  If $\ell(\lambda)=8$, the $\k$-highest weight tableaux are the ones in \cite[Corollary 8]{azreco} with the linear inequalities

\begin{align}&m_{12{\bf 3}5678}+m_{12{\bf 3}567}\le m_2, \quad 0\le m_{1234{\bf 6}78}+m_{1234{\bf 6}7}\le m_{23},\quad 0\le m_{123456\bf 7}\le m_{236} \nonumber\\
  &m_{1234568}=m_{{\bf 2}34567}+ m_{12{\bf 3}567}+ m_{1234{\bf 6}7}\label{ineqCorollary 8azreco}
\end{align}

and the corresponding $\widehat{\mathfrak{g}}$-dominant tableaux with linear inequalities are


\vskip0.4cm
\begin{figure}[htbp] \centering \includegraphics[width=0.9\textwidth]{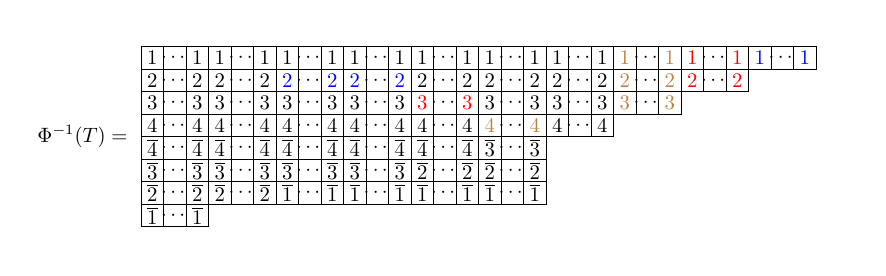}
\label{fig:figura1}
 \end{figure}
\begin{align}&1\le m_{1{ 2}3 4\overline 4\overline 4\overline 3\overline 2\overline 1},\nonumber\\
&m_{1{\oo 2}3 4\overline 4\overline 3\overline 1}\le m_{\oo 1}, \quad m_{12{\red 3}4\overline 4 \overline 2\overline 1}\le m_{\red 12},\quad m_{1 2{3}{\bro 4}\overline 3\overline 2\overline 1}\le m_{ \bro 123},\quad m_{12{\ora 34}\overline 2\overline 1}\le m_{123\overline 3}
\end{align}
or

\begin{figure}[htbp] \centering \includegraphics[width=0.8\textwidth]{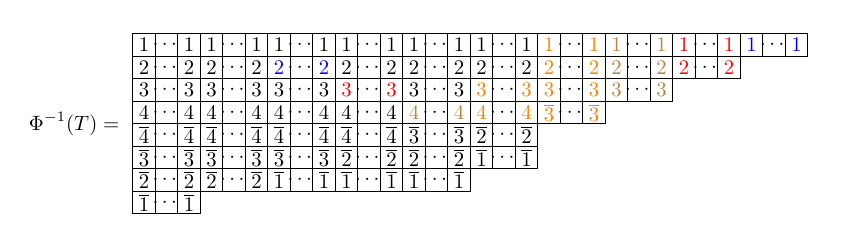}
\label{fig:figura2}
 \end{figure}
\begin{align}&1\le m_{1{ 2}3 4\overline 4\overline 4\overline 3\overline 2\overline 1},\nonumber\\
&m_{1{\oo 2}3 4\overline 4\overline 3\overline 1}\le m_{\oo 1}, \quad m_{12{\red 3}4\overline 4 \overline 2\overline 1}\le m_{\red 12},\quad m_{1 2{3}{\bro 4}\overline 3\overline 2\overline 1}\le m_{ \bro 123},\quad m_{12{\ora 34}\overline 2\overline 1}\le m_{123\overline 3}
\end{align}
or
\begin{figure}[htbp] \centering \includegraphics[width=0.8\textwidth]{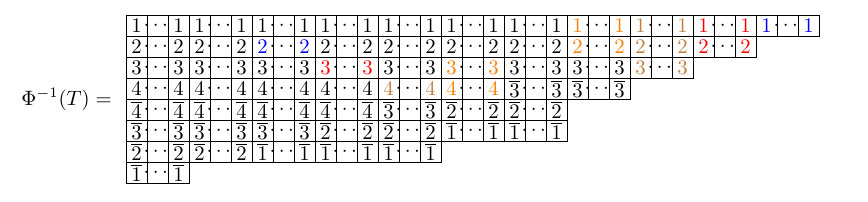}
\label{fig:figura3}
 \end{figure}
\begin{align}&1\le m_{1{ 2}3 4\overline 4\overline 4\overline 3\overline 2\overline 1},\nonumber\\
&m_{1{\oo 2}3 4\overline 4\overline 3\overline 1}\le m_{\oo 1}, \quad m_{12{\red 3}4\overline 4 \overline 2\overline 1}\le m_{\red 12},\quad m_{1 2{3}{\bro 4}\overline 3\overline 2\overline 1}\le m_{ \bro 123},\quad m_{12{\ora 34}\overline 2\overline 1}\le m_{123\overline 3}
\end{align}

\item For $\ell(\lambda)=7$, Lemma \ref{lem:basic} guarantees that the $\k$-highest weight tableaux and respectively the $\widehat\g$-dominant tableaux are the same as above with $M=0$ in $(12345678)^M$ respectively in  $(1234\overline{4321})^M$. 

\item \label{n=8l=6notdone}For $\ell(\lambda)=6$, one has $1\le m_{123456}$ 


     $$(A)$$

 \begin{align}\label{n=4type4160x}
&S^{4,\dots,4,5,\dots,5}=T=\YT{0.15in}{}{
 {1,\cdots,{1}, 1,\cdots,{1}, 1,\cdots,{1}, 2,\cdots, 2,2,\cdots,{2},2,\cdots,{2},{2},\cdots,{2}},
 {2,\cdots,{2}, {2},\cdots,2, 2,\cdots,{2},{3},\cdots,{3},3,\cdots,3,3,\cdots,3},
 {{3},\cdots,{3},{3},\cdots,{3}, 3,\cdots,{3}, {6},\cdots,{6},{6},\cdots,{6}},
 {{4},\cdots,{4},{4},\cdots,{4}, 4,\cdots,{4},  7,\cdots,7},
 {{5},\cdots,{5},7,\cdots,7 },
 {{6},\cdots,{6}},
}
\\
&\redu^{-1}(7)=12347 \quad m_{12347}\le m_{236}
\end{align}


$$(B)$$

\begin{align}\label{type416} S^{1\dots 1 2\dots 2 3\dots 3}
&T=\YT{0.15in}{}{
 {{1},\cdots,{1},{1},\cdots,{1}, {1},\cdots,1,\red{\bf \circled{\bf 2}},\cdots,\red{\bf \circled{\bf 2}},2,\cdots, 2, 2,\cdots, 2,2,\cdots,{2},2,\cdots,{2}},
 {{2},\cdots,{2},   {2}, \cdots,{2},         \blue{2},\cdots,2,\red{3},\cdots,3,3,\cdots,3,{3},\cdots,{3},3,\cdots,3},
 {{3},\cdots,{3},   {3},\cdots,{3},       \blue{ \circled{\bf 3}},\cdots,\blue{ \circled{\bf 3}},\red{4},\cdots,4,6,\cdots,6,{6},\cdots,{6}},
 {{4},\cdots,{4},{4},\cdots,{4},        \blue{5},\cdots,5,\red{5},\cdots,5,7,\cdots,7},
 {{5 },\cdots,{5},\ora{\bf\circled {\bf 6}},\cdots,\ora{\bf\circled {\bf 6}},  \blue{6},\cdots,6,\red{6},\cdots,6},
 {{6},\cdots,{6}},
 }\\
 & m_{12356}\le m_2 \mbox{ and }0\le m_{12346}\le m_{23}
\end{align}

\begin{align}\label{l=6n=4type31XXXY} S^{1\dots 1 2\dots 2 3\dots 3,4\dots 4, 5\dots 5}
&T=\YT{0.15in}{}{
 {1,\cdots,{1},1,\cdots,{1},1,\cdots,{1}, {1},\cdots,{1}, \red 2,\cdots,\red 2, 2,\cdots, 2,2,\cdots,{2},2,\cdots,{2},2,\cdots,{2},{2},\cdots,{2}},
 {2,\cdots,{2},{2},\cdots,2,{2},\cdots,2, { 2},\cdots,2,{3},\cdots,{3},{3},\cdots,{3},{3},\cdots,{3},3,\cdots,3,3,\cdots,3},
 {3,\cdots,{3},{3},\cdots,{3},{3},\cdots,{3},   \blue 3,\cdots, \blue 3,{4},\cdots,{4},{4},\cdots,{4},{6},\cdots,{6},{6},\cdots,{6}},
 {4,\cdots,{4},{4},\cdots,{4},{4},\cdots,{4}, {5},\cdots,{5}, {5},\cdots,{5},{5},\cdots,{5}, 7,\cdots,7},
 {{5},\cdots,{5},6,\cdots, 6,7,\cdots, 7,  {7},\cdots,{7},{7},\cdots,{7}},
{6,\cdots, 6},
}\\
&m_{12357}\le m_2,\quad m_{12347}+m_{12357}+m_{23457}\le m_{236},\quad\\
& 0\le m_{12346}- m_{12357}-m_{23457}-m_{2345}\le m_{23}
\end{align}


$$(C)$$

\vskip 2cm

 \begin{figure}[htbp] \centering \includegraphics[width=0.8\textwidth]{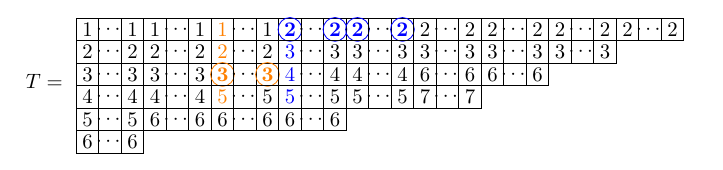}
\label{fig:figuraEXl=6n=4type32v4}
 \end{figure}
 \begin{align}
&m_{12356}+m_{1235}\le m_2 \mbox{ and }0\le m_{12346}-m_{1235}-m_{2345}\le m_{23}
\end{align}

or\begin{align}\label{n=4type36}
&T=\YT{0.15in}{}{
 {{1},\cdots,1,{1},\cdots,1,\bro{1},\cdots,1,1,\cdots,1,\ora{1},\cdots,1,   \blue{\circled{\bf 2}},\cdots,\blue{\circled{\bf 2}}, 2,\cdots, 2, 2,\cdots,{2},2,\cdots,{2},,2,\cdots,{2}},
 {{2},\cdots,2,{2},\cdots,2,\bro{2},\cdots,2,2,\cdots,2, \ora {2},\cdots,2,   \blue{3},\cdots,3,3,\cdots,{3},3,\cdots,3,3,\cdots,3},
 {{3},\cdots,3,\bro{3},\cdots,3,\bro{3},\cdots,3,  3,\cdots,3,\ora{\circled{\bf 3}},\cdots,\ora{\circled{\bf 3}},\blue{4},\cdots,4, 6,\cdots,{6},6,\cdots,{6}},
 {{4},\cdots,4,{4},\cdots,4,\bro{4},\cdots,4, 4,\cdots,4,\ora{5},\cdots,5, \blue{5},\cdots,5,7,\cdots,{7}},
 {{5},\cdots,5,\bro{\circled{\bf 6}},\cdots,\bro{\circled{\bf 6}},6,\cdots,6},
{{6},\cdots,6},
}\nonumber\\
&m_{12356}+m_{1235}\le m_2 \mbox{ and }0\le m_{12346}-m_{1235}-m_{2345}\le m_{23}
\end{align}

Under the transformation $\Phi^{-1}$  on those tableaux $(A)$ one has respectively with corresponding linear inequalities
$$\Phi^{-1}(A)$$
%
\begin{figure}[htbp] \centering \includegraphics[width=0.7\textwidth]{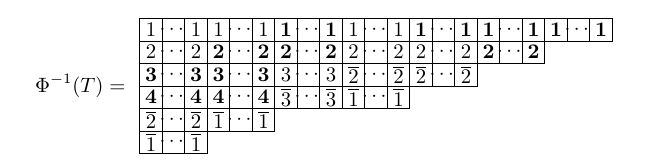}
\label{fig:figuraEXgwn=4type4160xv4}
 \end{figure}
\begin{align}
1\le m_{1234\overline{21}}\le m_{123\overline 3}\quad  m_{1234\overline 1}\le m_{12\overline 2}
  \end{align}

 $$\Phi^{-1}(B)$$
  \begin{figure}[htbp] \centering \includegraphics[width=0.7\textwidth]{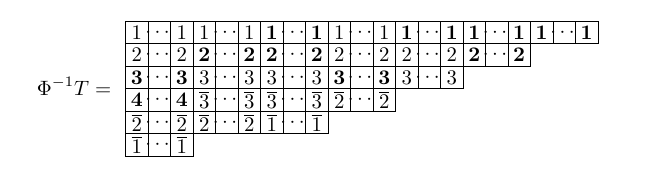}
\label{fig:figuraEXgtype416v4}
 \end{figure}
 \begin{align}
1\le m_{1234\overline{21}}\le m_{123\overline{31}}\le m_{123\overline{2}}\le m_{12}
\label{Bineqgtype416*}
 \end{align}


\vskip2cm
 \begin{figure}[htbp] \centering \includegraphics[width=0.7\textwidth]{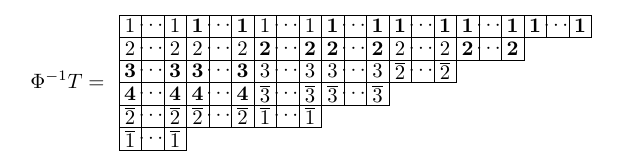}
\label{fig:figuraEXgl=6n=4type31Xv4}
 \end{figure}
 \begin{align}
1\le m_{1234\overline{21}}\le m_{123\overline 3},\quad
  m_{1234\overline 2} \le m_{123\overline{31}}\le m_{12\overline 2}
 \end{align}
\vskip3cm
 $$\Phi^{-1}(C)$$


 \begin{figure}[htbp] \centering \includegraphics[width=0.7\textwidth]{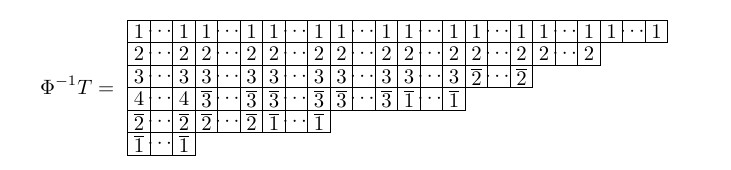}
\label{fig:figuraineqEXgl=6n=4type32v4}
 \end{figure}
 \begin{align}
 1\le m_{1234\overline{2}\overline{1}}\le m_{123\overline 3}\quad m_{123\overline{31}}\le m_{1}, \quad  \quad m_{123\overline{1}}\le m_{12\overline 2}\label{Bineqgl=5n=4type31X***}
 \end{align}

    \item \label{n=8l=5notdone}If $\ell(\lambda)=5$, we apply Lema \ref{lem:basic} $(1)$ and $(2)$ $(a)$. The $\k$-highest weight tableaux are the same as above
with the columns $(123456)^M$, $M>0$, deleted on the left, together
with the same inequalities { plus the length condition.}
Under the transformation $\Phi^{-1}$  on those tableaux obtained from $(A)$ one has

$$\Phi^{-1}(A)$$
  \begin{align}\label{gwn=4type30}
&\Phi^{-1}(T)=\YT{0.15in}{}{
 {1,\cdots,{1}, \bf 1,\cdots,{1}, 1,\cdots, 1,1,\cdots,{1},1,\cdots,{1},{1},\cdots,{1}},
 {{2},\cdots,2, \bf 2,\cdots,{2},{2},\cdots,{2},2,\cdots,2,2,\cdots,2},
 {{3},\cdots,{3},\bf  3,\cdots,{3}, {\overline 2},\cdots,{\overline 2},{\overline 2},\cdots,{\overline 2}},
 {{4},\cdots,{4}, \bf 4,\cdots,{4},  \overline 1,\cdots,\overline 1},
 {\overline 1,\cdots,\overline 1 },
}
 \end{align}
where \begin{align}1\le m_{1234\overline 1}\le m_{12\overline 2}
  \end{align}

$$\Phi^{-1}(B)$$
\begin{align}\label{gl=5type416}
&\Phi^{-1}T=\YT{0.15in}{}{
 {{\bf 1},\cdots,{\bf 1},  1,\cdots, 1, \bf 1,\cdots, \bf1,\bf 1,\cdots, \bf 1,\bf 1,\cdots,\bf 1,\bf 1,\cdots,{\bf 1} },
 {{2},\cdots,{2},           {\bf 2},\cdots,\bf 2,{ 2},\cdots, 2,\bf 2,\cdots,\bf 2,{\bf 2},\cdots,{\bf2} },
 {{\bf 3},\cdots,{\bf 3},          3,\cdots,3,\bf 3,\cdots,\bf 3,{\bf 3},\cdots,{\bf 3}},
 {{\bf 4},\cdots,\bf 4,        \overline 3,\cdots,\overline 3,\overline 2,\cdots,\overline 2},
 {\overline 2,\cdots,\overline 2,  \overline 1,\cdots,\overline 1},
 }
 \end{align}
 \begin{align} 1\le m_{1234\overline{2}}\le m_{123\overline{31}}\le   m_{123\overline{2}}\le m_{12}\label{Bineqgl=5type416*}
 \end{align}
 \begin{align}\label{gl=5n=4type31X}
&\Phi^{-1}T=\YT{0.15in}{}{
 {{\bf 1},\cdots,{\bf 1}, {1},\cdots,1, 1,\cdots, 1,\bf 1,\cdots, \bf 1, \bf 1,\cdots,\bf  1,\bf 1,\cdots,{\bf 1},\bf 1,\cdots,{\bf 1}},
 {   {2}, \cdots,{2},         {\bf 2},\cdots,\bf 2,{\bf 2},\cdots,\bf 2,\bf 2,\cdots,\bf 2,{2},\cdots,{2},\bf 2,\cdots,\bf 2},
 {   {\bf 3},\cdots,{\bf 3},        {\bf 3},\cdots,{\bf 3},3,\cdots,3, 3,\cdots, 3,{\overline 2},\cdots,{\overline 2}},
 {{\bf 4},\cdots,\bf 4, \bf 4,\cdots,\bf 4,       \overline 3,\cdots,\overline 3,\overline 3,\cdots,\overline 3},
 {\overline 2,\cdots,\overline 2,  \overline 1,\cdots,\overline 1,\overline 1,\cdots,\overline 1},
 }
 \end{align}

\begin{align}
1\le  m_{1234\overline{2}}\le m_{123\overline 3}\quad m_{123\overline{31}}\le m_{1}, \quad  \quad m_{1234\overline{1}}\le m_{12\overline 2}
 \end{align}

$$\Phi^{-1}C$$

\begin{align}\label{gCl=5n=4type32}
&\Phi^{-1}T=\YT{0.15in}{}{
 {{\bf 1},\cdots,\bf 1,\bf 1,\cdots,\bf 1,1,\cdots,1,\bf 1, \cdots,\bf 1, 1,\cdots, 1,\bf 1,\cdots, \bf 1, 1,\cdots,1,1,\cdots,{1}},
 {{2},\cdots,2,  2,\cdots,2,{\bf 2},\cdots,\bf 2,{\bf 2},\cdots,\bf 2,\bf 2,\cdots,{\bf 2},2,\cdots,2,2,\cdots,2},
 {{\bf 3},\cdots,\bf 3,  3,\cdots, 3,{3},\cdots,3, 3,\cdots,3, \bf 3,\cdots,{\bf 3},\overline 2,\cdots,{\overline 2}},
 {{\bf 4},\cdots,\bf 4, {\overline 3},\cdots,\overline 3,{\overline 3},\cdots,\overline 3, \overline 3,\cdots,\overline 3,\overline 1,\cdots,{\overline 1}},
 {{\overline 2},\cdots,\overline 2,  {\overline 2},\cdots,\overline 2,{\overline 1},\cdots,\overline 1},
 }
 \end{align}

\begin{align}
1\le m_{1234\overline{2}}\le m_{123\overline 3},\quad m_{123\overline{31}}\le m_{1}, \quad  \quad m_{123\overline{1}}\le m_{12\overline 2}
 \end{align}

        \item If $\ell(\lambda)=4$,  then either

       with  $\ell(\mu)=4$
        \begin{align}\label{Zn=4l=4type426}
&T=\YT{0.15in}{}{
 { 1,\cdots,1,2,\cdots, 2, 2,\cdots, 2,2,\cdots,{2},2,\cdots,{2}},
 {         2,\cdots,2,3,\cdots,3,{3},\cdots,{3},3,\cdots,3},
 {        3,\cdots,3,6,\cdots,6,{6},\cdots,{6}},
 {        4,\cdots,4,7,\cdots,7},
}
\end{align}
$$ m_{1234}\ge 0,\quad 1\le m_{2367}$$

or with $\ell(\mu)\le 3$ { and  one falls in the case $n=3$ with $\ell(\lambda)=4=3+1$}.

Under the transformation $\Phi^{-1}$  one has  respectively for $\ell(\mu)=4$

\begin{align}\label{n=4l=4type426}
&\Phi^{-1}(T)=\YT{0.15in}{}{
 { 1,\cdots,1,1,\cdots, 1, 1,\cdots, 1,1,\cdots,{1},1,\cdots,{1}},
 {         2,\cdots,2,2,\cdots,2,{2},\cdots,{2},2,\cdots,2},
 {        3,\cdots,3,3,\cdots,3,\overline 2,\cdots,\overline 2},
 {        4,\cdots,4,\overline 1,\cdots,\overline 1},
}
\end{align}
$$ 1\le m_{1234}, \quad m_{123\overline 1}\le m_{12\overline 2}$$
or { one falls in the case $n=3$ with $\ell(\mu)\le 3$ and $\ell(\lambda)=3+1$}.

            \item If $\ell(\lambda)=3$, the $\k$-highest weight tableaux are the ones in the previous case with $\ell(\mu)\le 3$ and $(1234)^M$ with $M=0$, that is, one falls in the case $n=3$.
Similarly under the transformation $\Phi^{-1}$ one falls in the case $n=3$ with $\ell(\mu)\le 3$.

                 \item If $\ell(\lambda)=2$, then either $\ell(\mu)=1$ or $\ell(\mu)=2$,

                \begin{align}\label{n=4typel=2}&T=\YT{0.15in}{}{
 {{1},\cdots,1,    2,\cdots, 2, 2,\cdots, 2},
 {  {2},\cdots,2,   {3},\cdots,3},
}
\end{align}
  with no constraints on the multiplicity of the columns,

  and \begin{align}\label{n=4gl=2}\Phi^{-1}(T)=
&\YT{0.15in}{}{
 {1,\cdots,1,1,\cdots,1,     1,\cdots, 1 },
 {2,\cdots,2,    { \overline 1},\cdots, \overline 1},
}
\end{align}
with no constraints on the multiplicity of the columns.

      \item If $\ell(\lambda)=1$, then $\ell(\mu)=\ell(\lambda)=1$, $Q=\emptyset$ and $S^{H,\mu}=2^M$, $M\ge 0$, ${\LRAII^{AII}}^{-1}(S^{H,\mu},\emptyset)=S^{H,\mu},$ and
          $$\Phi^{-1}(S^{H,\mu})=\Phi^{-1}(2^M)=(1^M), \quad M\ge 0.$$

\end{enumerate}
\end{thm}
\begin{proof} To check the results, we may see that $P^{AII}(T)=S^{H,\mu}$, taking into account the inequalities that $T$ do satisfy, and that in this computation the reduce map  is always of the type \begin{align}
\redu_l:SST_{2n}(\varpi_l)&\rightarrow SpT_{2n}(\varpi_t)
\end{align} with $t\in \{0,1\}$ which guarantees that the recording tableau belongs to $Rec^{1,0}$.

(1)
Taking into account the inequalities in \eqref{ineqCorollary 8azreco}, for each of the tableau in \cite[Corollary 8]{azreco} one has respectively
$P^{AII}(T)=S^{H,\mu}$ and
\begin{align*}&\mu=\wt_\k(T)=(m_2+m_{23}+m_{236}+m_{2367}+m_{2345678}, m_{23}+m_{236}+m_{2367}+m_{1235678}, \\
&m_{236}+m_{2367}+m_{1234678},
m_{236}+m_{234567})\in Par_{\le 4}
\end{align*}
where $m_{1235678}\le m_2$ and $m_{1234678}\le m_{23}$.

\begin{align*}&\mu=\wt_\k(T)=(m_2+m_{23}+m_{236}+m_{2367}+m_{234567},m_{23}+m_{236}+m_{2367}+m_{1235678},\\
&m_{236}+m_{2367}+m_{1234678}, m_{2367}+m_{234567}-m_{1234568}+m_{1234567})\in Par_{\le 4}
\end{align*}

where $m_{234567}-m_{1234568}=0$ and $m_{1235678}\le m_2$, $m_{1234678}\le m_{23}$, $m_{1234567}\le m_{236}$.

\begin{align*}&\mu=\wt_\k(T)=(m_2+m_{23}+m_{236}+m_{2367}+m_{234567},m_{23}+m_{236}+m_{2367}+m_{1235678},\\
&m_{236}+m_{2367}+m_{1234678}, m_{2367}+m_{234567}-m_{1234568}+m_{1234567})\in Par_{\le 4}
\end{align*}

where  $m_{1234568}=m_{234567}+m_{123467}+m_{123567}$ and $m_{1235678}\le m_2$, $m_{1234678}\le m_{23}$, $m_{1234567}\le m_{236}$.


 As for $\Phi^{-1}(T)$,
 the proof goes along the same lines as in the previous theorem. We show  that
 they are $\widehat{\mathfrak{gl}}_{16}$-dominant by verifying that the inequalities \eqref{ineqdom} in Proposition \ref{prop:gineq} (\cite[ Proposition 56]{schumanntorres})  hold, for $1\le i\le 7$, $1\le l\le 8$.
 \end{proof}

\section{Branching rules for the pair (Gl{2n}(C),Sp{2n}(C)) and bijections}\label{sec:last}
Fix $n\in \mathbb{N}$. Let $\lambda \in Par_{\le 2n}$ be a partition with at most $2n$ parts, and $\mu \in Par_{\le n}$ such that $\mu\subset \lambda$.
A partition $\nu\in Par_{\le 2n}$  is said to be \emph{even} if its Young shape has all columns of even length.

A \emph{Littlewood}–\emph{Richardson}-\emph{Sundaram} (LRS) tableau $T$ of skew shape $\lambda/\mu$ in the alphabet $\{1,\dots,2n\}$ and weight the even partition $\nu$ is a skew-tableau of skew shape $\lambda/\mu$ which
has a dominant word of weight $\nu$, and every entry $T(n+i,1)\ge 2i$. The set of all such skew-tableaux is denoted by $LRS_{2n}(\lambda/\mu, \nu)$. Let us write
$$LRS_{2n}(\lambda,\mu):=\bigcup_{\nu \mbox{ even}}  LRS_{2n}(\lambda/\mu,\nu).$$

Let us consider the two following bijections for the Naito-Sagaki conjecture \cite{naitosagaki} where $\lambda\in Par_{\le 2n}$ and $\mu\in Par_{\le n}$ with $\mu\subset\lambda$ are fixed. There is another recent bijection  which does not rely on the two others which is not considered here \cite{muniz}. The first $ \daleth$ is based on the Watanabe quantum branching rule \cite{watanabe}, and  is  a composition of four bijections: $\Phi$ by Naito-Suzuki-Watanabe \cite{nsw}, the inverse quantum LR bijection ${\LRAII^{AII}}^{-1}$ \cite{watanabe,azreco,azslack}, $\pi$ a projection, and the bijection $\lozenge$ \cite[Theorem 1]{azreco}, \cite[Lemma 8.3.2]{watanabe}. That is, $\daleth:=\Phi^{-1}\circ {\LRAII^{AII}}^{-1}\circ \pi^{-1}\circ \lozenge$,

\begin{align} SST^{\widehat{\mathfrak{g}}-dom}_{2n}(\lambda,\mu)\overset{\sim}{\underset{\Phi^{-1}}\longleftarrow } SST^{\mathfrak{k}-hw}_{2n}(\lambda,\mu)  \overset{\sim}{\underset{{\LRAII^{AII}}^{-1}}\longleftarrow} \{S^{H,\mu}\}\times Rec_{2n}(\lambda/\mu)\overset{\sim}{\underset{\pi^{-1}}\longleftarrow } Rec_{2n}(\lambda/\mu)\overset{\sim}{\underset{\lozenge}\longleftarrow } LRS_{2n}(\lambda,\mu):\daleth
\end{align}
and $\daleth LRS_{2n}(\lambda,\mu)=SST^{\widehat{\mathfrak{g}}-dom}_{2n}(\lambda,\mu)$.

The second is by Schumann-Torres \cite[Theorem 35]{schumanntorres}
\begin{align}\phi:\domres(\lambda,\mu)\overset{\sim}\longrightarrow  LRS_{2n}(\lambda,\mu)
\end{align}

We illustrate  $\daleth$ using the Example 36 in \cite{schumanntorres} where $\phi$ is computed,

\begin{ex} Let $n=3$, $\lambda=\omega_1+\omega_2+\omega_3+\omega_4$ and $\mu=\omega_1+\omega_2+\omega_3$. Then

\begin{align*} & LRS_6(\lambda,\mu)=
\{\YT{0.14in}{}{{,,,1},
{,,1},
{,2},
{2},
}, \YT{0.14in}{}{{,,,1},
{,,2},
{,1},
{2},
}, \YT{0.14in}{}{{,,,1},
{,,2},
{,3},
{4},
}\}\overset{\sim}{\underset{\lozenge}\rightarrow }\\
& \overset{\sim}{\underset{\lozenge}\rightarrow }Rec_{6}(\lambda/\mu)=
\{Q_1=\YT{0.14in}{}{{,,,1},
{,,2},
{,1},
{2},
}, Q_2=\YT{0.14in}{}{{,,,1},
{,,1},
{,2},
{2},
}, Q_3= \YT{0.14in}{}{{,,,1},
{,,1},
{,1},
{1},
}\}
\end{align*}

Then
\begin{align*}&\domres(\lambda,\mu)=SST^{\widehat{\mathfrak{g}}-dom}_{6}(\lambda,\mu)=\phi^{-1}LRS_6(\lambda,\mu)\\
&=\{\phi^{-1}(Q_1)=
T_1=\YT{0.15in}{}{{1,1,1,1},
{2,2,2},
{3,\overline 1},
{\overline 2},
},\phi^{-1}(Q_2)= T_2=\YT{0.15in}{}{{1,1,1,1},
{2,2,\overline 1},
{3,3},
{\overline 3},
}, \phi^{-1}(Q_3)=T_3=\YT{0.15in}{}{{1,1,1,1},
{2,2,2},
{3,\overline 2},
{\overline 1},
}\}
\end{align*}

Let $S^{H,\mu}=\YT{0.15in}{}{{2,2,2},
{3,3},
{6},
}\in SpT_6(\mu)$ be the symplectic $\k$-highest weight in $SpT_6(\mu)$, and let
\begin{align}\daleth: SST^{\widehat{\mathfrak{g}}-dom}_{6}(\lambda,\mu)\overset{\sim}{\underset{\Phi^{-1}}\longleftarrow } SST^{\mathfrak{k}-lw}_{6}(\lambda,\mu)  \overset{\sim}{\underset{{\LRAII^{AII}}^{-1}}\longleftarrow} \{S^{H,\mu}\}\times Rec_{6}(\lambda/\mu)\overset{\sim}{\underset{\pi^{-1}}\longleftarrow } Rec_{6}(\lambda/\mu)\overset{\sim}{\underset{\lozenge}\longleftarrow } LRS_{6}(\lambda,\mu)\nonumber
\end{align}

We compute ${\LRAII^{AII}}^{-1}(\{S^{H,\mu}\}\times Rec_{6}(\lambda/\mu))=SST^{\mathfrak{k}-hw}_{6}(\lambda,\mu)$. For ${\LRAII^{AII}}^{-1}(S^{H,\mu}, Q_2)$, one has

\begin{align*}
&\c \circ (\rm{red}_{2}^{-1},\rm{id})\circ (\underset{{(1,2)}}\leftarrow S^{H,\mu}=\YT{0.15in}{}{{2,2,2},
{3,3},
{6},
})
=\c\circ (\rm{red}_{2}^{-1},\rm{id})((2,3), \YT{0.15in}{}{
 {2,2},
 {3},
{6},
})=\YT{0.15in}{}{
 {2,2,2},
 {3,3},
{5,6},
{6},
}
\end{align*}
\begin{align*}
&\c \circ (\rm{red}_{2}^{-1},\rm{id})\circ (\underset{{(3,4)}}\leftarrow \YT{0.15in}{}{
 {2,2,2},
 {3,3},
{5,6},
{6},
})
=\c\circ (\rm{red}_{2}^{-1},\rm{id})((5,6), \YT{0.15in}{}{
 {2,2,2},
 {3,3},
{6},
})=\YT{0.15in}{}{
 {1,2,2,2},
 {2,3,3},
{5,6},
{6},
}=S_2\in SST^{\mathfrak{k}-hw}_{6}(\lambda,\mu)
\end{align*}
Then
$${\LRAII^{AII}}^{-1}(S^{H,\mu}, Q_2)=S_2\in SST^{\mathfrak{k}-hw}_{6}(\lambda,\mu)$$

Now we compute,
$$\Phi^{-1}(S)=\pr^{-1}_{3,4}\circ \pr_{3,5}^{-1}\circ \pr^{-1}_{1,6}(S_2)=\YT{0.15in}{}{
 {1,1,1,1},
 {2,2,6},
{3,3},
{4},
}\in  SST^{\widehat{\mathfrak{g}}-dom}_{6}(\lambda,\mu)$$
\end{ex}
Recalling that $\i=6-i+1$ for $1\le i\le 3$, then
$$\YT{0.15in}{}{
 {1,1,1,1},
 {2,2,6},
{3,3},
{4},
}
=\YT{0.15in}{}{
 {1,1,1,1},
 {2,2,\overline 1},
{3,3},
{\overline 3},
}=T_2=\phi(Q_2).
$$

Then $\Phi^{-1}(S_2)=T_2\in SST^{\widehat{\mathfrak{g}}-dom}_{6}(\lambda,\mu)$ and $\daleth(Q_2)=\phi^{-1}(Q_2)=T_2$.

Similarly,

\begin{align*}&{\LRAII^{AII}}^{-1}(S^{H,\mu}, Q_3)=S_3=\YT{0.15in}{}{
 {1,2,2,2},
 {2,3,3},
{3,6},
{4},
}\in SST^{\mathfrak{k}-hw}_{6}(\lambda,\mu),\\
& \Phi^{-1}(S_3)=\YT{0.15in}{}{
 {1,1,1,1},
 {2,2,2},
{3,5},
{6},
}=\YT{0.15in}{}{
 {1,1,1,1},
 {2,2,2},
{3,\overline 2},
{\overline 1},
}=T_3\in  SST^{\widehat{\mathfrak{g}}-dom}_{6}(\lambda,\mu)\end{align*}
and $\daleth(Q_3)=\phi^{-1}(Q_3)=T_3$.

Lastly, for ${\LRAII^{AII}}^{-1}(S^{H,\mu}, Q_1)$, one has

\begin{align*}
&\c \circ (\rm{red}_{2}^{-1},\rm{id})\circ (\underset{{(1,3)}}\leftarrow S^{H,\mu}=\YT{0.15in}{}{{2,2,2},
{3,3},
{6},
})
=\c\circ (\rm{red}_{2}^{-1},\rm{id})((2,6), \YT{0.15in}{}{
 {2,2},
 {3,3},
})=\YT{0.15in}{}{
 {2,2,2},
 {3,3,3},
{4},
{6},
}
\end{align*}
\begin{align*}
&\c \circ (\rm{red}_{2}^{-1},\rm{id})\circ (\underset{{(2,4)}}\leftarrow \YT{0.15in}{}{
 {2,2,2},
 {3,3,3},
{4},
{6},
})
=\c\circ (\rm{red}_{2}^{-1},\rm{id})((3,6), \YT{0.15in}{}{
 {2,2,2},
 {3,3},
{4},
})=\YT{0.15in}{}{
 {1,2,2,2},
 {2,3,3},
{3,4},
{6},
}=S_1\in SST^{\mathfrak{k}-hw}_{6}(\lambda,\mu)
\end{align*}
Then

\begin{align*}&{\LRAII^{AII}}^{-1}(S^{H,\mu}, Q_1)=S_1=\YT{0.15in}{}{
 {1,2,2,2},
 {2,3,3},
{3,4},
{6},
}\in SST^{\mathfrak{k}-hw}_{6}(\lambda,\mu),\\
& \Phi^{-1}(S_1)=\YT{0.15in}{}{
 {1,1,1,1},
 {2,2,2},
{3,6},
{5},
}=\YT{0.15in}{}{
 {1,1,1,1},
 {2,2,2},
{3,\overline 1},
{\overline 2},
}=T_1\in  SST^{\widehat{\mathfrak{g}}-dom}_{6}(\lambda,\mu)\end{align*}
and $\daleth(Q_1)=\phi^{-1}(Q_1)=T_1$.

\bibliography{sample17}
\bibliographystyle{alpha}
\end{document}